\documentclass[reqno]{amsart}
\usepackage{amssymb,amsmath,amsthm,amsfonts,color}

\usepackage{mathrsfs,dsfont, comment,mathscinet}
\usepackage{hyperref}
\usepackage{enumerate,esint}
\usepackage{mathtools}

\DeclarePairedDelimiter{\floor}{\lfloor}{\rfloor}
\allowdisplaybreaks
\numberwithin{equation}{section}
\theoremstyle{plain}
\newtheorem{theorem}{Theorem}[section]
\newtheorem{proposition}[theorem]{Proposition}
\newtheorem{lemma}[theorem]{Lemma}
\newtheorem{example}[theorem]{Example}

\theoremstyle{definition}
\newtheorem{definition}[theorem]{Definition}
\newtheorem{remark}[theorem]{Remark}

\newcommand\R{\mathbb R}
\newcommand\M{\mathbb M}
\newcommand\N{\mathbb N}

\newcommand\Z{\mathbb{Z}}

\newcommand\iO{\int_\Omega}
\newcommand{\iOQ}{\int_{\Omega}\int_Q}
\newcommand{\iQQ}{\int_Q\int_Q}
\newcommand\iQ{\int_Q}
\newcommand\e{\varepsilon}

\newcommand\PP{\mathbb{P}}
\newcommand\QQ{\mathbb{Q}}
\newcommand\rd{\R^d}

\newcommand\rl{\R^l}

\newcommand\en{\e_n}

\newcommand\liminfn{\liminf_{n\to +\infty}}
\newcommand\ep{\varepsilon}
\newcommand{\scal}[2]{\langle #1,\,#2 \rangle}

\newcommand\wk{\rightharpoonup}
\newcommand\wkts{\overset{2-s}{\rightharpoonup}}
\newcommand\sts{\overset{2-s}{\to}}

\newcommand{\C}{\mathcal{C}^{\pdeor}}
\newcommand{\ck}{\mathcal{C}^{\pdeor_k}}

\newcommand{\cc}{\mathcal{C}}
\newcommand{\F}{\mathscr{F}_{\pdeor}}

\newcommand\Rn{\mathbb{R}^N}
\newcommand{\FF}{\mathcal{F}_{\pdeor}}
\newcommand{\FB}{\overline{\mathcal{F}}_{\pdeor}}
\newcommand{\FFN}{\mathcal{F}_{\pdeor(n\cdot)}}
\newcommand{\FBN}{\overline{\mathcal{F}}_{\pdeor(n\cdot)}}
\newcommand\rn[1]{R_{#1}^k}
\newcommand{\rnn}{R^k}
\newcommand\sn{S^k}
\newcommand\pdeep[1]{{\mathscr{A}}^{\rm div}_{#1}}
\newcommand\pdeepk[1]{{\mathscr{A}}_{k, #1}^{\rm div}}

\newcommand\pdeor{\mathscr{A}}
\newcommand{\px}{\mathbb{T}_x}
\newcommand{\py}{\mathbb{T}_y}
\newcommand{\A}{\mathbb{A}}
\newcommand\aaa{{\mathcal{A}}}
\newcommand{\cx}[0]{\mathcal{C}_{x}}
\newcommand{\qa}[0]{Q_{\pdeor}}
\newcommand{\iq}[0]{\int_Q}

\newcommand{\ze}{{Z}_{\e}}
\newcommand{\qe}{Q_{\e,z}}

\newcommand\be[1]{\begin{equation}\label{#1}}
\newcommand\ee{\end{equation}}
\newcommand\ba[1]{\begin{align}\label{#1}}
\newcommand\ea{\end{align}}
\newcommand\bas{\begin{align*}}
\newcommand\eas{\end{align*}}
\newcommand\nn{\nonumber}

\title [Homogenization for $\pdeor(x)-$quasiconvexity]
{Homogenization of integral energies under periodically oscillating differential constraints}
\author[E. Davoli] {Elisa Davoli} 
\address[Elisa Davoli]{Department of Mathematics\\ Carnegie Mellon University\\Forbes Avenue\\Pittsburgh PA 15213}
\email[E. Davoli]{edavoli@andrew.cmu.edu}
\author[I. Fonseca] {Irene Fonseca} 
\address[Irene Fonseca]{Department of Mathematics\\ Carnegie Mellon University\\Forbes Avenue\\Pittsburgh PA 15213}
\email[I. Fonseca]{fonseca@andrew.cmu.edu}

\begin{document} 
\vskip .2truecm
\begin{abstract}
\small{A homogenization result for a family of integral energies
$$u_{\ep}\mapsto\iO f(u_{\ep}(x))\,dx,\quad \ep\to 0^+,$$
is presented, where the fields $u_{\ep}$ are subjected to periodic first order oscillating differential constraints in divergence form. The work is based on the theory of $\pdeor$-quasiconvexity with variable coefficients and on two-scale convergence techniques.}
\end{abstract}
\maketitle
\section{Introduction}
This paper is the first step toward a complete understanding of homogenization problems for oscillating energies subjected to oscillating linear differential constraints, in the framework of $\pdeor-$quasiconvexity with variable coefficients. To be precise, we initiate the study of integral representations for limits of oscillating integral energies
$$u_{\ep}\mapsto \iO f\Big(x,\frac{x}{\ep^{\alpha}},u_{\ep}(x)\Big)\,dx,$$
where $\Omega\subset \R^N$ is an open bounded domain, $\ep\to 0^+$, and the fields $u_{\ep}\in L^p(\Omega;\rd)$ are subjected to periodically oscillating differential constraints such as
\be{eq:constrain-no-divergence}
{\mathscr{A}}_{\ep}u_{\ep}:=\sum_{i=1}^N  A^i\Big(\frac{x}{\ep^{\beta}}\Big)\frac{\partial u_{\ep}(x)}{\partial {x_i}}\to 0\quad\text{strongly in }W^{-1,p}(\Omega;\R^l),
\ee
or in divergence form
\be{eq:constrain-divergence}
\pdeep{\ep}u_{\ep}:=\sum_{i=1}^N \frac{\partial}{\partial {x_i}} \Big(A^i\Big(\frac{x}{\ep^{\beta}}\Big)u_{\ep}(x)\Big)\to 0\quad\text{strongly in }W^{-1,p}(\Omega;\R^l),
\ee
with $1<p<+\infty$, $A^i(x)\in {\rm Lin}\,(\R^d;\R^l)\equiv \M^{l\times d}$ for every $x\in\R^N$, $i=1,\cdots,N$, $d,l\geq 1$, and where $\alpha,\beta$ are two nonnegative parameters. Here, and in what follows, $\M^{l\times d}$ stands for the linear space of matrices with $l$ rows and $d$ columns. 

Oscillating divergence-type constraints as in \eqref{eq:constrain-divergence} appear in the homogenization theory of systems of second order elliptic partial differential equations. Indeed, if $u_{\ep}=\nabla v_{\ep}$, with $v_{\ep}\in W^{1,p}(\Omega)$ for every $\ep>0$, and $A^i(x)=A(x)\in \M^{N\times N}$ for $i=1,\cdots,N$, then considering \eqref{eq:constrain-divergence} reduces to the homogenization problem of finding the effective behavior of (weak) limits of $v_{\ep}$, where
$$\text{div }\Big(A\Big(\frac{x}{\ep}\Big)\nabla v_{\ep}\Big)\to 0\quad\text{strongly in }W^{-1,p}(\Omega),\,1<p<+\infty.$$
These problems have been extensively studied in the literature (see {e.g.} \cite{allaire}, \cite[Chapter 1, Section 6]{bensoussan.lions.papanicolau}, \cite{cioranescu.donato}, and the references therein).\\

Different regimes are expected to arise depending on the relation between $\alpha$ and $\beta$. 
Here we will consider $\beta>0$, and we will assume that the energy density $f$ is constant in the first two variables but the differential constraint in divergence form \eqref{eq:constrain-divergence} oscillates periodically. The limit scenario $\alpha>0, \beta=0$ and \eqref{eq:constrain-no-divergence} (treated in \cite{fonseca.kromer} for constant coefficients), {i.e.}, the energy density is oscillating but the differential constraint is fixed, and the situation in which $\alpha>0$ and $\beta>0$, will be the subject of forthcoming papers.\\

The key tool for our analysis is the notion of $\pdeor-$quasiconvexity. For $i=1\cdots,N$, consider matrix-valued maps $A^i\in C^{\infty}(\Omega;\M^{l\times d})$, and define $\pdeor$ as the differential operator such that
$$\pdeor v (x):=\sum_{i=1}^N A^i(x)\frac{\partial v(x)}{\partial {x_i}},\quad x\in\Omega,$$
for $v\in L^1_{\rm loc}(\Omega; \R^d)$, where $\frac{\partial v}{\partial {x_i}}$ is to be interpreted in the sense of distributions.
We require that the operator $\pdeor$ satisfies a uniform constant-rank assumption (see \cite{murat}) {i.e.}, there exists $r\in \N$ such that 
 \begin{equation}
 \label{cr}
 \text{rank }\left(\sum_{i=1}^N A^i(x)w_i\right)=r\quad\text{for every }w\in\mathbb{S}^{n-1},
 \end{equation} uniformly with respect to $x$, where $\mathbb{S}^{N-1}$ is the unit sphere in $\mathbb{R}^N$. The properties of $\pdeor$-quasiconvexity in the case of constant coefficients were first investigated by Dacorogna in \cite{dacorogna}, and then studied by Fonseca and M\"uller in \cite{fonseca.muller} (see also \cite{fonseca.dacorogna}). In \cite{santos}  Santos extended the analysis of \cite{fonseca.muller} to the case in which the coefficients of the differential operator $\pdeor$ depend on the space variable.

\begin{definition}
Let $f:\R^d \to \R$ be a continuous function, let $Q$ be the unit cube in $\R^N$  centered at the origin, 
$$Q=\Big(-\frac{1}{2},\frac{1}{2}\Big)^N,$$
and denote by  $C^{\infty}_{\rm per}(\R^N;\R^d)$ the set of smooth maps which are $Q$-periodic in $\R^N$. Consider the set
$$\cx:=\Big\{w\in C^{\infty}_{\rm per}(\R^N;\R^d):\,\int_Q{w(y)\,dy}=0,\, \sum_{i=1}^N A^i(x)\frac{\partial w(y)}{\partial {y_i}}=0 \Big\}.$$
For a.e. $x\in\Omega$, the \emph{$\pdeor-$quasiconvex envelope} of $f$ in $x\in\Omega$ is defined as
$$\xi\mapsto\qa f(x,\xi):=\inf\Big\{\iq f(\xi+w(y))\,dy:\,w\in\cx\Big\}.$$
$f$ is said to be \emph{$\pdeor$-quasiconvex} if $f(\xi)=\qa f(x,\xi)$ for a.e. $x\in\Omega$ and all $\xi\in\R^d$.
\end{definition}
We remark that when $\pdeor := \rm curl$, {i.e.}, when $v=\nabla\phi$ for some $\phi\in	W^{ 1,1}_{\rm loc} (\Omega; \R^m )$, then $d=m\times N$, then $\pdeor$-quasiconvexity reduces to Morrey's notion of quasiconvexity (see \cite{acerbi.fusco, ball, marcellini, morrey}).

The following theorem was proved in \cite{santos} in the more general case when $f$ is a Carath\'eodory function, generalizing the corresponding results \cite[Theorems 3.6 and 3.7]{fonseca.muller} in the case of constant coefficients ({i.e.} $A^i(x)\equiv A^i\in\M^{l\times d}$ for every $i=1,\cdots,N$).
\begin{theorem}
\label{thm:santos}
\nonumber
Let $\Omega$ be an open bounded domain in $\R^N$, let $A^i\in C^{\infty}(\Omega;\M^{l\times d})\cap W^{1,\infty}(\Omega;\M^{l\times d})$, $i=1,\cdots, N$, $d\geq 1$, $1<p<+\infty$, and assume that the operator $\pdeor$ satisfies the constant rank condition \eqref{cr}. Let $f:\rd\to [0,+\infty)$ be a continuous function satisfying
\begin{enumerate}
\item [(i)]$\quad 0\leq f(v)\leq C(1+|v|^p),$
\item[(ii)] $|f(v_1)-f(v_2)|\leq C(1+|v_1|^{p-1}+|v_2|^{p-1})|v_1-v_2|$
\end{enumerate}
for all $v,v_1,v_2\in\rd$, and for some $C>0$. Then $\pdeor-$quasiconvexity is a necessary and sufficient condition for lower semicontinuity of the functional
$$v\mapsto\iO f(v(x))\,dx$$
for sequences $v_{\ep}\wk v$ weakly in $L^p(\Omega;\rd)$ and such that $\pdeor v_{\ep}\to 0$ strongly in $W^{-1,p}(\Omega;\rl)$.
\end{theorem}

In the case of constant coefficients, in \cite{braides.fonseca.leoni} Braides, Fonseca and Leoni provided an integral representation formula for relaxation problems in the context of $\pdeor$-quasiconvexity and 
 presented (via $\Gamma$-convergence) homogenization results for periodic integrands evaluated along $\pdeor-$free fields. 
 Their homogenization results were later generalized in \cite{fonseca.kromer}, where Fonseca and Kr\"omer worked still in the framework of constant coefficients but under weaker assumptions on the energy density $f$.\\ 
 
This paper is devoted to extending the previous homogenization results to the case in which $\pdeor$ is a differential operator with nonconstant $L^{\infty}$-coefficients, the energies under consideration are of the type
$$u_{\ep}\mapsto\int_{\Omega} f(u_{\ep}(x))\,dx,$$
where $u_{\ep}\wk u$ weakly in $L^p(\Omega;\R^d)$, and $$\pdeep{\e}u_{\ep}:=\sum_{i=1}^N \frac{\partial}{\partial {x_i}}\Big(A^i\Big(\frac{x}{\ep}\Big)u_{\ep}(x)\Big)\to 0\quad\text{strongly in }W^{-1,q}(\Omega;\R^l)$$
for all $1\leq q<p$. {We point out that the result in Theorem \ref{thm:santos} \cite{santos} covers the case $q=p$. Our analysis includes the case when $q=p$ if the operator $\pdeor$ has smooth coefficients. However, in the general situation when $\pdeor$ has bounded coefficients, the assumption $1\leq q<p$ is required, in order to satisfy some truncation and $p$-equiintegrability arguments (see the proofs of Theorems \ref{thm:main-homogenized-pde} and \ref{thm:main-homogenized-pde-bdd}). }

Our starting point is a characterization of the set $\C$ of limits of $\pdeep{\e}-$vanishing fields $u_{\ep}$. We show in Proposition \ref{lemma:limit-maps} that a function $u\in L^p(\Omega;\rd)$ belongs to $\C$ if and only if there exists a map $w\in L^p(\Omega;L^p_{\rm per}(\R^N;\rd))$ such that $\iQ w(x,y)\,dy=0$ for a.e. $x\in\Omega$,
$$u_{\ep}\sts u+w$$
strongly two-scale in $L^p(\Omega\times Q;\rd)$ (see Definition \ref{def:2-scale}), and $u+w$ satisfies the differential constraints
\be{eq:const1}\sum_{i=1}^N \frac{\partial}{\partial {x_i}}\Big(\iQ A^i(y)(u(x)+w(x,y))\,dy\Big)=0\ee
in $W^{-1,p}(\Omega;\rl)$, and
\be{eq:const2}\sum_{i=1}^N \frac{\partial}{\partial {y_i}}(A^i(y)(u(x)+w(x,y)))=0\ee
in $W^{-1,p}(Q;\rl)$ for a.e. $x\in\Omega$. This generalizes the classical characterization of 2-scale limits of solutions to linear elliptic partial differential equations in divergence form in \cite[Theorem 2.3]{allaire} to the case of first order linear systems.

For every $u\in\C$, we denote by $\C_u$ the class of maps $w$ as above. We then prove that the homogenized energy is given by the functional
$$\F(u):=\begin{cases}\inf_{r>0}\inf\Big\{\liminf_{n\to +\infty}\FBN^{\,r}(u_n):\,u_n\wk u\quad\text{weakly in }L^p(\Omega;\R^d)\Big\}\\\text{if }u\in\C,\\
+\infty\quad\text{otherwise in }L^p(\Omega;\R^d),\end{cases}$$
where
$$\FBN^{\,r}(v):=\begin{cases}\inf\Big\{\iO f(v(x)+w_n(x,y))\,dy\,dx:w\in {\mathcal C}^{\pdeor(n\cdot)}_{v},\,\|w\|_{L^p(\Omega\times Q;\R^d)}\leq r\Big\}\\
\text{if }v\in\,{\mathcal C}^{\pdeor(n\cdot)}_r,\\
+\infty\quad\text{otherwise in }L^p(\Omega;\R^d),\end{cases}$$
the classes ${\mathcal{C}_v^{\pdeor(n\cdot)}}$ are defined analogously to $\C_v$ by replacing the operators $A^i(\cdot)$ with $A^i(n\cdot)$ in \eqref{eq:const1} and \eqref{eq:const2}, and $$\mathcal{C}^{\pdeor(n\cdot)}_r:=\{v\in L^p(\Omega;\R^d):\,\exists w\in {\mathcal{C}_v^{\pdeor(n\cdot)}}\quad\text{with }\|w\|_{L^p(\Omega\times Q;\R^d)}\leq r\},\quad r>0.$$
Our main result is the following.
\begin{theorem}
\label{thm:second-main}
Let $1<p<+\infty$. Let $A^i\in L^{\infty}(Q;\M^{l\times d})$, $i=1,\cdots,N$, and let $f:\R^d\to [0,+\infty)$ be a continuous map satisfying the growth condition
$$0\leq f(v)\leq C(1+|v|^p)\quad\text{for every }v\in\rd,\quad\text{and some }C>0.$$
Then, for every $u\in\C$ there holds
 \begin{align*}
 &\inf\Big\{\liminf_{\ep\to 0}\iO f(u_{\ep}(x))\,dx:\,u_{\ep}\wk u\quad\text{weakly in }L^p(\Omega;\rd)\\
 &\qquad\text{and }\pdeep{\e}u_{\e}\to 0\quad\text{strongly in }W^{-1,q}(\Omega;\rl)\text{ for every }1\leq q<p\Big\}\\
 &\quad=\inf\Big\{\limsup_{\ep\to 0}\iO f(u_{\ep}(x))\,dx:\,u_{\ep}\wk u\quad\text{weakly in }L^p(\Omega;\rd)\\
 &\qquad\text{and }\pdeep{\e}u_{\e}\to 0\quad\text{strongly in }W^{-1,q}(\Omega;\rl)\text{ for every }1\leq q<p\Big\}=\F(u).
 \end{align*}
 \end{theorem}
 \begin{remark}\begin{enumerate} \item[(i)] As a consequence of Theorem \ref{thm:santos}, we expected the homogenized energy to be related to the effective energy for an ``$\pdeor$-quasiconvex" envelope of the function $f$, with the role of the differential constraint $\pdeor$ to be replaced by the limit constraints \eqref{eq:const1} and \eqref{eq:const2}. We stress the fact that here the oscillatory behavior of the differential constraint as $\ep\to 0$ forces the relaxation with respect to \eqref{eq:const1} and \eqref{eq:const2} and the homogenization in the differential constraint to happen somewhat simultaneously. Indeed, for every $n$ the functional $\FBN^{\,r}$ is obtained as a truncated version of a relaxation with respect to the limit differential constraints dilated by a factor $n$, and is evaluated on a fixed element of a sequence of maps approaching $u$, whereas the limit functional $\F(u)$ is deduced by a ``diagonal" procedure, as $n$ tends to $+\infty$.
 \item[(ii)] {The truncation in the definition of the functionals $\FBN^{\,r}$ plays a key role in the proof of the limsup inequality
 \bas&\inf\Big\{\limsup_{\ep\to 0}\iO f(u_{\ep}(x))\,dx:\,u_{\ep}\wk u\quad\text{weakly in }L^p(\Omega;\rd)\\
 &\qquad\text{and }\pdeep{\e}u_{\e}\to 0\quad\text{strongly in }W^{-1,q}(\Omega;\rl)\text{ for every }1\leq q<p\Big\}\leq\F(u),
 \end{align*}
 because it provides boundedness of the ``recovery sequences" and thus allows us to apply a diagonalization argument (see Step 3 in the proof of Proposition \ref{thm:limsup}).}
 \item[(iii)]
 We remark that, as opposed to the case in which the operators $A^i$ are constant, we cannot expect the homogenized energy to be local, {i.e.}, that there exists $f_{hom}:\Omega\times \R^d\to[0,+\infty)$ such that
\be{eq:locality}\FF(u)=\int_{\Omega}f_{hom}(x,u(x))\,dx.\ee We show in Example \ref{example:nonlocal}  that  locality in the sense of \eqref{eq:locality} may fail even when the function $f$ is convex in its second variable. 
\end{enumerate}
\end{remark}
As in \cite{fonseca.kromer}, the proof of this result is based on the so-called \emph{unfolding operator}, introduced in \cite{cioranescu.damlamian.griso, cioranescu.damlamian.griso08} (see also \cite{visintin, visintin1} and Subsection \ref{subsection:unfolding}). A first difference with \cite[Theorem 1.1]{fonseca.kromer} ({i.e.}, with the case in which the operators $A^i$ are constant) is the fact that we are unable to work with exact solutions of the system $\pdeep{\ep} u_{\ep}=0$, but instead we consider sequences of asymptotically $\pdeep{\ep}-$vanishing fields. As pointed out in \cite{santos}, in the case of variable coefficients the natural framework is pseudo-differential operators. In this setting, we do not project directly onto the kernel of a differential constraint $\pdeor$, but rather we construct an ``approximate" projection operator $P$ such that for every field $v\in L^p$, the $W^{-1,p}$ norm of $\pdeor Pv$ is controlled by the $W^{-1,p}$ norm of $v$ itself (we refer to \cite[Subsection 2.1]{
santos} for a detailed explanation of this issue, and to the references therein for a treatment of the main properties of pseudo-differential operators). 

The crucial difference with respect to the case of constant coefficients is the structure of the set $\C$. In the case in which the condition $\pdeep{\ep}u_{\ep}\to 0$ is replaced by $\pdeor u_{\ep}=0$, with $\pdeor$ being independent of the space variable, \eqref{eq:const1} and \eqref{eq:const2} decouple (see \cite[Theorem 1.2]{fonseca.kromer}) becoming separate requirements on $w$ and $u$. However, in our situation they can not be dealt with separately, and this forces the structure of the homogenized energy to be much more involved.

The oscillatory behavior of the differential constraint and its $\ep$-dependent structure render this problem quite technical due to the difficulty in obtaining a suitable projection operator on the limit differential constraint. Moreover, due to the coupling between \eqref{eq:const1} and \eqref{eq:const2} and the dependence of the operators on $\ep$, the pseudo-differential operators method cannot be applied directly here. In order to solve this problem, in Lemma \ref{lemma:projection} we are led to impose a uniform invertibility requirement on the differential operator. To be precise, we require $l\times N=d$ and we assume that there exists a positive constant $\gamma$ such that the operator $\mathcal{A}(y)\in{\rm Lin}\,(\R^d;\R^d)$, defined as
$$\mathcal{A}(y)\xi:=\left(\begin{array}{c}(A^1(y)\xi)^T\\\vdots\\(A^N(y)\xi)^T\end{array}\right)\in \M^{N\times l}\cong\R^{d}\quad\text{for every }\xi\in\R^d,$$
satisfies
\begin{align}\nonumber(H)\quad \mathcal{A}(y)\lambda\cdot\lambda\geq \gamma|\lambda|^2\quad\text{for every }\lambda\in\rd\text{ and }y\in\R^N.
\end{align}
We remark that assumption $(H)$ is quite natural, as it represents a higher-dimensional version of the classical uniform ellipticity assumption (see {e.g.} \cite[(2.2)]{allaire}).

The strategy of our argument consists in first proving Theorem \ref{thm:second-main} in the case in which the operators $A^i$ are smooth. The general case is then deduced by means of an approximation argument of bounded operators by smooth ones, and by an application of Severini-Egoroff's theorem and
$p-$equiintegrability (see Section \ref{section:bdd}). 

Our main theorem is consistent with the relaxation results obtained in \cite{braides.fonseca.leoni} in the case of constant coefficients. When the linear operators $A^i$ are constant, we prove in Subsection \ref{subsection:general} that the homogenized energy $\F$ and Theorem \ref{thm:second-main} reduce to the $\pdeor-$quasiconvex envelope of $f$ and \cite[Theorem 1.1]{braides.fonseca.leoni}, respectively.\\ 

This article is organized as follows. In Section \ref{section:prel} we introduce notation and recall some preliminary results on two-scale convergence and on the unfolding operator. In Section \ref{section:pde} we provide a characterization of the limits of $\pdeep{\ep}$-vanishing fields (see Proposition \ref{lemma:limit-maps}). Section \ref{section:smooth} is devoted to the proof of our main result, Theorem \ref{thm:second-main}, for smooth operators $\pdeep{\ep}$. The argument is extended to the case in which $\pdeep{\ep}$ are only bounded in Section \ref{section:bdd}.

\section{Preliminary results}
\label{section:prel}
Throughout this paper $\Omega\subset \R^N$ is an open bounded domain and $\mathcal{O}(\Omega)$ is the set of open subsets of $\Omega$.  $Q$ is the unit cube in $\R^N$ centered at the origin and with normals to its faces parallel to the vectors in the standard orthonormal basis of $\R^N$, 
$\{e_1,\cdots,e_N\}$, {i.e.},
$$Q:=\Big(-\frac{1}{2},\frac{1}{2}\Big)^N.$$
Given $1<p<+\infty$, we denote by $p'$ its conjugate exponent, that is 
$$\frac{1}{p}+\frac{1}{p'}=1.$$
Whenever a map $v\in L^p, C^{\infty},\cdots$, is $Q-$periodic, that is
$$v(x+e_i)=v(x)\quad i=1,\cdots, N$$
for a.e. $x\in \R^N$, we write $v\in L^p_{\rm per}, C^{\infty}_{\rm per},\cdots$, respectively. We will implicitly identify the spaces $L^p(Q)$ and $L^p_{\rm per}(\R^N)$. We designate the Lebesgue measure of a measurable set $A\subset\R^N$ by $|A|$.
We adopt the convention that $C$ will stand for a generic positive constant, whose value may change from expression to expression in the same formula.

\subsection{Two-scale convergence}
We recall the definition  and some properties of two-scale convergence which apply to our framework. For a detailed treatment of the topic we refer to, {e.g.}, \cite{allaire, lukkassen.nguetseng.wall, nguetseng}. Throughout this subsection $1<p<+\infty$.
\begin{definition}
\label{def:2-scale}
If $v\in L^p(\Omega;L^p_{\rm per}(\Rn;\rd))$ and $\{u_{\ep}\}\in L^p(\Omega;\rd)$, we say that $\{u_{\ep}\}$ \emph{weakly two-scale converge to $v$} in $L^p(\Omega\times Q;\rd)$, $u_{\ep}\wkts v$,  if
$$\iO u_{\ep}(x)\cdot \varphi\Big(x,\frac{x}{\ep}\Big)\,dx\to \iOQ v(x,y)\cdot \varphi(x,y)\,dy\,dx$$
for every $\varphi\in L^{p'}(\Omega;C_{\rm per}(\Rn;\rd))$. 

We say that $\{u_{\ep}\}$ \emph{strongly two-scale converge to $v$} in $L^p(\Omega\times Q;\rd)$, $u_{\ep}\sts v$, if $u_{\ep}\wkts v$ and $$\lim_{\ep\to 0}\|u_{\ep}\|_{L^p(\Omega;\rd)}=\|v\|_{L^p(\Omega\times Q;\rd)}.$$
\end{definition}
Bounded sequences in $L^p(\Omega;\rd)$ are pre-compact with respect to weak two-scale convergence. To be precise (see \cite[Theorem 1.2]{allaire}),
\begin{proposition}
\label{prop:2-scale-compactness}
Let $\{u_{\ep}\}\subset L^p(\Omega;\rd)$ be bounded. Then there exists $v\in L^p(\Omega;L^p_{\rm per}(\Rn;\rd))$ such that, up to a subsequence, $u_{\ep}\wkts v$ weakly two-scale, and, in particular,
$$u_{\ep}\wk u:=\iQ v(x,y)\,dy\quad\text{weakly in }L^p(\Omega;\rd).$$
\end{proposition}
The following result will play a key role throughout the paper in the proofs of limsup inequalities (see \cite[Lemma 1.3]{allaire}, \cite[Lemma 2.1]{visintin1}, and \cite[Proposition 2.4, Lemma 2.5 and Remark 2.6]{fonseca.kromer}).
\begin{proposition}
\label{prop:simple-2-scale}
Let $v\in L^p(\Omega;C_{\rm per}(\Rn;\rd))$ or $v\in L^p_{\rm per}(\Rn;C(\overline{\Omega};\rd))$. Then the sequence $\{u_{\ep}\}$, defined as
$$u_{\ep}(x):=v\Big(x,\frac{x}{\ep}\Big)$$
is $p-$equiintegrable, and
$$u_{\ep}\sts v\quad\text{strongly two-scale in }L^p(\Omega;\rd).$$ 
\end{proposition}
\subsection{The unfolding operator}
\label{subsection:unfolding}
We collect here the definition and some properties of the \emph{unfolding operator} (see {e.g.} \cite{cioranescu.damlamian.griso, cioranescu.damlamian.griso08, visintin, visintin1}).
\begin{definition}
Let $u\in L^p(\Omega;\rd)$. For every $\ep>0$ the unfolding operator
 $T_{\e}:L^p(\Omega;\rd)\to L^p(\R^N;L^p_{\rm per}(\Rn;\rd))$ is defined componentwise as
\be{eq:unfolding-operator}
T_{\e}(u)(x,y):=u\Big(\e\floor[\Big]{\frac{x}{\e}}+\e(y-\floor{y})\Big)\quad\text{for a.e. }x\in\Omega\text{ and }y\in\Rn,\ee
where $u$ is extended by zero outside $\Omega$ and $\floor{\cdot}$ denotes the least integer part. 
\end{definition}
\begin{proposition}
\label{prop:isometry}
$T_{\ep}$ is a nonsurjective linear isometry from $L^p(\Omega;\rd)$ to $L^p(\R^N\times Q;\rd)$.
\end{proposition}
The next theorem relates the notion of two-scale convergence to $L^p$ convergence of the unfolding operator (see \cite[Proposition 2.5 and Proposition 2.7]{visintin1},\cite[Theorem 10]{lukkassen.nguetseng.wall}).
\begin{theorem}
\label{thm:equivalent-two-scale}
Let $\Omega$ be an open bounded domain and let $v\in L^p(\Omega;L^p_{\rm per}(\R^N;\R^d))$. Assume that $v$ is extended to be $0$ outside $\Omega$. Then the following conditions are equivalent:
\begin{enumerate} 
\item[(i)]$u^{\ep}\wkts v\quad\text{weakly two scale in }L^p(\Omega\times Q;\rd),$
\item[(ii)] $T_{\ep}u_{\ep}\wk v$ weakly in $L^p(\R^N\times Q;\rd)$.
\end{enumerate}
Moreover, $$u^{\ep}\sts v\quad\text{strongly two scale in }L^p(\Omega\times Q;\rd)$$
if and only if
$$T_{\ep}u_{\ep}\to v\quad\text{strongly in }L^p(\R^N\times Q;\rd).$$
\end{theorem}
The following proposition is proved in \cite[Proposition A.1]{fonseca.kromer}.
\begin{proposition}
\label{prop:conv-unf-op}
If $u\in L^p(\Omega;\R^d)$ (extended by $0$ outside $\Omega$) then
$$\|u-T_{\ep}u\|_{L^p(\R^N\times Q;\R^d)}\to 0$$
as $\ep\to 0$.
\end{proposition}

 \section{Characterization of limits of $\pdeep{\ep}$-vanishing fields}
 \label{section:pde}
 Let $1<p<+\infty$, and for every $\ep>0$ denote by $\pdeep{\e}:L^p(\Omega;\rd)\to W^{-1,p}(\Omega;\rl)$ the first order differential operator
\be{eq:def-pde-epsilon}
\pdeep{\e}u:=\sum_{i=1}^N\frac{\partial}{\partial {x_i}}\Big(A^i\Big(\frac{x}{\e}\Big)u(x)\Big)
\ee
for $u\in L^p(\Omega;\R^d)$.
In this section we focus on the case in which the operators $A^i$ are smooth and $Q-$periodic, $A^i\in C^{\infty}_{\rm per}(\Rn;\M^{l\times d})$, for all $i=1,\cdots,N$. We will also require that $N\times l=d$, and for every $y\in\R^N$ the operator $\mathcal{A}(y)\in{\rm Lin}\,(\R^d;\R^d)$, defined as
$$\mathcal{A}(y)\xi:=\left(\begin{array}{c}(A^1(y)\xi)^T\\\vdots\\(A^N(y)\xi)^T\end{array}\right)\in \M^{N\times l}\cong \R^{d}\quad\text{for every }\xi\in\R^d,$$
satisfies the uniform ellipticity condition
\be{eq:A-invertibility}
\mathcal{A}(y)\lambda\cdot\lambda\geq\gamma|\lambda|^2\quad\text{for every }\lambda\in\rd\text{ and }y\in\R^N
\ee
where $\gamma>0$ is a positive constant.

We first {prove} a corollary of \cite[Lemma 2.14]{fonseca.muller}.
\begin{lemma}
\label{lemma:div}
Let $1<p<+\infty$ and consider the differential operator
$${\rm div}:\,L^p(Q;\R^{N})\to W^{-1,p}(Q)$$
defined as
$${\rm div }\,R:=\sum_{i=1}^N\frac{\partial R_{i}(y)}{\partial y_i}\quad\text{for every }R\in L^p(Q;\R^{N}).$$
Then, there exists an operator
$$\mathbb{T}:L^p(Q;\R^N)\to L^p(Q;\R^N)$$
such that
\begin{enumerate}
\item[(P1)]$\mathbb{T}$ is linear and bounded, and vanishes on constant maps,
\item[(P2)]$\mathbb{T}\circ\mathbb{T} R=\mathbb{T}R\text{ and }{\rm div }\,(\mathbb{T}R)=0\quad\text{for every }R\in L^p(Q;\R^{N})$,
\item[(P3)]$\text{there exists a constant }C=C(p)>0\text{ such that }$
$$\|R-\mathbb{T}R\|_{L^p(Q;\R^{N})}\leq C\|{\rm div}R\|_{W^{-1,p}(Q)},$$
$\text{for all }R\in L^p(Q;\R^{N})\text{ with }\iQ R(y)\,dy=0,$
\item[(P4)] if $\{v^n\}$ is bounded in $L^p(\Omega\times Q;\R^N)$ and $p-$equiintegrable in $\Omega\times Q$, then setting $w^n(x,\cdot):=\mathbb{T}v^n(x,\cdot)$ for a.e. $x\in\Omega$, the sequence $\{w^n\}$ is $p-$equiintegrable in $\Omega\times Q$,
\item[(P5)] if $\psi\in C^{1}(\Omega;C^{1}_{per}(\R^N;\R^{N}))\cap W^{2,2}(\Omega;W^{2,2}_{per}(\R^N;\R^{N}))$ then setting $\varphi(x,\cdot):=\mathbb{T}\psi(x,\cdot)$ for every $x\in\R^N$, there holds $\varphi\in C^{1}(\Omega;C^{1}_{per}(\R^N;\R^{N}))$.
\end{enumerate}
\end{lemma}
{
\begin{proof}
We first notice that we can rewrite the operator ${\rm div}$ as
$${\rm div }\,R:=\sum_{i=1}^N D^i\frac{\partial R(y)}{\partial y_i}\quad\text{for every }R\in L^p(Q;\R^{N}),$$
where $D^i\in \textrm{Lin }(\R^N;\R)$, $D^i_{m}:=\delta_m^i$ for every $i,m=1,\cdots,N$.

For every $\lambda\in \R^N\setminus \{0\}$, we associate to {\rm div} the linear operator
$$\mathbb{D}(\lambda):=\sum_{i=1}^N D^i\lambda_i\in {\rm Lin}\,(\R^N;\R).$$
By \cite[Remark 3.3 (iv)]{fonseca.muller} the operator ${\rm div }$ satisfies the constant rank condition \eqref{cr}.
Properties (P1), (P2) and (P3) follow directly by \cite[Lemma 2.14]{fonseca.muller}. The proof of (P4) is a straightforward adaptation of the proof of \cite[Lemma 2.14 (iv)]{fonseca.muller}. For convenience of the reader, we sketch the proof of (P5).

 Fix $\psi\in C^{1}(\Omega;C^{1}_{per}(\R^N;\R^{N}))\cap W^{2,2}(\Omega;W^{2,2}_{per}(\R^N;\R^{N}))$.
For every $\lambda\in\R^N\setminus \{0\}$, let $\PP_{D}(\lambda):\R^N\to \R^N$ be the linear projection onto ${\rm Ker}\,\mathbb{D}(\lambda)$. Writing the operator $\mathbb{T}$ explicitly (see \cite[Lemma 2.14]{fonseca.muller}), we have
 $$\mathbb{T}\psi(x,y):=\sum_{\lambda\in\mathbb{Z}^N\setminus\{0\}}\PP_D(\lambda)\hat{\psi}(x,\lambda)e^{2\pi i y\cdot\lambda}$$
 where 
 $$\hat{\psi}(x,\lambda):=\iq \psi(x,z)e^{-2\pi i z\cdot\lambda}\,dz,\quad\text{for }\lambda\in\Z^N\setminus\{0\}\quad\text{and }x\in\Omega,$$
 are the Fourier coefficients associated to $\psi(x,\cdot)$. 
 
  By the regularity of $\psi$ and by 
  Plancherel's theorem for Fourier series, we obtain that
 $$\Bigg(4\pi^2\sum_{\lambda\in \Z^N\setminus \{0\}}|\lambda_i|^2|\hat{\psi}(x,\lambda)|^2\Bigg)^{\tfrac{1}{2}}=\Big\|\frac{\partial\psi(x,y)}{\partial y_i}\Big\|_{L^2(Q;\R^N)}\leq C,\quad\text{for every }x\in\Omega,$$
for $i=1,\cdots,N$.
  Therefore, by \cite[Proposition 2.7]{fonseca.muller} and by Cauchy-Schwartz's inequality, for every $n\in\mathbb{N}$ there holds
  \begin{align*}
  &\sum_{\lambda\in\mathbb{Z}^N\setminus\{0\},\,|\lambda|\geq n}\Big|\PP_D(\lambda)\hat{\psi}(x,\lambda)e^{2\pi i y\cdot\lambda}\Big|\leq C\sum_{\lambda\in\mathbb{Z}^N\setminus\{0\},\,|\lambda|\geq n}|\hat{\psi}(x,\lambda)|\\
  &\quad\leq C\Big(\sum_{\lambda\in\mathbb{Z}^N\setminus\{0\}}|\hat{\psi}(x,\lambda)|^2|\lambda|^2\Big)^{\frac{1}{2}}\Big(\sum_{\lambda\in\mathbb{Z}^N\setminus\{0\},\,|\lambda|\geq n}\frac{1}{|\lambda|^2}\Big)^{\frac{1}{2}}\\
  &\quad\leq C\Big(\sum_{\lambda\in\mathbb{Z}^N\setminus\{0\},\,|\lambda|\geq n}\frac{1}{|\lambda|^2}\Big)^{\frac{1}{2}}.
  \end{align*}
  The previous inequality implies that for every $n$ the tail of the series of functions 
  \be{eq:series-of-functions}\sum_{\lambda\in\mathbb{Z}^N\setminus\{0\}}\PP_D(\lambda)\hat{\psi}(x,\lambda)e^{2\pi i y\cdot\lambda}\ee
  is uniformly bounded by the tail of a convergent series. This yields the uniform convergence of the series in \eqref{eq:series-of-functions}
  and hence the continuity of the map $\mathbb{T}\psi(x,y)$. The differentiability of $\mathbb{T}\psi(x,y)$ follows by an analogous argument.
  \end{proof}}
  Using the previous result we can prove the following projection lemma.
\begin{lemma}
\label{lemma:projection}
Let $1<p<+\infty$, let $A^i\in L^{\infty}(Q;\M^{l\times d}),$ $i=1,\cdots,N,$ with ${\mathcal{A}}$ satisfying the 
 invertibility requirement in \eqref{eq:A-invertibility}. 
Let $\{v^n\},v\subset L^p(\Omega\times Q;\rd)$ be such that
\begin{eqnarray}
\label{eq:pde-0}
&&v^n\wk v\quad\text{weakly in }L^p(\Omega\times Q;\rd),\\
\label{eq:pde-1}
&&\sum_{i=1}^N\frac{\partial}{\partial {x_i}}\Big(\iQ A^i(y)v^n(x,y)\,dy\Big)\to 0\quad\text{strongly in }W^{-1,p}(\Omega;\rl),\\
\label{eq:pde-2}
&&\sum_{i=1}^N\frac{\partial}{\partial {y_i}}\Big(A^i(y)v^n(x,y)\Big)\to 0\quad\text{strongly in }L^p(\Omega;W^{-1,p}(Q;\R^l)).
\end{eqnarray}
Then there exists a subsequence $\{v^{n_k}\}$ and a sequence $\{w^k\}\subset L^p(\Omega\times Q;\rd)$ such that
\begin{eqnarray}
\label{eq:lp-distance}
&&v^{n_k}-w^k\to 0\quad\text{strongly in }L^p(\Omega\times Q;\rd),\\
\label{eq:pde-satisfied-1}
&&\sum_{i=1}^N \frac{\partial}{\partial {x_i}}\Big(\iQ A^i(y)w^k(x,y)\,dy\Big)=0\quad\text{in }W^{-1,p}(\Omega;\rl),\\
\label{eq:pde-satisfied-2}
&&\sum_{i=1}^N \frac{\partial}{\partial {y_i}}\Big(A^i(y)w^{k}(x,y)\Big)=0\quad\text{in }W^{-1,p}(Q;\rl)\text{ for a.e. }x\in\Omega.
\end{eqnarray}
\end{lemma}
\begin{proof}
We first notice that \eqref{eq:pde-0}--\eqref{eq:pde-2} imply that
\begin{eqnarray*}
&&\sum_{i=1}^N\frac{\partial}{\partial {x_i}}\Big(\iQ A^i(y)v(x,y)\,dy\Big)= 0\quad\text{in }W^{-1,p}(\Omega;\R^l),\\
&&\sum_{i=1}^N\frac{\partial}{\partial {y_i}}\Big(A^i(y)v(x,y)\Big)= 0\quad\text{in }W^{-1,p}(Q;\R^l)\text{ for a.e. }x\in\Omega.
\end{eqnarray*}
By linearity, it is enough to consider the case in which $v=0$. Moreover, up to a translation and a dilation, we can assume that $\Omega$ is compactly contained in $Q$. 

By the compact embedding of $L^p(\Omega;\rd)$ into $W^{-1,p}(\Omega;\R^d)$, and by \eqref{eq:pde-0} and \eqref{eq:pde-1}, for every $\varphi\in C^{\infty}_c(\Omega;[0,1])$ there holds
\begin{align*}
&\sum_{i=1}^N \frac{\partial}{\partial {x_i}}\Big(\iQ A^i(y)\varphi(x)v^n(x,y)\,dy\Big)\\
&\quad=\sum_{i=1}^N\varphi(x)\frac{\partial}{\partial {x_i}}\Big(\iQ A^i(y)v^n(x,y)\,dy\Big)+\sum_{i=1}^N\frac{\partial\varphi(x)}{\partial {x_i}}\Big(\iQ A^i(y)v^n(x,y)\,dy\Big)\to 0
\end{align*}
strongly in $W^{-1,p}(\Omega;\R^d)$. On the other hand, by \eqref{eq:pde-2}, 
$$\Big\|\sum_{i=1}^N\frac{\partial}{\partial {y_i}}\Big(A^i(y)v^n(x,y)\Big)\Big\|_{W^{-1,p}(Q;\rl)}\to 0\quad\text{strongly in }L^p(\Omega).$$
Therefore, we may consider a sequence $\{\varphi_k\}\subset C^{\infty}_c(\Omega;[0,1])$ with $\varphi_k\nearrow 1$ and such that, setting $\tilde{v}^n_k:=\varphi_k v^n$ and extending $\tilde{v}^n_k$ by zero to $Q\setminus\Omega$ and then periodically, there holds
\begin{eqnarray*}
&&\tilde{v}^n_k\wk 0\quad\text{weakly in }L^p(Q\times Q;\rd),\\
&&\sum_{i=1}^N\frac{\partial}{\partial {x_i}}\Big(\iQ A^i(y)\tilde{v}^n_k(x,y)\,dy\Big)\to 0\quad\text{strongly in }W^{-1,p}(Q;\rl),\\
&&\Big\|\sum_{i=1}^N\frac{\partial}{\partial {y_i}}\Big(A^i(y)\tilde{v}^n_k(x,y)\Big)\Big\|_{W^{-1,p}(Q;\rl)}\to 0\quad\text{strongly in }L^p(Q),
\end{eqnarray*}
as $n\to +\infty$, $k\to +\infty$.

By a diagonal argument we extract a subsequence $\hat{v}^k:=\tilde{v}^{n(k)}_k$ such that
\begin{eqnarray}
\label{eq:new-pde-0}
&&\hat{v}^k\wk 0\quad\text{weakly in }L^p(Q\times Q;\rd),\\
\label{eq:new-pde-1}
&&\sum_{i=1}^N\frac{\partial}{\partial {x_i}}\Big(\iQ A^i(y)\hat{v}^k(x,y)\,dy\Big)\to 0\quad\text{strongly in }W^{-1,p}(Q;\rl),\\
\label{eq:new-pde-2}
&&\Big\|\sum_{i=1}^N\frac{\partial}{\partial {y_i}}\Big(A^i(y)\hat{v}^k(x,y)\Big)\Big\|_{W^{-1,p}(Q;\rl)}\to 0\quad\text{strongly in }L^p(Q).
\end{eqnarray}

Define the maps 
$$\rn{i}(x,y):=A^i(y)\hat{v}^k(x,y)\quad\text{for a.e. }x\in Q\text{ and }y\in Q,\quad i=1,\cdots,N,$$ 
and let $\rnn\in L^p(Q\times Q;\rd)$ be given by
$$\rnn_{ij}:=(\rn{i})_j,\quad\text{for all }i=1,\cdots,N,\,j=1\cdots,l.$$
 By \eqref{eq:new-pde-0}--\eqref{eq:new-pde-2},
\begin{eqnarray}
\label{eq:new-pde-0-r}
&&R^k\wk 0\quad\text{weakly in }L^p(Q\times Q;\R^d),\\
\label{eq:pde-1-r}
&&\sum_{i=1}^N \frac{\partial}{\partial {x_i}}\Big(\iQ \rn{i}(x,y)\,dy\Big)\to 0\quad\text{strongly in }W^{-1,p}(Q;\rl),\\
\label{eq:pde-2-r}
&&\Big\|\sum_{i=1}^N\frac{\partial}{\partial {y_i}}(\rn{i}(x,y))\Big\|_{W^{-1,p}(Q;\rl)}\to 0\quad\text{strongly in }L^p(Q).
\end{eqnarray}
Using Lemma \ref{lemma:div}, we consider the projection operators $\px$ and $\py$ onto the kernel of the divergence operator with respect to $x$ and the divergence operator with respect to $y$ in the set $Q$. We have
\begin{align}
\label{eq:property2-r}
&\Big\|\px\Big(\iQ \rnn(x,y)\,dy-\iQQ \rnn(w,y)\,dy\,dw\Big)\\
\nonumber&\qquad-\Big(\iQ \rnn(x,y)\,dy-\iQQ \rnn(w,y)\,dy\,dw\Big)\Big\|_{L^p(Q;\rd)}\\
\nonumber&\quad\leq C\Big\|\sum_{i=1}^N \frac{\partial}{\partial {x_i}}\Big(\iQ \rn{i}(x,y)\,dy\Big)\Big\|_{W^{-1,p}(Q;\rl)},
\end{align}
and
\begin{align}
\label{eq:property1-r}
&\Big\|\py\Big(\rnn(x,y)-\iQ\rnn(x,z)\,dz\Big)-\Big(\rnn(x,y)-\iQ\rnn(x,z)\,dz\Big)\Big\|_{L^p(Q\times Q;\rd)}\\
\nonumber &\quad\leq C\Big\|\Big\|\sum_{i=1}^N\frac{\partial }{\partial {y_i}}(\rnn_i(x,y))\Big\|_{W^{-1,p}(Q;\rl)}\Big\|_{L^p(Q)},
\end{align}
which in turn yields
\begin{align}
\label{eq:property1-r-bis}
&\Big\|\iQ\py\Big(\rnn(x,y)-\iQ\rnn(x,z)\,dz\Big)\,dy\Big\|_{L^p(Q;\R^d)}\\
\nonumber&\quad=\Big\|\iQ\Big[\py\Big(\rnn(x,y)-\iQ\rnn(x,z)\,dz\Big)-\Big(\rnn(x,y)-\iQ\rnn(x,z)\,dz\Big)\Big]\,dy\Big\|_{L^p(Q;\R^d)}\\
\nonumber&\quad\leq C\Big\|\Big\|\sum_{i=1}^N\frac{\partial }{\partial {y_i}}(\rnn_i(x,y))\Big\|_{W^{-1,p}(Q;\rl)}\Big\|_{L^p(Q)}.
\end{align}
Set
\bas
\sn(x,y)&:=\py\Big(\rnn(x,y)-\iQ\rnn(x,z)\,dz\Big)-\iQ\Big(\py\Big(\rnn(x,z)-\iQ\rnn(x,\xi)\,d\xi\Big)\Big)\,dz\\
&\quad+\px\Big(\iQ \rnn(x,z)\,dz-\iQQ \rnn(w,z)\,dz\,dw\Big)+\iQQ\rnn(w,z)\,dz\,dw
\end{align*}
for a.e. $(x,y)\in Q\times Q$.
Combining \eqref{eq:new-pde-0-r}--\eqref{eq:property2-r}, we deduce the inequality
\ba{eq:rn-to-0}
&\|\rnn-\sn\|_{L^p(Q\times Q;\rd)}\\
\nn&\quad\leq \Big\|\py\Big(\rnn(x,y)-\iQ\rnn(x,z)\,dz\Big)-\Big(\rnn(x,y)-\iQ\rnn(x,z)\,dz\Big)\Big\|_{L^p(Q\times Q;\rd)}\\
\nn&\qquad+\Big\|\px\Big(\iQ \rnn(x,z)\,dz-\iQQ \rnn(w,z)\,dz\,dw\Big)\\
\nn&\quad\qquad-\Big(\iQ \rnn(x,z)\,dz-\iQQ \rnn(w,z)\,dz\,dw\Big)\Big\|_{L^p(Q;\rd)}\\
\nn&\qquad+ \Big\|\iQ\py\Big(\rnn(x,z)-\iQ\rnn(x,\xi)\,d\xi\Big)\,dz\Big\|_{L^p(Q;\rd))}+\Big|\iQ\iQ R^k(x,y)\,dy\,dx\Big|,
\end{align}
whose right-hand-side converges to zero as $k\to +\infty$. On the other hand, by Lemma \ref{lemma:div} there holds 
\begin{eqnarray}
\label{eq:exact-pde-1}
&&\sum_{i=1}^N\frac{\partial \sn_{ir}(x,y)}{\partial {y_i}}=0\quad\text{in }W^{-1,p}(Q)\quad\text{for a.e. }x\in Q,\\
\label{eq:exact-pde-2}
&&\sum_{i=1}^N\frac{\partial}{\partial {x_i}}\Big(\iQ\sn_{ir}(x,y)\,dy\Big)=0\quad\text{in }W^{-1,p}(Q), 
\end{eqnarray}
for every $k$, for all $r=1,\cdots, l$.
 
Finally, define
$$w^k(x,y):=\mathcal{A}(y)^{-1}\left(\begin{array}{c}\sn_{1}(x,y)\\\vdots\\
\sn_{N}(x,y)\end{array}\right)\quad\text{for a.e. }x\in\Omega\text{ and }y\in \Rn$$
(where the components $\sn_{i}$ are defined analogously to the maps $\rn{i}$).
Properties \eqref{eq:pde-satisfied-1} and \eqref{eq:pde-satisfied-2} follow directly from \eqref{eq:exact-pde-1} and \eqref{eq:exact-pde-2}. Condition \eqref{eq:lp-distance} is a consequence of the boundedness of $\mathcal{A}^{-1}$ (due to \eqref{eq:A-invertibility}) and \eqref{eq:rn-to-0}.
\end{proof}
\begin{remark}
\label{remark:if-regular}
By property (P4) in Lemma \ref{lemma:div}, the boundedness of the operators $A^i$, $i=1,\cdots, N$, and the uniform invertibility condition \eqref{eq:A-invertibility}, it follows that if $\{v^n\}$ is p-equiintegrable, then $\{w^k\}$ is p-equiintegrable as well.

In view of property (P5) in Lemma \ref{lemma:div} if $A^i\in C^{\infty}_{\rm per}(\R^N;\M^{l\times d})$, $i=1\cdots,N$, and $\{v^n\}\subset C^{\infty}_c(\Omega;C^{\infty}_{\rm per}(\R^N;\rd))$, then the sequence $\{w^k\}$ constructed in the proof of Lemma \ref{lemma:div} inherits the same regularity. 
\end{remark}
\medskip
In order to characterize the limit differential constraint, for $u\in L^p(\Omega;\R^d)$ and $n\in\N$ we introduce the classes
\ba{eq:limit-test-functions}
\mathcal{C}^{\pdeor(n\cdot)}_u&:=\Big\{w\in L^p(\Omega;L^p_{\rm per}(\R^N;\rd)):\,\iQ w(x,y)\,dy=0\quad\text{for a.e. }x\in\Omega,\\
\nn&\quad\sum_{i=1}^N\frac{\partial}{\partial {x_i}}\Big(\iQ A^i(ny)(u(x)+w(x,y))\,dy\Big)=0\quad\text{in }W^{-1,p}(\Omega;\rl),\\
\nn&\quad\sum_{i=1}^N\frac{\partial}{\partial {y_i}}\Big(A^i(ny)(u(x)+w(x,y))\Big)=0\quad\text{in }W^{-1,p}(Q;\rl)\text{ for a.e. }x\in\Omega\Big\},
\end{align}
and
\be{eq:limit-domain}
\mathcal{C}^{\pdeor(n\cdot)}:=\{u\in L^p(\Omega;\rd):\, \mathcal{C}^{\pdeor(n\cdot)}_u\neq \emptyset \}.
\ee
For simplicity we will also adopt the notation $\C_u:=\mathcal{C}^{\pdeor(1\cdot)}_u$ and $\C:=\mathcal{C}^{\pdeor(1\cdot)}$. Lemma \ref{lemma:projection} allows us to provide a first characterization of the set $\C$ in the case in which $A^i\in C^{\infty}_{\rm per}(\R^N;\M^{l\times d})$, $i=1\cdots,N$.
  \begin{proposition}
 \label{lemma:limit-maps}
 Let $1<p<+\infty$. Let $A^i\in C^{\infty}_{\rm per}(\R^N;\M^{l\times d})$, $i=1,\cdots,N$, with ${\mathcal{A}}$ satisfying the invertibility requirement in \eqref{eq:A-invertibility}. Let $\C$ be the class introduced in \eqref{eq:limit-domain} and let $\pdeep{\e}$ be the operator defined in \eqref{eq:def-pde-epsilon}. Then 
\ba{eq:charact-C}
 \C&=\Big\{u\in L^p(\Omega;\rd):\,\text{there exists a sequence }\{u_{\e}\}\subset L^p(\Omega;\rd)\text{ such that }\\
\nn&\quad u_{\e}\wk u\quad\text{weakly in }L^p(\Omega;\rd)\text{ and }\pdeep{\e}u_{\e}\to 0\quad\text{strongly in }W^{-1,p}(\Omega;\rl)\Big\}.
 \end{align}
 Moreover, for every $u\in\C$ and $w\in\C_u$ there exists a sequence $\{u_{\e}\}\subset L^p(\Omega;\rd)$ such that
  $$u_{\e}\sts u+w\quad\text{strongly two-scale in }L^p(\Omega\times Q;\rd),$$
   and 
   $$\pdeep{\e}u_{\e}\to 0\quad\text{strongly in }W^{-1,p}(\Omega;\rl).$$
 \end{proposition}
 \begin{proof}
 Denote by $\mathcal{D}$ the set in the right-hand side of \eqref{eq:charact-C}. We divide the proof into two steps.\\
 \emph{Step 1}:
  We first show the inclusion
 $$\mathcal{D}\subset\C.$$
 Let $u\in \mathcal{D}$, and let $\{u_{\e}\}\subset L^p(\Omega;\rd)$ be such that 
 \be{eq:vep-wk-u}
 u_{\e}\wk u\quad\text{weakly in }L^p(\Omega;\rd)\ee
  and 
  \be{eq:vep-pde}
  \pdeep{\e}u_{\e}\to 0\quad\text{strongly in }W^{-1,p}(\Omega;\rl).
  \ee
 Consider a test function $\psi\in C^1_c(\Omega;\rl)$. We have
 \be{eq:ep-maps-to-zero}
 \scal{\pdeep{\e}u_{\e}}{\psi}\to 0,
 \ee
 where $\scal{\cdot}{\cdot}$ denotes the duality product between $W^{-1,p}(\Omega;\rl)$ and $W^{1,p'}_0(\Omega;\rl)$. By definition of the operators $\pdeep{\e}$, 
 $$\scal{\pdeep{\e}u_{\e}}{\psi}=-\iO \sum_{i=1}^N A^i\Big(\frac{x}{\e}\Big)u_{\e}(x)\cdot\frac{\partial \psi(x)}{\partial x_i}\,dx\quad\text{for every }\e>0.$$
 By Proposition \ref{prop:2-scale-compactness} there exists a map $w\in L^p(\Omega;L^p_{\rm per}(\Rn;\rd))$ with \mbox{$\iQ w(x,y)\,dy=0$} such that, up to the extraction of a (not relabeled) subsequence
 \begin{equation}
 \label{eq:subs-wkts}
 u_{\e}\wkts {v}\quad\text{weakly two-scale}
 \end{equation}
 where
 \be{eq:subs-wk2}
 {v}(x,y):=u(x)+w(x,y).
 \ee
 for a.e. $x\in\Omega$, $y\in\Rn$.
 Hence, by the definition of two-scale convergence,
 $$\scal{\pdeep{\e}u_{\e}}{\psi}\to-\iOQ \sum_{i=1}^N A^i(y){v}(x,y)\cdot\frac{\partial\psi(x)}{\partial x_i}\,dy\,dx,$$
 and by \eqref{eq:ep-maps-to-zero} we have that 
 \begin{equation}
 \label{eq:first-cond-C}
 \sum_{i=1}^N\frac{\partial}{\partial {x_i}}\Big(\iQ A^i(y)(u(x)+w(x,y))\,dy\Big)=0\quad\text{in }W^{-1,p}(\Omega;\rl).
 \end{equation}
 
 Let now $\varphi\in C^1_{\rm per}(\Rn;\rl)$, $\psi\in C^1_c(\Omega)$, and consider the sequence of test functions 
 $$\varphi_{\e}(x):=\e\varphi\Big(\frac{x}{\e}\Big)\psi(x)\quad\text{for }x\in\R^N.$$
 The sequence $\{\varphi_{\e}\}$ is uniformly bounded in $W^{1,p'}_0(\Omega;\rl)$, therefore by \eqref{eq:vep-pde}
\be{eq:ep-maps-to-zero-two}
 \scal{\pdeep{\e}u_{\e}}{\varphi_{\e}}\to 0,
 \ee
 with
 \begin{multline*}
 \scal{\pdeep{\e}u_{\e}}{\varphi_{\e}}=-\iO\sum_{i=1}^N A^i\Big(\frac{x}{\e}\Big)u_{\e}(x)\cdot\Big(\frac{\partial\varphi}{\partial y_i}\Big(\frac{x}{\e}\Big)\psi(x)+\e\varphi\Big(\frac{x}{\e}\Big)\frac{\partial\psi(x)}{\partial x_i}\Big)\,dx
 \end{multline*}
 for every $\e$.
 Passing to the subsequence of $\{u_{\e}\}$ extracted in \eqref{eq:subs-wkts}, and applying the definition of two-scale convergence, we obtain
 $$\iOQ \sum_{i=1}^N A^i(y){v}(x,y)\cdot\frac{\partial \varphi(y)}{\partial y_i}\psi(x)\,dy\,dx=0$$
 for every $\varphi\in C^1_{\rm per}(\Rn;\rl)$ and $\psi\in C^1_c(\Omega)$. By density, this equality still holds for an arbitrary $\varphi\in W^{1,{p'}}_0(Q;\rl)$, and so
 \begin{equation}
 \label{eq:second-cond-C}\sum_{i=1}^N\frac{\partial}{\partial {y_i}}\Big(A^i(y)(u(x)+w(x,y))\Big)=0\quad\text{in }W^{-1,p}(Q;\rl)\text{ for a.e. }x\in\Omega.
 \end{equation}
Combining $\eqref{eq:first-cond-C}$ and \eqref{eq:second-cond-C}, we deduce that $u\in \C$.\\
\emph{Step 2}: We claim that $\C\subset \mathcal{D}$. Let $u\in \C$, let $w\in \C_{u}$, and set 
$$v(x,y):=u(x)+w(x,y)\quad\text{for a.e }x\in\Omega\text{ and }y\in \Rn.$$
Let $\{v_{\delta}\}\subset C^{\infty}_c(\Omega; C^{\infty}_{\rm per}(\R^N;\rd))$ be such that 
\be{eq:approx-v-delta}
v_{\delta}\to v\quad\text{strongly in }L^p(\Omega\times Q;\rd).
\ee
The sequence $\{v_{\delta}\}$ satisfies both \eqref{eq:pde-1} and \eqref{eq:pde-2}, hence by Lemma \ref{lemma:projection}  and Remark \ref{remark:if-regular} we can construct a sequence $\{\hat{v}_{\delta}\}\subset C^{\infty}(\Omega;C^{\infty}_{\rm per}(\Rn;\rd))$ such that 
\be{eq:strong-to-u}\hat{v}_{\delta}\to v\quad\text{strongly in }L^p(\Omega\times Q;\rd),\ee
\be{eq:pde-w-delta-1}\sum_{i=1}^N\frac{\partial}{\partial {x_i}}\Big(\iQ A^i(y)\hat{v}_{\delta}(x,y)\,dy\Big)=0\quad\text{in }W^{-1,p}(\Omega;\rl),\ee
and
\be{eq:pde-w-delta-2}\sum_{i=1}^N\frac{\partial}{\partial {y_i}}(A^i(y)\hat{v}_{\delta}(x,y))=0\quad\text{in }W^{-1,p}(Q;\rl)\text{ for a.e. }x\in\Omega.\ee
Consider now the maps
$$u^{\e}_{\delta}(x):=\hat{v}_{\delta}\Big(x,\frac{x}{\e}\Big)$$
for every $x\in\Omega$.
By Proposition \ref{prop:simple-2-scale} we have
$$u^{\e}_{\delta}\sts\hat{v}_{\delta}\quad\text{strongly two-scale in }L^p(\Omega\times Q;\rd)$$
 as $\ep\to 0$, and hence, by Theorem \ref{thm:equivalent-two-scale}
\be{eq:equivalent-unf}
T_{\e}u^{\e}_{\delta}\to\hat{v}_{\delta}\quad\text{strongly in }L^p(\R^N\times Q;\rd)
\ee
(where $T_{\ep}$ is the unfolding operator defined in \eqref{eq:unfolding-operator}).
We observe that by \eqref{eq:pde-w-delta-2},
\be{eq:no-contribution}\sum_{i=1}^N \frac{\partial A^i}{\partial y_i}\Big(\frac{x}{\e}\Big)\hat{v}_{\delta}\Big(x,\frac{x}{\e}\Big)+A^i\Big(\frac{x}{\e}\Big)\frac{\partial \hat{v}_{\delta}}{\partial y_i}\Big(x,\frac{x}{\e}\Big)=0\ee
for all $x\in\Omega$, for every $\e$ and $\delta$. Moreover, by Propositions \ref{prop:2-scale-compactness} and \ref{prop:simple-2-scale},
\ba{eq:good-contribution}
 &\sum_{i=1}^NA^i\Big(\frac{x}{\e}\Big)\frac{\partial \hat{v}_{\delta}}{\partial x_i}\Big(x,\frac{x}{\e}\Big)\wk \sum_{i=1}^N\iQ A^i(y)\frac{\partial\hat{v}_{\delta}}{\partial {x_i}}(x,y)\,dy\\
 \nn&\quad=\sum_{i=1}^N\frac{\partial}{\partial {x_i}}\Big(\iQ A^i(y)\hat{v}_{\delta}(x,y)\,dy\Big)=0
 \end{align}
as $\e\to 0$, weakly in $L^p(\Omega;\rd)$, where the last equality follows by \eqref{eq:pde-w-delta-1}.
Finally, since
\bas
\pdeep{\ep}u^{\ep}_{\delta}(x)&=\sum_{i=1}^N\frac{\partial}{\partial x_i}\Big(A^i\Big(\frac{x}{\ep}\Big)u^{\ep}_{\delta}(x)\Big)\\
\nn&\quad=\frac{1}{\ep}\sum_{i=1}^N \Big[\frac{\partial A^i}{\partial y_i}\Big(\frac{x}{\e}\Big)\hat{v}_{\delta}\Big(x,\frac{x}{\e}\Big)+A^i\Big(\frac{x}{\e}\Big)\frac{\partial \hat{v}_{\delta}}{\partial y_i}\Big(x,\frac{x}{\e}\Big)\Big]\\
\nn&\qquad+\sum_{i=1}^NA^i\Big(\frac{x}{\e}\Big)\frac{\partial \hat{v}_{\delta}}{\partial x_i}\Big(x,\frac{x}{\e}\Big),
\end{align*}
by \eqref{eq:no-contribution}, \eqref{eq:good-contribution}, and the compact embedding of $L^p$ into $W^{-1,p}$, we conclude that
\be{eq:conv-pde-prop}
\pdeep{\ep}u^{\ep}_{\delta}\to 0\quad\text{strongly in }W^{-1,p}(\Omega;\rl),\ee
as $\ep\to 0$.
Collecting \eqref{eq:strong-to-u}, \eqref{eq:equivalent-unf}, and \eqref{eq:conv-pde-prop}, we deduce that
$$\lim_{\delta\to 0}\lim_{\ep\to 0}\Big(\|T_{\ep}u^{\e}_{\delta}-(u+w)\|_{L^p(\Omega\times Q)}+\|\pdeep{\ep}u^{\e}_{\delta}\|_{W^{-1,p}(\Omega;\rl)}\Big)=0.$$
By Attouch's diagonalization lemma \cite[Lemma 1.15 and Corollary 1.16]{attouch}, there exists a subsequence $\{\delta(\e)\}$ such that
$$\lim_{\ep\to 0}\Big(\|T_{\ep}u^{\e}_{\delta(\e)}-(u+w)\|_{L^p(\Omega\times Q)}+\|\pdeep{\ep}u^{\e}_{\delta(\e)}\|_{W^{-1,p}(\Omega;\rl)}\Big)=0.$$
Setting $u^{\e}:=u^{\e}_{\delta_{\e}}$, we finally obtain 
$$\pdeep{\e}u_{\ep}\to 0\quad\text{strongly in }W^{-1,p}(\Omega;\rl)$$
and
$$u^{\e}\sts u+w\quad\text{strongly two-scale in }L^p(\Omega\times Q;\rd),$$
and hence, by Proposition \ref{prop:2-scale-compactness}, 
$$u^{\e}\wk u\quad\text{weakly in }L^p(\Omega;\rd).$$
This yields that $u\in\mathcal{D}$ and completes the proof of the proposition. 
  \end{proof}
   \begin{remark}
  \label{rk:not-so-easy}
  The regularity of the operators $A^i$ played a key role in Step 2. In the case in which $A^i\in L^{\infty}_{\rm per}(\R^N;\M^{l\times d})$, $i=1,\cdots, N$, but we have no further smoothness assumptions on the operators, the argument in Step 1 still guarantees that   
  \ba{eq:charact-C-bdd-first-side}
 &\Big\{u\in L^p(\Omega;\rd):\,\text{there exists a sequence }\{u_{\e}\}\subset L^p(\Omega;\rd)\text{ such that }\\
 \nn&\quad u_{\e}\wk u\quad\text{weakly in }L^p(\Omega;\rd)\\
 \nn&\quad\text{ and }\pdeep{\e}u_{\e}\to 0\quad\text{strongly in }W^{-1,q}(\Omega;\rl)\quad\text{for every }1\leq q<p\Big\}\subset\C.
 \end{align}
 Indeed, arguing as in Step 1 we obtain that there exists $w\in L^p(\Omega;L^p_{\rm per}(\R^N;\rd))$ with $\iQ w(x,y)\,dy=0$, such that
 $$u_{\e}\wkts u+w\quad\text{weakly two-scale in }L^p(\Omega\times Q;\rd),$$
 and
 \ba{eq:pde-q}
 &\sum_{i=1}^N\frac{\partial}{\partial {x_i}}\Big(\iQ A^i(y)(u(x)+w(x,y))\,dy\Big)=0\quad\text{in }W^{-1,q}(\Omega;\rl)\\
\nn&\sum_{i=1}^N\frac{\partial}{\partial {y_i}}(A^i(y)(u(x)+w(x,y)))=0\quad\text{in }W^{-1,q}(Q;\rl)\text{ for a.e. }x\in\Omega,
 \end{align}
 for all $1\leq q<p$. Since $u+w\in L^p(\Omega;L^p_{\rm per}(\R^N;\rd))$, it follows that \eqref{eq:pde-q} holds also for $q=p$. Therefore we deduce the inclusion \eqref{eq:charact-C-bdd-first-side}. 
 
 The proof of the opposite inclusion, on the other hand, is not a straightforward consequence of Proposition \ref{lemma:limit-maps}. In fact, in the case in which the operators $A^i$ are only bounded, the second conclusion in Remark \ref{remark:if-regular} does not hold anymore, and we are not able to guarantee that the projection operator provided by Lemma \ref{lemma:projection} preserves the regularity of smooth functions. Therefore, the measurability of the maps $u^{\ep}_{\delta}$ is questionable (see \cite[discussion below Definition 1.4]{allaire}).
 This difficulty will be overcome in Lemma \ref{lemma:limit-maps-bdd} by means of an approximation of the operators $A^i$ with $C^{\infty}$ operators.
  \end{remark}
 \section{Homogenization for smooth operators}
 \label{section:smooth}
   We recall that
   \be{eq:def-limit-energy}
\FFN^{\,r}(u):=\inf \Big\{\iOQ f(u(x)+w(x,y))\,dy\,dx:\,w\in{\mathcal{C}}_u^{\pdeor(n\cdot)},\|w\|_{L^p(\Omega\times Q;\R^d)}\leq r\Big\}
\ee
for every $u\in {\mathcal{C}}^{\pdeor(n\cdot)}_r$ and $r>0$, where
$$\mathcal{C}^{\pdeor(n\cdot)}_r:=\{v\in L^p(\Omega;\R^d):\,\exists w\in {\mathcal{C}_v^{\pdeor(n\cdot)}}\quad\text{with }\|w\|_{L^p(\Omega\times Q;\R^d)}\leq r\},$$
\be{eq:def-extended-energy}
\FBN^{\,r}(u):=\begin{cases}\FFN^{\,r}(u)&\text{if }u\in{\mathcal{C}}^{\pdeor(n\cdot)}_r,\\
+\infty&\text{otherwise in }L^p(\Omega;\rd),
\end{cases}
\ee
for every $r>0$,
and
\be{eq:def-lsc-energy}
\F(u):=\begin{cases}\inf_{r>0}\inf\Big\{\liminfn\FBN^{\,r}(u_n):u_n\wk u\quad\text{weakly in }L^p(\Omega;\rd)\Big\}\\
\text{if }u\in\C,\\
+\infty\quad\text{otherwise in }L^p(\Omega;\rd).\end{cases}
\ee
\begin{remark}
\label{rk:important}
We observe that for every $u\in\C$ there holds
$${\mathcal S}:=\Big\{\{u_n\}:\,u_n\wk u\quad\text{weakly in }L^p(\Omega;\R^d),\,u_n\in\mathcal{C}^{\pdeor(n\cdot)}\quad\text{for every }n\in\N\Big\}\neq\emptyset$$
{and $\F(u)<+\infty$. Indeed, let $u\in\C$ and $w\in\C_u$. Then a change of variables and the periodicity of $w$ yield immediately that
$$\iQ w(x,ny)\,dy=0\quad\text{for a.e. }x\in\Omega$$
and
$$\sum_{i=1}^n\frac{\partial}{\partial x_i}\Big(\iQ(A^i(ny)(u(x)+w(x,ny)))\,dy\Big)=0\quad\text{in }W^{-1,p}(\Omega;\R^l).$$
Proving that 
$$\sum_{i=1}^n\frac{\partial}{\partial y_i}(A^i(ny)(u(x)+w(x,ny)))=0\quad\text{in }W^{-1,p}(Q;\R^l)\text{ for a.e. }x\in\Omega$$
is equivalent to showing that
\be{eq:remark-to-do}
\sum_{i=1}^n\frac{\partial}{\partial y_i}(A^i(y)(u(x)+w(x,y)))=0\quad\text{in }W^{-1,p}(nQ;\R^l)\text{ for a.e. }x\in\Omega.
\ee
To this purpose, arguing as in Step 2 of the proof of Proposition \ref{lemma:limit-maps}, construct $\{\hat{v}^{\delta}\}\subset C^{\infty}(\Omega;\, C^{\infty}_{\rm per}(\R^N;\R^d))$ such that
\ba{eq:remark-lp-conv}
&\hat{v}^{\delta}\to u+w\quad\text{strongly in }L^p(\Omega\times Q;\R^d),\\
&\sum_{i=1}^n\frac{\partial}{\partial x_i}\Big(\iQ(A^i(ny)\hat{v}^{\delta}(x,y))\,dy\Big)=0\quad\text{in }W^{-1,p}(\Omega;\R^l),\\
&\sum_{i=1}^n\frac{\partial}{\partial y_i}(A^i(y)\hat{v}^{\delta}(x,y))=0\quad\text{in }W^{-1,p}(Q;\R^l)\text{ for a.e. }x\in\Omega.
\end{align}
By the smoothness and the periodicity of $\{\hat{v}^{\delta}\}$ there holds
$$\sum_{i=1}^n\frac{\partial}{\partial y_i}(A^i(y)\hat{v}^{\delta}(x,y))=0\quad\text{in }W^{-1,p}(nQ;\R^l)\text{ for a.e. }x\in\Omega$$
and \eqref{eq:remark-to-do} follows in view of \eqref{eq:remark-lp-conv}. By the previous argument,}
  $w(x,ny)\in\mathcal{C}^{\pdeor(n\cdot)}_{u}$, therefore the set $\mathcal{S}$ contains always the sequence $u_n:=u$ for every $n$.

\end{remark}
\begin{theorem}
\label{thm:main-homogenized-pde}
 Let $1<p<+\infty$. Let $A^i\in C^{\infty}_{\rm per}(\R^N;\M^{l\times d})$, $i=1,\cdots, N$, assume that the operator ${\mathcal{A}}$ satisfies the invertibility requirement in \eqref{eq:A-invertibility}, and let $\pdeep{\e}$ be the operator defined in \eqref{eq:def-pde-epsilon}. Let $f:\R^d\to [0,+\infty)$ be a continuous function satisfying the growth condition
\be{eq:growth-p-pde}
0\leq f(v)\leq C(1+|v|^p)\quad\text{for every }v\in\rd,
\ee
where $C>0$. Then, for every $u\in L^p(\Omega;\rd)$ there holds
 \bas
 &\inf\Big\{\liminf_{\ep\to 0}\iO f(u_{\ep}(x))\,dx:\,u_{\ep}\wk u\quad\text{weakly in }L^p(\Omega;\rd)\\
 &\qquad\text{and }\pdeep{\e}u_{\e}\to 0\quad\text{strongly in }W^{-1,p}(\Omega;\rl)\Big\}\\
 &\quad=\inf\Big\{\limsup_{\ep\to 0}\iO f(u_{\ep}(x))\,dx:\,u_{\ep}\wk u\quad\text{weakly in }L^p(\Omega;\rd)\\
 &\qquad\text{and }\pdeep{\e}u_{\e}\to 0\quad\text{strongly in }W^{-1,p}(\Omega;\rl)\Big\}=\F(u).
 \end{align*}
\end{theorem}
Before starting the proof of Theorem \ref{thm:main-homogenized-pde}, we first state without proving a corollary of \cite[Lemma 2.15]{fonseca.muller} and one of \cite[Lemma 2.8]{fonseca.kromer}, and we prove an adaptation of \cite[Lemma 2.15]{fonseca.muller} to our framework.
\begin{lemma}
\label{lemma:cor-f-m}
Let $1<p<+\infty$. Let $\{u_{\ep}\}$ be a bounded sequence in $L^p(\Omega;\R^N)$ such that
\be{eq:cor-f-m}{\rm div}\,u_{\ep}\to 0\quad\text{strongly in }W^{-1,p}(\Omega)\ee
and 
$$u_{\ep}\wk u\quad\text{weakly in }L^p(\Omega;\R^N).$$
Then there exists a $p$-equiintegrable sequence $\{\tilde{u}_{\ep}\}$ such that
\bas
&{\rm div }\,\tilde{u}_{\ep}=0\quad\text{in }W^{-1,p}(\Omega)\quad\text{for every }\ep,\\
&\tilde{u}_{\ep}-u_{\ep}\to 0\quad\text{strongly in }L^q(\Omega;\R^N)\quad\text{for every }1\leq q<p,\\
&\tilde{u}_{\ep}\wk u\quad\text{weakly in }L^p(\Omega;\R^N).
\end{align*}
\end{lemma}
\begin{remark}
\label{rk:add-for-nonsmooth}
A direct adaptation of the proof of \cite[Lemma 2.15]{fonseca.muller} yields also that the thesis of Lemma \ref{lemma:cor-f-m} still holds if we replace \eqref{eq:cor-f-m} with the condition
$${\rm div}\,u_{\ep}\to 0\quad\text{strongly in }W^{-1,q}(\Omega)$$
for every $1\leq q<p$.
\end{remark}
\begin{lemma}
\label{lemma:A-free-ext}
Let $1<p<+\infty$, and let $D\subset Q$. Let $\{u_{\ep}\}\subset L^p(D;\R^N)$ be $p$-equiintegrable, with
$$u_{\ep}\wk 0\quad\text{weakly in }L^p(D;\R^N),$$
and
$${\rm div}\,u_{\ep}\to 0\quad\text{strongly in }W^{-1,p}(D).$$
Then there exists a $p$-equiintegrable sequence $\{\tilde{u}_{\ep}\}\subset L^p(Q;\R^N)$ such that
\bas
&\tilde{u}_{\ep}-u_{\ep}\to 0\quad\text{strongly in }L^p(D;\R^N),\\
&\tilde{u}_{\ep}\to 0\quad\text{strongly in }L^p(Q\setminus D;\R^N),\\
&{\rm div}\,\tilde{u}_{\ep}=0\quad\text{in }W^{-1,p}(Q),\\
&\|\tilde{u}_{\ep}\|_{L^p(Q;\R^N)}\leq C\|u_{\ep}\|_{L^p(D;\R^N)},\\
&\iQ\tilde{u}_{\ep}(x)\,dx=0\quad\text{for every }\ep.
\end{align*}
\end{lemma}
More generally, we have
\begin{lemma}
\label{lemma:truncation}
Let $1<p<+\infty$, $u\in L^p(\Omega;\rd)$ and let $\{u^{\ep}\}\subset L^p(\Omega;\rd)$ be such that
\begin{eqnarray}
&&\label{eq:1}\,u_{\e}\wk u\text{ weakly in }L^p(\Omega;\rd)\\
&&\label{eq:2}\pdeep{\e}u_{\e}\to 0\text{ strongly in  }W^{-1,p}(\Omega;\rl).
\end{eqnarray}
Then there exists a $p$-equiintegrable sequence $\{\tilde{u}_{\e}\}$ such that
\begin{eqnarray*}
&&\tilde{u}_{\e}\wk u\quad\text{weakly in }L^p(\Omega;\rd),\\
&&\pdeep{\e}\tilde{u}_{\e}\to 0\quad\text{strongly in }W^{-1,q}(\Omega;\rl)\quad\text{for every } 1\leq q<p,\\
&&\tilde{u}_{\e}-u_{\e}\to 0\quad\text{strongly in }L^q(\Omega;\rd)\quad\text{for every } 1\leq q<p.
\end{eqnarray*}
\end{lemma}
{\begin{proof}
The proof follows by adapting \cite[Lemma 2.15]{fonseca.muller}. We sketch it for convenience of the reader.

Consider the truncation function
$$\tau_k(z):=\begin{cases}z&\text{if }|z|\leq k,\\
k\Big(\frac{z}{|z|}\Big)&\text{if }|z|>k,\end{cases}$$
and set $u^k_{\ep}:=\tau_k\circ u_{\ep}$. Then by \eqref{eq:1},
$$u^k_{\ep}\wk u\quad\text{weakly in }L^p(\Omega;\rd)$$
as $k\to +\infty$ and $\ep\to 0$, in this order. Moreover, by \eqref{eq:1}
$$\|u^k_{\ep}-u_{\ep}\|^q_{L^q(\Omega;\rd)}\leq \int_{|u_{\ep}|\geq k}2^q|u_{\ep}(x)|^q\,dx\leq\frac{2^q}{k^{p-q}}\iO|u_{\ep}(x)|^p\,dx\to 0$$
as $k\to +\infty$, for every $1\leq q<p$. Therefore,
$$\pdeep{\ep}(u^k_{\ep}-u_{\ep})\to 0\quad\text{strongly in }W^{-1,q}(\Omega;\rl)$$
as $k\to +\infty$, for every $1\leq q<p$, and by \eqref{eq:2}
$$\pdeep{\ep}u^k_{\ep}\to 0\quad\text{strongly in }W^{-1,q}(\Omega;\rl)$$
as $k\to +\infty$ and $\ep\to 0$ in this order, for every $1\leq q<p$. The thesis follows now by a diagonal argument.
\end{proof}}
The following propositions is a corollary of \cite[Proposition 3.5 (ii)]{fonseca.kromer}.
\begin{proposition}
\label{prop:f-k}
Let $f:\rd\to [0,+\infty)$ be a continuous map satisfying \eqref{eq:growth-p-pde} for some $1<p<+\infty$. Let $\{u_{\ep}\}\subset L^p(\Omega;\rd)$ be a bounded sequence and let $\{\tilde{u}_{\e}\}\subset L^p(\Omega;\rd)$ be $p$-equiintegrable and such that
$$u_{\ep}-\tilde{u}_{\ep}\to 0\quad\text{in measure}.$$
Then
$$\liminf_{\ep\to 0}\iO f(u_{\ep}(x))\,dx\geq\liminf_{\ep\to 0}\iO f(\tilde{u}_{\ep}(x))\,dx.$$
Moreover, if $g:\R^N\times \rd\to[0,+\infty)$ satisfies
\begin{enumerate}
\item[(i)] $g(\cdot,\xi)$ is measurable and $Q$-periodic for every $\xi\in\R^d$,
\item[(ii)] $g(y,\cdot)$ is continuous for a.e. $y\in\R^N$,
\item[(iii)] there exists a constant $C$ such that
$$0\leq g(y,\xi)\leq C(1+|\xi|^p)\quad\text{for a.e. }y\in\R^N\text{and }\xi\in\R^d,$$
then
$$\liminf_{\ep\to 0}\iO g\Big(\frac{x}{\ep}, u_{\ep}(x)\Big)\,dx\geq\liminf_{\ep\to 0}\iO g\Big(\frac{x}{\ep},\tilde{u}_{\ep}(x)\Big)\,dx.$$
\end{enumerate}
\end{proposition}
The next proposition is another corollary of \cite[Proposition 3.5 (ii)]{fonseca.kromer}.
{We sketch its proof for the convenience of the reader.}
\begin{proposition}
\label{prop:to-add1}
Let $1\leq p<+\infty$ and $\lambda\in (0,1]$. Let $g:\R^N\times \rd\to[0,+\infty)$ be such that
\begin{enumerate}
\item[(i)] $g(\cdot,\xi)$ is measurable and $Q$-periodic for every $\xi\in\R^d$,
\item[(ii)] $g(y,\cdot)$ is continuous for a.e. $y\in\R^N$,
\item[(iii)] there exists a constant $C$ such that
$$0\leq g(y,\xi)\leq C(1+|\xi|^p)\quad\text{for a.e. }y\in\R^N\text{and }\xi\in\R^d,$$
\end{enumerate}
and let $V$ be a $p$-equiintegrable subset of $L^p(\Omega\times Q;\R^d)$.
Then there exists a constant $C$ such that
$$\Big\|g\Big(\frac{y}{\lambda}, v_1(x,y)\Big)-g\Big(\frac{y}{\lambda}, v_2(x,y)\Big)\Big\|_{L^1(\Omega\times Q)}\leq C\|v_1-v_2\|_{L^p(\Omega\times Q;\R^d)}$$
for every $v_1,v_2\in V$.
\end{proposition}
{\begin{proof}
Fix $\ep>0$. By the $p$-equiintegrability of $V$ and $(iii)$ there exists $\delta_1>0$ such that for every $E\subset \Omega\times Q$ measurable with $|E|\leq \delta_1$, and for all $v\in V$, there holds
\be{eq:1-prop}\int_E \Big|g\Big(\frac{y}{\lambda}, v(x,y)\Big)\Big|\,dx\,dy<\frac{\ep}{7}.\ee
In view of the $p$-equiintegrability of $V$, there exists also $R>0$ such that
\be{eq:2-prop}
\sup_{v\in V}|\{|v|\geq R\}|<\frac{\delta_1}{2}.
\ee
Without loss of generality, up to a translation and a dilation, we can assume that $\Omega\subset Q$. By Scorza-Dragoni Theorem (see {e.g.} \cite[Proposition B.1]{fonseca.kromer}) there exists a compact set $K\subset Q$  such that 
\be{eq:2bis-prop}|Q\setminus K|<3^{-N}\frac{\delta_1}{|\Omega|},\ee
and $g$ is uniformly continuous on $K\times \bar{B}_R(0)$.
Hence, by the periodicity of $g$ with respect to its first variable, there exists $\delta_2>0$ such that
\be{eq:3-prop}
|g(\eta,\xi_1)-g(\eta,\xi_2)|<\frac{\ep}{7|\Omega\times Q|}
\ee
for every $\xi_1,\xi_2\in \bar{B}_R(0)$ such that $|\xi_1-\xi_2|<\delta_2$ and $\eta\in K+\Z^N$. We observe that by Chebyshev's inequality there exists $\delta>0$ such that if 
\be{eq:delta}\|v_1-v_2\|_{L^p(\Omega\times Q;\R^d)}\leq \delta,\ee
 then
\be{eq:3bis-prop}|\{|v_1-v_2|\geq \delta_2\}|<\delta_1.\ee
Let $v_1,v_2\in V$ be such that \eqref{eq:delta} holds true. Decompose the set $\Omega\times Q$ as
\bas
\Omega\times Q=\{|v_1-v_2|<\delta_2\}\cup\{|v_1-v_2|\geq \delta_2\}.
\end{align*}
By \eqref{eq:1-prop} and \eqref{eq:3bis-prop}, we have
\bas
\int_{\{|v_1-v_2|\geq \delta_2\}}\Big|g\Big(\frac{y}{\lambda}, v_1(x,y)\Big)-g\Big(\frac{y}{\lambda}, v_2(x,y)\Big)\Big|\,dy\,dx\leq \frac{2\ep}{7}.
\end{align*}
We perform a decomposition of the set $\{|v_1-v_2|< \delta_2\}$ as
\bas
&\{|v_1-v_2|< \delta_2\}
=\{|v_1-v_2|< \delta_2,\,|v_1|<R\text{ and }|v_2|<R\}\\
&\quad\cup\{|v_1-v_2|< \delta_2,\,\max\{|v_1|,|v_2|\}\geq R\}.
\end{align*}
In view of \eqref{eq:2-prop},
$$|\{|v_1-v_2|< \delta_2,\,\max\{|v_1|,|v_2|\}\geq R\}|\leq |\{|v_1|\geq R\}|+|\{|v_2|\geq R\}|\leq \delta_1,$$
hence by \eqref{eq:1-prop}  
$$\int_{\{|v_1-v_2|< \delta_2,\,\max\{|v_1|,|v_2|\}\geq R\}}\Big|g\Big(\frac{y}{\lambda}, v_1(x,y)\Big)-g\Big(\frac{y}{\lambda}, v_2(x,y)\Big)\Big|\,dy\,dx\leq \frac{2\ep}{7}.$$
Defining $m_{\lambda}:=\floor[\big]{\frac{1}{2\lambda}}+1$, by \eqref{eq:2bis-prop} and since $\lambda<1$ there holds
$${2 m_{\lambda}\lambda}={2\lambda}\Big(1+\floor[\Big]{\frac{1}{2\lambda}}\Big)\leq 2\lambda+1\leq 3,$$
and
\be{eq:4-prop}|Q\setminus \lambda(K+\Z^N)|\leq \lambda^N|[-m_{\lambda},m_{\lambda}]^N\setminus (K+\Z^N)|\leq \lambda^N (2m_{\lambda})^N|Q\setminus K|\leq \frac{\delta_1}{|\Omega|}.\ee
Setting $E_R:=\{|v_1-v_2|< \delta_2,\,|v_1|<R\text{ and }|v_2|<R\}$, a further decomposition yields
\bas
E_R=\{E_R\cap[\Omega\times(Q\setminus \lambda(K+\Z^N))]\}\cup\{E_R\cap[\Omega\times (Q\cap\lambda(K+\Z^N))]\}.
\end{align*}
By \eqref{eq:1-prop} and \eqref{eq:4-prop} we obtain the estimate
\be{eq:5-prop}\int_{E_R\cap[\Omega\times(Q\setminus \lambda(K+\Z^N))]}\Big|g\Big(\frac{y}{\lambda},v_1(x,y)\Big)-g\Big(\frac{y}{\lambda},v_2(x,y)\Big)\Big|\,dy\,dx\leq \frac{2\ep}{7}\ee
and by \eqref{eq:3-prop},
\be{eq:6-prop}
\int_{E_R\cap[\Omega\times (Q\cap\lambda(K+\Z^N))]}\Big|g\Big(\frac{y}{\lambda},v_1(x,y)\Big)-g\Big(\frac{y}{\lambda},v_2(x,y)\Big)\Big|\,dy\,dx|\leq \frac{\ep}{7}.
\ee
The thesis follows combining \eqref{eq:delta}--\eqref{eq:6-prop}.
\end{proof}}

We now start the proof of Theorem \ref{thm:main-homogenized-pde}. First we prove the liminf inequality. The argument relies on the use of the \emph{unfolding operator} (see Subsection \ref{subsection:unfolding} and \cite[Appendix A]{fonseca.kromer} and the references therein). 
\begin{proposition}
\label{thm:liminf}
Under the assumptions of Theorem \ref{thm:main-homogenized-pde} for every $u\in L^p(\Omega;\rd)$ there holds
\begin{multline}
\label{eq:liminf-to-prove}
\inf\Big\{\liminf_{\e\to 0} \int_{\Omega}f(u_{\e}(x))\,dx:\,u_{\e}\wk u\text{ weakly in }L^p(\Omega;\rd)\\
\text{and }\pdeep{\e}u_{\e}\to 0\text{ strongly in  }W^{-1,p}(\Omega;\rl)\Big\}\geq \F(u).
\end{multline}
\end{proposition}
\begin{proof}
Let $\C$ be the class introduced in \eqref{eq:limit-domain}.
We first notice that, by Proposition \ref{lemma:limit-maps}, if $u\in L^p(\Omega;\rd)\setminus \C$ then 
\begin{multline*}
\Big\{\{u_{\e}\}\subset L^p(\Omega;\rd):\,u_{\e}\wk u\text{ weakly in }L^p(\Omega;\rd)\\
\text{and }\pdeep{\e}u_{\e}\to 0\text{ strongly in  }W^{-1,p}(\Omega;\rl)\}\Big\}=\emptyset,
\end{multline*}
hence \eqref{eq:liminf-to-prove} follows trivially.

Define $g:\R^N\times \R^d\to [0,+\infty)$ as $$g(y,\xi)=f({\mathcal{A}}(y)^{-1}\xi)\quad\text{for every }y\in\R^N,\text{ and }\xi\in\R^d.$$
By the continuity of $f$, $g$ is measurable with respect to the first variable (it is the composition of a continuous function with a measurable one), and continuous with respect to the second variable. By \eqref{eq:growth-p-pde}, there holds 
\be{eq:prop-g}0\leq g(y,\xi)\leq C(1+|{\mathcal{A}}(y)^{-1}\xi|^p)\leq C(1+|\xi|^p),\ee
where the last inequality follows by the uniform invertibility assumption \eqref{eq:A-invertibility}. We divide the proof of the proposition into five steps.\\
\emph{Step 1:} We first show that for every $u\in\C$ there holds
\ba{eq:step1}
\inf&\Big\{\liminf_{\ep\to 0} \int_{\Omega}f(u_{\e}(x))\,dx:\, u_{\e}\wk u\quad\text{weakly in }L^p(\Omega;\R^d)\\
\nn&\qquad\text{and }\pdeep{\e}u_{\e}\to 0\quad\text{strongly in }W^{-1,p}(\Omega;\R^l)\Big\}\\
\nn&\geq\quad \inf_{w\in\C_u}\inf \Bigg\{\liminfn \int_{\Omega} g\Big(\frac{x}{\en},v_{\en}(x)\Big)\,dx:\\
\nn&\qquad v_{\en}\wk \iQ {\mathcal{A}}(y)(u(x)+w(x,y))\,dy\quad\text{weakly in }L^p(\Omega;\R^d),\\
\nn&\qquad {\rm div}\, v_{\en}\to 0\quad\text{strongly in }W^{-1,p}(\Omega;\R^l),\\
\nn&\qquad\text{and }\aaa\Big(\frac{x}{\en}\Big)^{-1}v_{\en}\wk u\quad\text{weakly in }L^p(\Omega;\R^d)\Bigg\}.
\end{align}

Indeed let $u\in\C$ and let $\{u_{\ep}\}$ be as in \eqref{eq:step1}. Up to the extraction of a subsequence $\{\en\}$, 
$$\liminf_{\ep\to 0} \int_{\Omega}f(u_{\e}(x))\,dx=\lim_{n\to +\infty}\int_{\Omega}f(u_{\en}(x))\,dx,$$
and by Proposition \ref{prop:2-scale-compactness} there exists $w\in\C_u$ such that
\be{eq:step1-uen-2-scale}u_{\en}\wkts u+w\quad\text{weakly two-scale in }L^p(\Omega\times Q;\R^d).\ee
By the definition of $g$ it is straightforward to see that
$$\lim_{n\to +\infty}\int_{\Omega}f(u_{\en}(x))\,dx=\lim_{n\to +\infty}\int_{\Omega}g\Big(\frac{x}{\en},\aaa\Big(\frac{x}{\en}\Big)u_{\en}(x)\Big)\,dx.$$
Setting $v_{\en}:=\aaa\big(\frac{x}{\en}\big)u_{\en}$, in view of the assumptions on $\{u_{\ep}\}$ in \eqref{eq:step1}, it follows that
$${\rm div}\, v_{\en}=\sum_{i=1}^N\frac{\partial}{\partial x_i}\Big(A^i\Big(\frac{x}{\en}\Big)u_{\en}(x)\Big)\to 0\quad\text{strongly in }W^{-1,p}(\Omega;\rl),$$
and
\be{eq:step1-inv-op}\aaa\Big(\frac{x}{\en}\Big)^{-1}v_{\en}\wk u\quad\text{weakly in }L^p(\Omega;\R^d).\ee
Finally, by \eqref{eq:step1-uen-2-scale} for every $\varphi\in L^{{p}'}(\Omega;\R^d)$ there holds
\bas
&\lim_{n\to +\infty}\int_{\Omega}v_{\en}(x)\cdot\varphi(x)\,dx
=\lim_{n\to +\infty}\int_{\Omega}u_{\en}(x)\cdot\aaa\Big(\frac{x}{\en}\Big)^T\varphi(x)\,dx\\
&\quad=\iOQ (u(x)+w(x,y))\cdot\aaa(y)^T\varphi(x)\,dy\,dx.
\end{align*}
This completes the proof of \eqref{eq:step1}.\\
\emph{Step 2:} We claim that
\ba{eq:step2}
&\inf_{w\in\C_u}\inf \Bigg\{\liminfn \int_{\Omega} g\Big(\frac{x}{\en},v_{\en}(x)\Big)\,dx:\\
\nn&\qquad v_{\en}\wk \iQ {\mathcal{A}}(y)(u(x)+w(x,y))\,dy\quad\text{weakly in }L^p(\Omega;\R^d),\\
\nn&\qquad {\rm div}\, v_{\en}\to 0\quad\text{strongly in }W^{-1,p}(\Omega;\R^l),\\
\nn&\qquad\text{and }\aaa\Big(\frac{x}{\en}\Big)^{-1}v_{\en}\wk u\quad\text{weakly in }L^p(\Omega;\R^d)\Bigg\}\\
\nn&\geq\quad \inf_{w\in\C_u}\inf \Bigg\{\liminfn \int_{\Omega} g\Big(\frac{x}{\en},\iQ {\mathcal{A}}(y)(u(x)+w(x,y))\,dy+\tilde{v}_{\en}(x)\Big)\,dx:\\
\nn&\qquad \{\tilde{v}_{\en}\}\text{ is p-equiintegrable},\\
\nn&\qquad \tilde{v}_{\en}\wk 0\quad\text{weakly in }L^p(\Omega;\R^d),\\
\nn&\qquad{\rm div}\, \tilde{v}_{\en}= 0\quad\text{in }W^{-1,p}(\Omega;\R^l),\, \text{and} \\
\nn&\qquad \aaa\Big(\frac{x}{\en}\Big)^{-1}\tilde{v}_{\en}\wk u(x)-\Big(\iQ\aaa(z)^{-1}\,dz\Big)\iQ\aaa(y)(u(x)+w(x,y))\,dy\\
\nn&\qquad\text{weakly in }L^p(\Omega;\R^d)\Bigg\}.
\end{align}
Let $\{v_{\en}\}$ be as in \eqref{eq:step2}. By Lemma \ref{lemma:cor-f-m}, we construct a p-equiintegrable sequence $\{\bar{v}_{\en}\}$ such that
\ba{eq:step2-properties} 
&{\rm div}\,\bar{v}_{\en}=0\quad\text{in }W^{-1,p}(\Omega;\R^l),\\
\label{eq:step2-p2}&v_{\en}-\bar{v}_{\en}\to 0\quad\text{strongly in }L^q(\Omega;\R^d)\quad\text{for every }1\leq q<p,\\
\label{eq:step2-p3}&\bar{v}_{\en}\wk\iQ {\mathcal{A}}(y)(u(x)+w(x,y))\,dy\quad\text{weakly in }L^p(\Omega;\R^d).
\end{align}
Moreover, by Proposition \ref{prop:f-k},
$$\lim_{n\to +\infty}\int_{\Omega}g\Big(\frac{x}{\en},v_{\en}(x)\Big)\,dx\geq\liminfn\int_{\Omega}g\Big(\frac{x}{\en},\bar{v}_{\en}(x)\Big)\,dx.$$
By \eqref{eq:step2-p3} and by the uniform invertibility assumption \eqref{eq:A-invertibility}, there exists a constant $C$ such that 
$$\Big\|\aaa\Big(\frac{x}{\en}\Big)^{-1}\bar{v}_{\en}\Big\|_{L^p(\Omega;\R^d)}\leq C\quad\text{for every }n.$$
Therefore, there exists a map $\phi\in L^p(\Omega;\R^d)$ such that, up to the extraction of a (not relabeled) subsequence,
\be{eq:step2-inv-op}
\aaa\Big(\frac{x}{\en}\Big)^{-1}\bar{v}_{\en}\wk\phi\quad\text{weakly in }L^p(\Omega;\R^d).\ee 
By the properties of $\{v_{\en}\}$ in \eqref{eq:step2} and \eqref{eq:step2-p2}, the convergence in \eqref{eq:step2-inv-op} holds for the entire sequence $\{\bar{v}_{\en}\}$ and
$$\phi=u.$$ 

Claim \eqref{eq:step2} follows by setting 
$$\tilde{v}_{\en}(x):=\bar{v}_{\en}-\iQ\aaa(y)(u(x)+w(x,y))\,dy\quad\text{for a.e. }x\in\Omega.$$
Indeed, by \eqref{eq:step2-p3} we have
$$\tilde{v}_{\en}\wk 0\quad\text{weakly in }L^p(\Omega;\R^d),$$
and since $w\in\C_u$, by \eqref{eq:step2-properties}
$${\rm div}\,\tilde{v}_{\en}=0\quad\text{in }W^{-1,p}(\Omega;\R^l).$$
Finally, 
$$\aaa\Big(\frac{x}{\en}\Big)^{-1}\tilde{v}_{\en}(x)=\aaa\Big(\frac{x}{\en}\Big)^{-1}\bar{v}_{\en}(x)-\aaa\Big(\frac{x}{\en}\Big)^{-1}\iQ\aaa(y)(u(x)+w(x,y))\,dy$$
for a.e. $x\in\Omega$, therefore by \eqref{eq:step2-inv-op} and Riemann-Lebesgue lemma (see {e.g.} \cite{fonseca.leoni}),
$$\aaa\Big(\frac{x}{\en}\Big)^{-1}\tilde{v}_{\en}\wk u(x)-\Big(\iQ\aaa(z)^{-1}\,dz\Big)\iQ\aaa(y)(u(x)+w(x,y))\,dy$$
weakly in $L^p(\Omega;\R^d).$\\
\emph{Step 3:} We show that for every $w\in\C_u$ and every p-equiintegrable sequence $\{\tilde{v}_{\en}\}$ satisfying
\ba{eq:tvn-1}
&\qquad \tilde{v}_{\en}\wk 0\quad\text{weakly in }L^p(\Omega;\R^d),\\
\label{eq:tvn-2}&\qquad{\rm div}\, \tilde{v}_{\en}= 0\quad\text{in }W^{-1,p}(\Omega;\R^l), \\
\label{eq:tvn-3}&\qquad \aaa\Big(\frac{x}{\en}\Big)^{-1}\tilde{v}_{\en}\wk u(x)-\Big(\iQ\aaa(z)^{-1}\,dz\Big)\iQ\aaa(y)(u(x)+w(x,y))\,dy
\end{align}
weakly in $L^p(\Omega;\R^d)$, there exists a p-equiintegrable family $\{v_{\nu,n}:\,\nu\in\N,\,n\in\N\}$ such that
\ba{eq:step3-p1}&{\rm div}\,v_{\nu,n}= 0\quad\text{ in }W^{-1,p}(\Omega;\R^l),\\
\label{eq:step3-p2}&v_{\nu,n}\wk 0\quad\text{weakly in }L^p(\Omega;\R^d)\quad\text{as }n\to +\infty,\\
\label{eq:step3-p3}&\aaa\Big(\nu\floor[\Big]{\frac{1}{\nu\en}}x\Big)^{-1}v_{\nu,n}(x)\wk u-\Big(\iQ\aaa(z)^{-1}\,dz\Big)\iQ\aaa(y)(u(x)+w(x,y))\,dy
\end{align}
weakly in $L^p(\Omega;\R^d)$, as $n\to +\infty$ and
\ba{eq:step3-energy}
&\liminfn \int_{\Omega} g\Big(\frac{x}{\en},\iQ {\mathcal{A}}(y)(u(x)+w(x,y))\,dy+\tilde{v}_{\en}(x)\Big)\,dx\\
\nonumber &\quad\geq \sup_{\nu\in\N}\liminfn \int_{\Omega} g\Big(\nu \floor[\Big]{\frac{1}{\nu\en}}x,\iQ {\mathcal{A}}(y)(u(x)+w(x,y))\,dy+v_{\nu,n}(x)\Big)\,dx.
\end{align}
To prove the claim we argue as in \cite[Proposition 3.8]{fonseca.kromer}. Fix $\nu\in\N$, let
$$\theta_{\nu,n}:=\nu\en\floor[\Big]{\frac{1}{\nu\en}}\in[0,1]$$
and set $$k_{\nu,n}:=\frac{\theta_{\nu,n}}{\nu\en}\in\N_0.$$
We notice that 
\be{eq:theta-nu-n}\theta_{\nu,n}\to 1\quad \text{ as }n\to +\infty.\ee Without loss of generality, we can assume that $\Omega\subset\subset Q$. By Lemma \ref{lemma:A-free-ext} we extend every map $\tilde{v}_{\en}$ to a map $\tilde{\tilde{v}}_{\en}\in L^p(Q;\R^d)$ such that $\{\tilde{\tilde{v}}_{\en}\}$ is p-equiintegrable, and satisfies the following properties
\ba{eq:step3-properties}
&\tilde{\tilde{v}}_{\en}-\tilde{v}_{\en}\to 0\quad\text{strongly in }L^p(\Omega;\R^d),\\
\nonumber&\tilde{\tilde{v}}_{\en}\to 0\quad\text{strongly in }L^p(Q\setminus\Omega;\R^d),\\
\nonumber&{\rm div}\,\tilde{\tilde{v}}_{\en}=0\quad\text{in }W^{-1,p}(Q;\R^l).
\end{align}
In particular, by \eqref{eq:tvn-3} and \eqref{eq:step3-properties} it follows that 
\be{eq:step3-rlim}\aaa\Big(\frac{x}{\en}\Big)^{-1}\tilde{\tilde{v}}_{\en}(x)\wk u(x)-\Big(\iQ\aaa(z)^{-1}\,dz\Big)\iQ\aaa(y)(u(x)+w(x,y))\,dy\ee
weakly in $L^p(\Omega;\R^d)$.
By Proposition \ref{prop:f-k} and the definition of $\theta_{\nu,n}$ and $k_{\nu,n}$,
\ba{eq:step3-in1}
&\liminfn \int_{\Omega} g\Big(\frac{x}{\en},\iQ {\mathcal{A}}(y)(u(x)+w(x,y))\,dy+\tilde{v}_{\en}(x)\Big)\,dx\\
\nonumber &\quad\geq\liminfn \int_{\Omega} g\Big(\nu \frac{k_{\nu,n}}{\theta_{\nu,n}}x,\iQ {\mathcal{A}}(y)(u(x)+w(x,y))\,dy+\tilde{\tilde{v}}_{\en}(x)\Big)\,dx.
\end{align}
For $\Omega'\subset\subset\Omega$ fixed, there holds $\theta_{\nu,n}\Omega'\subset\Omega$ for $n$ large enough. Since $g$ is nonnegative (see \eqref{eq:prop-g}), by \eqref{eq:step3-in1} we have
\ba{eq:step3-in2}
&\liminfn \int_{\Omega} g\Big(\frac{x}{\en},\iQ {\mathcal{A}}(y)(u(x)+w(x,y))\,dy+\tilde{v}_{\en}(x)\Big)\,dx\\
\nonumber &\quad\geq\liminfn \int_{\theta_{\nu,n}\Omega'} g\Big(\nu \frac{k_{\nu,n}}{\theta_{\nu,n}}x,\iQ {\mathcal{A}}(y)(u(x)+w(x,y))\,dy+\tilde{\tilde{v}}_{\en}(x)\Big)\,dx\\
\nonumber &\quad=\liminfn(\theta_{\nu,n})^N \int_{\Omega'} g\Big(\nu k_{\nu,n}x,\iQ {\mathcal{A}}(y)(u(\theta_{\nu,n}x)+w(\theta_{\nu,n}x,y))\,dy+\tilde{\tilde{v}}_{\en}(\theta_{\nu,n}x)\Big)\,dx\\
\nonumber&\quad\geq \liminfn \int_{\Omega'} g\Big(\nu k_{\nu,n}x,\iQ {\mathcal{A}}(y)(u(x)+w(x,y))\,dy+\tilde{\tilde{v}}_{\en}(\theta_{\nu,n}x)\Big)\,dx
\end{align}
where the last inequality follows by \eqref{eq:theta-nu-n}, and since 
$$u(\theta_{\nu,n}x)+w(\theta_{\nu,n}x,y)-u(x)-w(x,y)\to 0\quad\text{strongly in }L^p(\Omega\times Q;\R^d).$$ Letting $\Omega'$ tend to $\Omega$, by the p-equiintegrability of $\{\tilde{\tilde{v}}_{\en}\}$, there holds
\ba{eq:step3-proof-energy}
&\liminfn \int_{\Omega} g\Big(\frac{x}{\en},\iQ {\mathcal{A}}(y)(u(x)+w(x,y))\,dy+\tilde{v}_{\en}(x)\Big)\,dx\\
\nonumber&\quad\geq \liminfn \int_{\Omega} g\Big(\nu k_{\nu,n}x,\iQ {\mathcal{A}}(y)(u(x)+w(x,y))\,dy+\tilde{\tilde{v}}_{\en}(\theta_{\nu,n}x)\Big)\,dx.
\end{align}
Set 
$$v_{\nu,n}(x):=\tilde{\tilde{v}}_{\en}(\theta_{\nu,n}x)\quad\text{for a.e. }x\in\Omega,\,\nu\in\N\text{ and }n\geq n_0(\nu),$$
where $n_0(\nu)$ is big enough so that $\theta_{\nu,n}\Omega\subset Q$ for $n\geq n_0(\nu)$. Inequality \eqref{eq:step3-energy} is a direct consequence of \eqref{eq:step3-proof-energy}. The p-equiintegrability of $\{v_{\nu,n}\}$, \eqref{eq:step3-p1}, and \eqref{eq:step3-p2} follow in view of \eqref{eq:theta-nu-n} and \eqref{eq:step3-properties}. 

To conclude the proof of the claim it remains to establish \eqref{eq:step3-p3}. We first remark that, by \eqref{eq:A-invertibility} and \eqref{eq:step3-p2}, the sequence $\big\{\aaa\big(\nu\floor[\big]{\frac{1}{\nu\en}}x\big)^{-1}v_{\nu,n}(x)\big\}$ is uniformly bounded in $L^p(\Omega;\R^d)$. Therefore, there exists  a map $L\in L^p(\Omega;\R^d)$ such that, up to the extraction of a (not relabeled) subsequence, there holds
 \be{eq:step3-could-be-sub}\aaa\Big(\nu\floor[\big]{\frac{1}{\nu\en}}x\Big)^{-1}v_{\nu,n}(x)\wk L(x)\quad\text{weakly in }L^p(\Omega;\R^d).\ee
 Let $\varphi\in C^{\infty}_c(\Omega;\R^d)$. Then,
 \be{eq:step3-conseq}\int_{\Omega}\aaa\Big(\nu\floor[\big]{\frac{1}{\nu\en}}x\Big)^{-1}v_{\nu,n}(x)\cdot\varphi(x)\,dx\to \int_{\Omega}L(x)\cdot\varphi(x)\,dx.\ee
For $n$ big enough, $\theta_{\nu,n}\,{\rm supp}\,\varphi\subset\Omega$. Hence,
 \ba{eq:step3-long-estimate}
 &\int_{\Omega}\aaa\Big(\nu\floor[\Big]{\frac{1}{\nu\en}}x\Big)^{-1}v_{\nu,n}(x)\cdot\varphi(x)\,dx\\
 \nn&\quad=\frac{1}{(\theta_{\nu,n})^N} \int_{\theta_{\nu,n}\,{\rm supp }\,\varphi}\aaa\Big(\frac{y}{\en}\Big)^{-1}\tilde{\tilde{v}}_{\en}(y)\cdot\varphi\Big(\frac{y}{\theta_{\nu,n}}\Big)\,dy\\
 \nn&\quad=\frac{1}{(\theta_{\nu,n})^N} \int_{\Omega}\aaa\Big(\frac{y}{\en}\Big)^{-1}\tilde{\tilde{v}}_{\en}(y)\cdot\Big(\varphi\Big(\frac{y}{\theta_{\nu,n}}\Big)-\varphi(y)\Big)\,dy\\
 \nn&\qquad+\frac{1}{(\theta_{\nu,n})^N} \int_{\Omega}\aaa\Big(\frac{y}{\en}\Big)^{-1}\tilde{\tilde{v}}_{\en}(y)\cdot\varphi(y)\,dy.
 \end{align}
 By \eqref{eq:A-invertibility} the first term in the right-hand side of \eqref{eq:step3-long-estimate} is bounded by
 \bas
 &\Big|\frac{1}{(\theta_{\nu,n})^N} \int_{\Omega}\aaa\Big(\frac{y}{\en}\Big)^{-1}\tilde{\tilde{v}}_{\en}(y)\cdot\Big(\varphi\Big(\frac{y}{\theta_{\nu,n}}\Big)-\varphi(y)\Big)\,dy\Big|\\
 &\quad\leq C\|\tilde{\tilde{v}}_{\en}\|_{L^p(\Omega;\R^d)}\|\nabla\varphi\|_{L^{\infty}(\Omega;\M^{d\times N})}\sup_{y\in\Omega}\Big|\frac{y}{\theta_{\nu,n}}-y\Big|,
 \end{align*}
 which is infinitesimal as $n\to +\infty$ due to \eqref{eq:theta-nu-n}. By \eqref{eq:theta-nu-n} and \eqref{eq:step3-rlim}, the second term in the right-hand side of \eqref{eq:step3-long-estimate} satisfies
 \bas
 &\lim_{n\to +\infty}\frac{1}{(\theta_{\nu,n})^N} \int_{\Omega}\aaa\Big(\frac{y}{\en}\Big)^{-1}\tilde{\tilde{v}}_{\en}(y)\cdot\varphi(y)\,dy\\
 &\quad=\int_{\Omega}\Big[u(x)-\iQ\aaa(z)^{-1}\,dz\iQ\aaa(y)(u(x)+w(x,y))\,dy\Big]\cdot\varphi(x)\,dx.
 \end{align*}
 Arguing by density, we conclude that 
 $$L(x)=u(x)-\Big(\iQ\aaa(z)^{-1}\,dz\Big)\iQ\aaa(y)(u(x)+w(x,y))\,dy\quad\text{for a.e. }x\in\Omega$$
 and \eqref{eq:step3-could-be-sub} holds for the entire sequence. This completes the proof of \eqref{eq:step3-p3}.\\
\emph{Step 4:} We claim that for every $w\in\C_u$, and every $p$-equiintegrable family $\{v_{\nu,n}:\,\nu\in\N,\,n\in\N\}$ satisfying \eqref{eq:step3-p1}--\eqref{eq:step3-p3}, there exists a $p$-equiintegrable family $\{w_{\nu,n}:\,\nu\in\N,\,n\in\N\}$ such that
\ba{eq:step4-p1}
&\quad {\rm div}_y\,w_{\nu,n}(x,y):=\sum_{i=1}^N\frac{\partial w_{\nu,n}}{\partial y_i}(x,y)=0\quad\text{in }W^{-1,p}(Q;\R^l)\text{ for a.e. }x\in\Omega,\\
\label{eq:step4-p2}&\quad w_{\nu,n}\wk 0\quad\text{weakly in }L^p(\Omega\times Q;\R^d),\quad\text{as }n\to +\infty,\\
\label{eq:step4-p5}&\quad\iQ w_{\nu,n}(x,y)\,dy=0\quad\text{for a.e. }x\in\Omega,\\
\label{eq:step4-p3}&\quad\aaa\Big(\floor[\Big]{\frac{1}{\nu\en}}y\Big)^{-1}w_{\nu,n}(x,y)\wk u(x)-\Big(\iQ\aaa(z)^{-1}\,dz\Big)\iQ\aaa(y)(u(x)+w(x,y))\,dy\\
\nn&\qquad\text{weakly in }L^p(\Omega;\R^d)\quad\text{as }n\to +\infty\text{ and }\nu\to +\infty,\text{ in this order},
\end{align}
and
\ba{eq:step4-p6}
&\liminfn \int_{\Omega} g\Big(\nu \floor[\Big]{\frac{1}{\nu\en}}x,\iQ {\mathcal{A}}(y)(u(x)+w(x,y))\,dy+v_{\nu,n}(x)\Big)\,dx\\
\nn&\quad\geq\liminfn\iOQ g\Big(\floor[\Big]{\frac{1}{\nu\en}}y,\iQ\aaa(z)(u(x)+w(x,z))\,dz+w_{\nu,n}(x,y)\Big)\,dy\,dx+\sigma_{\nu},
\end{align}
where $\sigma_{\nu}\to 0$ as $\nu\to +\infty$.
Let $\{v_{\nu,n}\}$ be as above. We argue similarly to \cite[Proof of Proposition 3.9]{fonseca.kromer}. We extend $u,w$, and $\{v_{\nu,n}\}$ to $0$ outside $\Omega$, and define
\bas
&Q_{\nu,z}:=\frac{1}{\nu}z+\frac{1}{\nu}Q,\quad z\in\Z^N,\\
&\Z_{\nu}:=\{z\in\Z^N:\,Q_{\nu,z}\cap\Omega\neq\emptyset\},\\
&I_{\nu,n}:=\int_{\Omega}g\Big(\nu \floor[\Big]{\frac{1}{\nu\en}} x,\iQ\aaa(y)(u(x)+w(x,y))\,dy+v_{\nu,n}(x)\Big)\,dx.
\end{align*}
By a change of variables, since $g(\cdot,0)=0$ by \eqref{eq:prop-g}, and by the periodicity of $g$ in its first variable, we obtain the following chain of equalities 
\bas
&I_{\nu,n}=\sum_{z\in\Z_{\nu}}\frac{1}{\nu^N}\iQ g\Big(\nu\floor[\Big]{\frac{1}{\nu\en}}\Big(\frac{z}{\nu}+\frac{\eta}{\nu}\Big),\iQ\aaa(y)\Big(u\Big(\frac{z}{\nu}+\frac{\eta}{\nu}\Big)+w\Big(\frac{z}{\nu}+\frac{\eta}{\nu},y\Big)\Big)\,dy\\
&\qquad+v_{\nu,n}\Big(\frac{z}{\nu}+\frac{\eta}{\nu}\Big)\Big)\,d\eta\\
&\quad=\sum_{z\in\Z_{\nu}}\int_{Q_{\nu,z}}\iQ g\Big(\floor[\Big]{\frac{1}{\nu\en}}\eta,\iQ\aaa(y)\Big(u\Big(\frac{\floor{\nu x}}{\nu}+\frac{\eta}{\nu}\Big)+w\Big(\frac{\floor{\nu x}}{\nu}+\frac{\eta}{\nu},y\Big)\Big)\,dy\\
&\qquad+v_{\nu,n}\Big(\frac{\floor{\nu x}}{\nu}+\frac{\eta}{\nu}\Big)\Big)\,d\eta\,dx\\
&\quad=\sum_{z\in\Z_{\nu}}\int_{Q_{\nu,z}}\iQ g\Big(\floor[\Big]{\frac{1}{\nu\en}}\eta,\iQ\aaa(y)(u(x)+w(x,y))\,dy+T_{\frac{1}{\nu}}v_{\nu,n}(x,\eta)\Big)\,d\eta\,dx+\sigma_{\nu}
\end{align*}
where $T_{\frac{1}{\nu}}$ is the unfolding operator defined in \eqref{eq:unfolding-operator}, and 
\bas
&\sigma_{\nu}:=\sum_{z\in\Z_{\nu}}\int_{Q_{\nu,z}}\iQ \Bigg\{g\Big(\floor[\Big]{\frac{1}{\nu\en}}\eta,\iQ\aaa(y)\Big(u\Big(\frac{\floor{\nu x}}{\nu}+\frac{\eta}{\nu}\Big)+w\Big(\frac{\floor{\nu x}}{\nu}+\frac{\eta}{\nu},y\Big)\Big)\,dy\\
&\qquad+T_{\frac{1}{\nu}}v_{\nu,n}(x,\eta)\Big)-g\Big(\floor[\Big]{\frac{1}{\nu\en}}\eta,\iQ\aaa(y)(u(x)+w(x,y))\,dy+T_{\frac{1}{\nu}}v_{\nu,n}(x,\eta)\Big)\Bigg\}\,d\eta\,dx.
\end{align*}
By Proposition \ref{prop:conv-unf-op},
$$\Big\|\iQ\aaa(y)(u(x)+w(x,y))\,dy-\iQ\aaa(y)\Big(T_{\frac{1}{\nu}}u(x,\eta)+T_{\frac{1}{\nu}}w((x,\eta),y)\Big)\,dy\Big\|_{L^p(\Omega\times Q;\R^d)}\to 0$$
as $\nu\to +\infty$. Moreover, by Proposition \ref{prop:conv-unf-op}, the sequence
$$\Big\{\iQ\aaa(y)(u(x)+w(x,y))\,dy-\iQ\aaa(y)\Big(T_{\frac{1}{\nu}}u(x,\eta)+T_{\frac{1}{\nu}}w((x,\eta),y)\Big)\,dy\Big\}$$
is p-equiintegrable and
\bas
&\Big\|\iQ\aaa(y)(u(x)+w(x,y))\,dy\\
\quad&-\iQ\aaa(y)\Big(T_{\frac{1}{\nu}}u(x,\eta)+T_{\frac{1}{\nu}}w((x,\eta),y)\Big)\,dy\Big\|_{L^p\Big(\Big(\cup_{z\in\Z_{\nu}}Q_{\nu,z}\setminus\Omega\Big)\times Q;\R^d\Big)}\to 0
\end{align*}
as $\nu\to +\infty$.
Hence, by Proposition \ref{prop:to-add1}, $\sigma_{\nu}=\sigma_{\nu}(\Omega)\to 0$ as $\nu\to +\infty$.

 We set
$$\hat{v}_{\nu,z,n}(y):=T_{\frac{1}{\nu}}v_{\nu,n}\Big(\frac{z}{\nu},y\Big)\quad\text{for a.e. }y\in Q,\text{ for every }z\in\Z_{\nu}.$$
For fixed $\nu$ and $z\in\Z_{\nu}$, the sequence $\{\hat{v}_{\nu,z,n}\}$ is p-equiintegrable. Moreover,
$${\rm div}_y\,\hat{v}_{\nu,z,n}=0\quad\text{in }W^{-1,p}(Q;\R^l)$$
and
$$\hat{v}_{\nu,z,n}\wk 0\quad\text{weakly in }L^p(Q;\R^d),\text{ as }n\to +\infty.$$
Setting $w_{\nu,z,n}:=\hat{v}_{\nu,z,n}-\iQ\hat{v}_{\nu,z,n}(y)\,dy$, the sequence $\{w_{\nu,z,n}\}\subset L^p(Q;\R^d)$ is p-equiintegrable and such that
\bas
\nn&{\rm div}_y\,w_{\nu,z,n}=0\quad\text{in }W^{-1,p}(Q;\R^l),\\
\nn&w_{\nu,z,n}\wk 0\quad\text{weakly in }L^p(Q;\R^d),\text{ as }n\to +\infty\\
\nn&\iQ w_{\nu,z,n}(y)\,dy=0,
\end{align*}
and
\be{eq:difference-v-w}w_{\nu,z,n}-\hat{v}_{\nu,z,n}\to 0\quad\text{strongly in }L^p(Q;\R^d),\text{ as }n\to +\infty.\ee
By applying again Proposition \ref{prop:to-add1} we obtain
\bas
&\int_{Q_{\nu,z}}\iQ g\Big(\floor[\Big]{\frac{1}{\nu\en}}y,\iQ\aaa(\xi)(u(x)+w(x,\xi))\,d\xi+T_{\frac{1}{\nu}}v_{\nu,n}\Big(\frac{z}{\nu},y\Big)\Big)\,dy\,dx\\
&\quad\geq \int_{Q_{\nu,z}}\iQ g\Big(\floor[\Big]{\frac{1}{\nu\en}}y,\iQ\aaa(\xi)(u(x)+w(x,\xi))\,d\xi+w_{\nu,z,n}(y)\Big)\,dy\,dx+\tau_{z,\nu,n}
\end{align*}
with $$\tau_{z,\nu,n}\to 0\quad\text{as }n\to +\infty.$$
Therefore, by \eqref{eq:prop-g} and since $$\Omega\subset \cup_{z\in\Z^N}Q_{\nu,z},$$
 we deduce \eqref{eq:step4-p6} with
$$w_{\nu,n}(x,y):=\sum_{z\in\Z_{\nu}}\chi_{Q_{\nu,z}\cap\Omega}(x)w_{\nu,z,n}(y)\quad\text{for a.e. }x\in\Omega,\,y\in Q.$$
We observe that for $\nu$ fixed only a finite number of terms in the sum above are different from zero, hence properties \eqref{eq:step4-p1}, \eqref{eq:step4-p2} and \eqref{eq:step4-p5} follow immediately.

To prove \eqref{eq:step4-p3}, we notice that the sequence $\Big\{\aaa\Big(\floor[\Big]{\frac{1}{\nu\en}}y\Big)^{-1}w_{\nu,n}(x,y)\Big\}$ is uniformly bounded in $L^p(\Omega\times Q;\R^d)$ by \eqref{eq:A-invertibility} and \eqref{eq:step4-p2}, therefore it is enough to work with a convergent subsequence and check that the limit is uniquely determined. Fix $\varphi\in C^{\infty}_c(\Omega; C^{\infty}_{\rm per}(Q;\R^d))$ and set
$$\psi(x):=u(x)-\Big(\iQ\aaa(z)^{-1}\,dz\Big)\iQ\aaa(y)(u(x)+w(x,y))\,dy$$
for a.e. $x\in\Omega$, $\psi:=0$ outside $\Omega$, and
$$\Z_{\varphi}^{\nu}:=\{z\in\Z^N:\,(Q_{\nu,z}\times Q)\cap\,{\rm supp}\,\varphi\neq\emptyset\}.$$
Then
\bas
&\iOQ \aaa\Big(\floor[\Big]{\frac{1}{\nu\en}}y\Big)^{-1}w_{\nu,n}(x,y)\cdot\varphi(x,y)\,dy\,dx\\
&\quad=\sum_{z\in\Z_{\varphi}^{\nu}}\int_{(Q_{\nu,z}\times Q)\cap\,{\rm supp}\,\varphi}\aaa\Big(\floor[\Big]{\frac{1}{\nu\en}}y\Big)^{-1}w_{\nu,z,n}(y)\cdot\varphi(x,y)\,dy\,dx\\
&\quad= \sum_{z\in\Z_{\varphi}^{\nu}}\iQ\int_{Q_{\nu,z}}\aaa\Big(\floor[\Big]{\frac{1}{\nu\en}}y\Big)^{-1}(w_{\nu,z,n}(y)-\hat{v}_{\nu,z,n}(y))\cdot\varphi(x,y)\,dy\,dx\\
&\qquad+\sum_{z\in\Z_{\varphi}^{\nu}}\iQ\int_{Q_{\nu,z}}\aaa\Big(\floor[\Big]{\frac{1}{\nu\en}}y\Big)^{-1}\hat{v}_{\nu,z,n}(y)\cdot\varphi(x,y)\,dy\,dx.
\end{align*}
By \eqref{eq:A-invertibility}, we have
\bas
&\Big|\sum_{z\in\Z_{\varphi}^{\nu}}\iQ\int_{Q_{\nu,z}}\aaa\Big(\floor[\Big]{\frac{1}{\nu\en}}y\Big)^{-1}(w_{\nu,z,n}(y)-\hat{v}_{\nu,z,n}(y))\cdot\varphi(x,y)\,dy\,dx\Big|\\
&\quad\leq C\sum_{z\in\Z_{\varphi}^{\nu}}\|w_{\nu,z,n}(y)-\hat{v}_{\nu,z,n}(y)\|_{L^p(Q;\R^N)}\|\varphi\|_{L^{\infty}(\Omega\times Q;\R^d)},
\end{align*}
which by \eqref{eq:difference-v-w} converges to zero as $n\to +\infty$ (here we used the fact that the previous series is actually a finite sum for every $\nu\in\N$ fixed). On the other hand, 
\ba{eq:step4-comput1}
&\sum_{z\in\Z_{\varphi}^{\nu}}\iQ\int_{Q_{\nu,z}}\aaa\Big(\floor[\Big]{\frac{1}{\nu\en}}y\Big)^{-1}\hat{v}_{\nu,z,n}(y)\cdot\varphi(x,y)\,dy\,dx\\
\nn&\quad=\sum_{z\in\Z_{\varphi}^{\nu}}\iQ\int_{Q_{\nu,z}}\Big[\aaa\Big(\floor[\Big]{\frac{1}{\nu\en}}y\Big)^{-1}T_{\frac{1}{\nu}}v_{\nu,n}\Big(\frac{z}{\nu},y\Big)-T_{\frac{1}{\nu}}\psi\Big(\frac{z}{\nu},y\Big)\Big]\cdot\varphi(x,y)\,dy\,dx\\
\nn&\qquad+\sum_{z\in\Z_{\varphi}^{\nu}}\iQ\int_{Q_{\nu,z}}T_{\frac{1}{\nu}}\psi\Big(\frac{z}{\nu},y\Big)\cdot\varphi(x,y)\,dy\,dx.
\end{align}
By the periodicity of $\aaa$, the first term in the right-hand side of \eqref{eq:step4-comput1} satisfies 
\bas
&\lim_{n\to +\infty}\sum_{z\in\Z_{\varphi}^{\nu}}\iQ\int_{Q_{\nu,z}}\Big[\aaa\Big(\floor[\Big]{\frac{1}{\nu\en}}y\Big)^{-1}T_{\frac{1}{\nu}}v_{\nu,n}\Big(\frac{z}{\nu},y\Big)-T_{\frac{1}{\nu}}\psi\Big(\frac{z}{\nu},y\Big)\Big]\cdot\varphi(x,y)\,dy\,dx\\
&\quad =\lim_{n\to +\infty}\Big(\frac{1}{\nu}\Big)^N\sum_{z\in\Z_{\varphi}^{\nu}}\iQ\iQ\Big[\aaa\Big(\floor[\Big]{\frac{1}{\nu\en}}y\Big)^{-1}v_{\nu,n}\Big(\frac{z}{\nu}+\frac{y}{\nu}\Big)-\psi\Big(\frac{z}{\nu}+\frac{y}{\nu}\Big)\Big]\cdot\varphi\Big(\frac{z}{\nu}+\frac{\eta}{\nu},y\Big)\,dy\,d\eta\\
&\quad =\lim_{n\to +\infty}\sum_{z\in\Z_{\varphi}^{\nu}}\int_{Q_{\nu,z}}\iQ\Big[\aaa\Big(\floor[\Big]{\frac{1}{\nu\en}}\nu x\Big)^{-1}v_{\nu,n}(x)-\psi(x)\Big]\cdot\varphi\Big(\frac{1}{\nu}\floor{\nu x}+\frac{1}{\nu}\eta,\nu x\Big)\,d\eta\,dx\\
&\quad =\lim_{n\to +\infty}\int_{\cup_{z\in\Z_{\varphi}^{\nu}}Q_{\nu,z}}\Big[\aaa\Big(\floor[\Big]{\frac{1}{\nu\en}}\nu x\Big)^{-1}v_{\nu,n}(x)-\psi(x)\Big]\cdot\Bigg(\iQ\varphi\Big(\frac{1}{\nu}\floor{\nu x}+\frac{1}{\nu}\eta,\nu x\Big)\,d\eta\Bigg)\,dx.
\end{align*}
By \eqref{eq:step3-p3}, and recalling that $\psi$ and $\{v_{\nu,n}\}$ have been extended to $0$ outside $\Omega$, we conclude that 
\bas
&\lim_{n\to +\infty} \int_{\Omega}\iQ \aaa\Big(\floor[\Big]{\frac{1}{\nu\en}}y\Big)^{-1}w_{\nu,n}(x,y)\cdot\varphi(x,y)\,dy\,dx\\
&\quad=\sum_{z\in\Z_{\nu}} \iQ \int_{Q_{\nu,z}}T_{\frac{1}{\nu}}\psi\Big(\frac{z}{\nu},y\Big)\cdot\varphi(x,y)\,dx\,dy,
\end{align*}
and so
$$\aaa\Big(\floor[\Big]{\frac{1}{\nu\en}}y\Big)^{-1}w_{\nu,n}(x,y)\wk \sum_{z\in\Z_{\nu}}\chi_{Q_{\nu,z}\cap\Omega}(x)T_{\frac{1}{\nu}}\psi\Big(\frac{z}{\nu},y\Big)$$
weakly in $L^p(\Omega\times Q;\R^d)$, as $n\to +\infty$. Finally, we claim that
\be{eq:step4-last-conv}\sum_{z\in\Z_{\nu}}\chi_{Q_{\nu,z}\cap\Omega}(x)T_{\frac{1}{\nu}}\psi\Big(\frac{z}{\nu},y\Big)\to \psi(x)\ee
strongly in $L^p(\Omega;\R^d)$ as $\nu\to +\infty$. 

Indeed, let $\varphi\in L^{p'}(\Omega\times Q;\R^d)$. Then by Holder's inequality
\bas
&\Big|\iOQ \Big(\sum_{z\in\Z_{\nu}}\chi_{Q_{\nu,z}\cap\Omega}(x)T_{\frac{1}{\nu}}\psi\Big(\frac{z}{\nu},y\Big)-\psi(x)\Big)\cdot\varphi(x,y)\,dy\,dx\Big|\\
&\quad=\Big|\sum_{z\in\Z_{\nu}}\int_{Q_{\nu,z}\cap\Omega}\iQ\Big(T_{\frac{1}{\nu}}\psi\Big(\frac{z}{\nu},y\Big)-\psi(x)\Big)\cdot\varphi(x,y)\,dy\,dx\Big|\\
&\quad=\Big|\sum_{z\in\Z_{\nu}}\int_{Q_{\nu,z}\cap \Omega}\iQ\Big(\psi\Big(\frac{z}{\nu}+\frac{y}{\nu}\Big)-\psi(x)\Big)\cdot\varphi(x,y)\,dy\,dx\Big|\\
&\quad=\Big|\sum_{z\in\Z_{\nu}}\int_{Q_{\nu,z}\cap \Omega}\iQ\Big(\psi\Big(\frac{1}{\nu}\floor{\nu x}+\frac{y}{\nu}\Big)-\psi(x)\Big)\cdot\varphi(x,y)\,dy\,dx\Big|\\
&\quad=\Big|\iOQ (T_{\frac{1}{\nu}}\psi(x,y)-\psi(x))\cdot\varphi(x,y)\,dy\,dx\Big|\\
&\quad\leq \|T_{\frac{1}{\nu}}\psi(x,y)-\psi(x)\|_{L^p(\Omega\times Q;\R^d)}\|\varphi\|_{L^{p'}(\Omega\times Q;\R^d)}.
\end{align*}
Property \eqref{eq:step4-last-conv}, and thus \eqref{eq:step4-p3}, follow in view of Proposition \ref{prop:conv-unf-op}.\\
\noindent\emph{Step 5:} by Steps 1--4 it follows that
\ba{eq:step5}
&\inf\Big\{\liminf_{\ep\to 0} \int_{\Omega}f(u_{\ep}(x))\,dx:\, u_{\ep}\wk u\quad\text{weakly in }L^p(\Omega;\R^d)\\
\nn&\qquad\text{and }\pdeep{\ep}u_{\ep}\to 0\quad\text{strongly in }W^{-1,p}(\Omega;\R^l)\Big\}\\
\nn&\quad\geq\inf_{w\in C_u}\inf\Big\{\liminf_{\nu\to +\infty}\liminfn\iOQ g\Big(\floor[\Big]{\frac{1}{\nu\en}}y,\iQ\aaa(z)(u(x)+w(x,z))\,dz+w_{\nu,n}(x,y)\Big)\,dy\,dx:\\
\nn&\qquad {\rm div}_y\,w_{\nu,n}(x,y)=0\quad\text{in }W^{-1,p}(Q;\R^l)\text{ for a.e. }x\in\Omega,\\
\nn&\qquad w_{\nu,n}\wk 0\quad\text{weakly in }L^p(\Omega\times Q;\R^d) \text{ as }n\to +\infty,\\
\nn&\qquad\aaa\Big(\floor[\Big]{\frac{1}{\nu \en}}y\Big)^{-1}w_{\nu,n}(x,y)\wk u(x)-\Big(\iQ\aaa(z)^{-1}\,dz\Big)\iQ\aaa(y)(u(x)+w(x,y))\,dy\\
\nn&\qquad\quad\text{weakly in }L^p(\Omega;\R^d)\text{ as }n\to +\infty\text{ and }\nu\to +\infty,\, \text{and }\\
\nn&\qquad\iQ w_{\nu,n}(x,y)\,dy=0\Big\}.
\end{align}
By a diagonalization argument, given $\{w_{\nu, n}\}$ as above we can construct $\{n(\nu)\}$ such that, setting
$$\ep_{\nu}:=\ep_{n(\nu)}\,\quad w_{\nu}(x,y):=w_{\nu, n(\nu)},$$
we obtain the following inequality
\ba{eq:step5-better}
&\inf\Big\{\liminf_{\e\to 0} \int_{\Omega}f(u_{\ep}(x))\,dx:\, u_{\ep}\wk u\quad\text{weakly in }L^p(\Omega;\R^d)\\
\nn&\qquad\text{and }\pdeep{\ep}u_{\e}\to 0\quad\text{strongly in }W^{-1,p}(\Omega;\R^l)\Big\}\\
\nn&\quad\geq\inf_{w\in C_u}\inf\Big\{\liminf_{\nu\to +\infty}\iOQ g\Big(\floor[\Big]{\frac{1}{\nu \ep_{\nu}}}y,\iQ\aaa(z)(u(x)+w(x,z))\,dz+w_{\nu}(x,y)\Big)\,dy\,dx:\\
\nn&\qquad {\rm div}_y\,w_{\nu}(x,y)=0\quad\text{in }W^{-1,p}(Q;\R^l)\text{ for a.e. }x\in\Omega,\\
\nn&\qquad w_{\nu}\wk 0\quad\text{weakly in }L^p(\Omega\times Q;\R^d) \text{ as }\nu\to +\infty,\\
\nn&\qquad\aaa\Big(\floor[\Big]{\frac{1}{\nu\ep_{\nu}}}y\Big)^{-1}w_{\nu}(x,y)\wk u(x)-\Big(\iQ\aaa(z)^{-1}\,dz\Big)\iQ\aaa(y)(u(x)+w(x,y))\,dy\\
\nn&\qquad\quad\text{weakly in }L^p(\Omega;\R^d)\text{ as }\nu\to +\infty,\,\text{and }\iQ w_{\nu}(x,y)\,dy=0\Big\}.
\end{align}
Associating to every sequence $\{w_{\nu}\}$ as in \eqref{eq:step5-better} the maps
$$\phi_{\nu}(x,y):=\iQ\aaa(z)(u(x)+w(x,z))\,dz+w_{\nu}(x,y)\quad\text{for a.e. }x\in\Omega\text{ and }y\in Q,$$ 
inequality \eqref{eq:step5-better} can be rewritten as
\ba{eq:step5-better2}
&\inf\Big\{\liminf_{\ep\to 0} \int_{\Omega}f(u_{\e}(x))\,dx:\, u_{\e}\wk u\quad\text{weakly in }L^p(\Omega;\R^d)\\
\nn&\qquad\text{and }\pdeep{\e}u_{\e}\to 0\quad\text{strongly in }W^{-1,p}(\Omega;\R^l)\Big\}\\
\nn&\quad\geq\inf_{w\in C_u}\inf\Big\{\liminf_{\nu\to +\infty}\iOQ f\Big(\aaa\Big(\floor[\Big]{\frac{1}{\nu\ep_{\nu}}}y\Big)^{-1}\phi_{\nu}(x,y)\Big)\,dy\,dx:\\
\nn&\qquad {\rm div}_y\,\phi_{\nu}(x,y)=0\quad\text{in }W^{-1,p}(Q;\R^l)\text{ for a.e. }x\in\Omega,\\
\nn&\qquad \phi_{\nu}\wk \iQ\aaa(z)(u(x)+w(x,z))\,dz\quad\text{weakly in }L^p(\Omega\times Q;\R^d) \text{ as }\nu\to +\infty,\,\\
\nn&\qquad\aaa\Big(\floor[\Big]{\frac{1}{\nu\ep_{\nu}}}y\Big)^{-1}\phi_{\nu}(x,y)\wk u(x)\quad\text{weakly in }L^p(\Omega;\R^d)\quad\text{as }\nu\to +\infty,\,\text{and }\\
\nn&\qquad\iQ \phi_{\nu}(x,z)\,dz=\iQ\aaa(z)(u(x)+w(x,z))\,dz\Big\}.
\end{align}
Finally, for $\{\phi_{\nu}\}$ as above, considering the maps 
$$v_{\nu}(x,y):=\aaa\Big(\floor[\Big]{\frac{1}{\nu\en}}y\Big)^{-1}\phi_{\nu}(x,y)\quad\text{for a.e. }x\in\Omega\text{ and }y\in Q,$$ 
we deduce that
\ba{eq:step5-better3}
&\inf\Big\{\liminf_{\ep\to 0} \int_{\Omega}f(u_{\e}(x))\,dx:\, u_{\e}\wk u\quad\text{weakly in }L^p(\Omega;\R^d)\\
\nn&\qquad\text{and }\pdeep{\e}u_{\e}\to 0\quad\text{strongly in }W^{-1,p}(\Omega;\R^l)\Big\}\\
\nn&\quad\geq\inf_{w\in \C_u}\inf\Big\{\liminf_{\nu\to +\infty}\iOQ f(v_{\nu}(x,y))\,dy\,dx:\\
\nn&\qquad {\rm div}_y\,\Big(\aaa\Big(\floor[\Big]{\frac{1}{\nu\ep_{\nu}}}y\Big)v_{\nu}(x,y)\Big)=0\quad\text{in }W^{-1,p}(Q;\R^l)\text{ for a.e. }x\in\Omega,\\
\nn&\qquad \aaa\Big(\floor[\Big]{\frac{1}{\nu\ep_{\nu}}}y\Big)v_{\nu}(x,y)\wk \iQ\aaa(z)(u(x)+w(x,z))\,dz\quad\text{weakly in }L^p(\Omega\times Q;\R^d)\text{ as }\nu\to +\infty,\\
\nn&\qquad v_{\nu}(x,y)\wk u(x)\quad\text{weakly in }L^p(\Omega\times Q;\R^d)\text{ as }\nu\to +\infty,\,\text{and }\\
\nn&\qquad\iQ \aaa\Big(\floor[\Big]{\frac{1}{\nu\ep_{\nu}}}z\Big)v_{\nu}(x,z)\,dz=\iQ\aaa(z)(u(x)+w(x,z))\,dz\Big\}\\
\nn&\quad\geq \inf\Big\{\liminf_{\nu\to +\infty}\iOQ f(u(x)+w_{\nu}(x,y))\,dy\,dx:\\
\nn&\qquad {\rm div}_y\,\Big(\aaa\Big(\floor[\Big]{\frac{1}{\nu\ep_{\nu}}}y\Big)(u(x)+w_{\nu}(x,y))\Big)=0\quad\text{in }W^{-1,p}(Q;\R^l)\quad\text{for a.e. }x\in\Omega,\\
\nn&\qquad {\rm div}_x\iQ\Big(\aaa\Big(\floor[\Big]{\frac{1}{\nu\ep_{\nu}}}y\Big)(u(x)+w_{\nu}(x,y)\Big)\,dy=0\quad\text{in }W^{-1,p}(\Omega;\R^l),\,\text{and }\\
\nn&\qquad w_{\nu}(x,y)\wk 0\quad\text{weakly in }L^p(\Omega\times Q;\R^d)\Big\}.
\end{align}
By \eqref{eq:step5-better3} it follows, in particular, that
\ba{eq:step5-better4}
&\inf\Big\{\liminf_{\e\to 0} \int_{\Omega}f(u_{\e}(x))\,dx:\, u_{\e}\wk u\quad\text{weakly in }L^p(\Omega;\R^d)\\
\nn&\qquad\text{and }\pdeep{\e}u_{\e}\to 0\quad\text{strongly in }W^{-1,p}(\Omega;\R^l)\Big\}\\
\nn&\quad\geq \inf\Big\{\liminf_{\nu\to +\infty}\iOQ f(u_{\nu}(x)+w_{\nu}(x,y))\,dy\,dx:\\
\nn&\qquad u_{\nu}\wk u\quad\text{weakly in }L^p(\Omega;\R^d),\,w_{\nu}\in \mathcal{C}^{\pdeor\Big(\floor[\Big]{\frac{1}{\nu\ep_{\nu}}}\cdot\Big)}_{u^{\nu}},\\
\nn&\qquad w_{{\nu}}(x,y)\wk 0\quad\text{weakly in }L^p(\Omega\times Q;\R^d)\Big\}.
\end{align}
Fix $\{u_{\nu}\}$ and $\{w_{\nu}\}$ as in \eqref{eq:step5-better4}. Then there exists a constant $r$ such that 
\be{eq:w-truncated}\sup_{\nu\in\N}\|w_{\nu}\|_{L^p(\Omega\times Q;\R^d)}\leq r.\ee
Therefore, setting
$$n_{\nu}:=\floor[\Bigg]{\frac{1}{\nu\ep_{\nu}}},$$
and
$$u_n:=\begin{cases}u_{\nu}&\text{if }n=n_{\nu},\\u&\text{otherwise},\end{cases}$$
we have
\ba{eq:step4-better5bis}
&\liminf_{\nu\to +\infty}\iOQ f(u_{\nu}(x)+w_{\nu}(x,y))\,dy\,dx\geq \liminf_{\nu\to+\infty}\mathcal{F}_{\pdeor\Big(\floor[\Big]{\frac{1}{\nu\ep_{\nu}}}\cdot\Big)}^{\, r}(u_{\nu})\\
\nn&\quad= \liminf_{\nu\to+\infty}\overline{\mathcal{F}}_{\pdeor\Big(\floor[\Big]{\frac{1}{\nu\ep_{\nu}}}\cdot\Big)}^{\, r}(u_{\nu})= \liminf_{\nu\to+\infty}\overline{\mathcal{F}}_{\pdeor(n_{\nu}\cdot)}^{\, r}(u_{n_\nu})\geq \liminfn\FBN^{\, r}(u_{n})\\
\nn&\quad\geq\inf\Big\{\liminfn \FBN^{\, r}(u_n):\,u_n\wk u\quad\text{weakly in }L^p(\Omega;\R^d)\Big\}.
\end{align}
The thesis follows by taking the infimum with respect to $r$ in the right-hand side of \eqref{eq:step4-better5bis} and by invoking \eqref{eq:step5-better4}.
\end{proof}
\begin{remark}
{We point out that the truncation by $r$ in \eqref{eq:w-truncated} and \eqref{eq:step4-better5bis} will be used in a fundamental way. Infact it guarantees that the sequences constructed in the proof of the limsup inequality (see Proposition \ref{thm:limsup}) are uniformly bounded in $L^p$, and hence it allows us to apply Attouch's diagonalization lemma (see \cite[Lemma 1.15 and Corollary 1.16]{attouch}) in Step 3 of the proof of Proposition \ref{thm:limsup}.}
\end{remark}
\begin{remark}
\label{rk:even-better}
In the case in which $A^i\in L^{\infty}(\R^N;\M^{l\times d})$, $i=1,\cdots, N$, the previous proof yields the inequality
\bas
&\inf\Big\{\liminf_{\e\to 0} \int_{\Omega}f(u_{\e}(x))\,dx:\,u_{\e}\wk u\quad\text{ weakly in }L^p(\Omega;\rd)\\
&\qquad\text{and }\pdeep{\e}u_{\e}\to 0\text{ strongly in  }W^{-1,q}(\Omega;\rl)\quad\text{for every }1\leq q<p\Big\}\\
&\quad\geq \F(u).
\end{align*}
To see that, arguing as in Step 1 we get
\bas
\inf&\Big\{\liminf_{\ep\to 0} \int_{\Omega}f(u_{\e}(x))\,dx:\, u_{\e}\wk u\quad\text{weakly in }L^p(\Omega;\R^d)\\
\nn&\qquad\text{and }\pdeep{\e}u_{\e}\to 0\quad\text{strongly in }W^{-1,q}(\Omega;\R^l)\quad\text{for every }1\leq q<p\Big\}\\
\nn&\geq\quad \inf_{w\in\C_u}\inf \Bigg\{\liminfn \int_{\Omega} g\Big(\frac{x}{\en},v_{\en}(x)\Big)\,dx:\\
\nn&\qquad v_{\en}\wk \iQ {\mathcal{A}}(y)(u(x)+w(x,y))\,dy\quad\text{weakly in }L^p(\Omega;\R^d),\\
\nn&\qquad {\rm div}\, v_{\en}\to 0\quad\text{strongly in }W^{-1,q}(\Omega;\R^l)\quad\text{for every }1\leq q<p,\\
\nn&\qquad\text{and }\aaa\Big(\frac{x}{\en}\Big)^{-1}v_{\en}\wk u\quad\text{weakly in }L^p(\Omega;\R^d)\Bigg\}.
\end{align*}
By Lemma \ref{lemma:cor-f-m} and Remark \ref{rk:add-for-nonsmooth}, inequality \eqref{eq:step2} is replaced by
\bas
&\inf_{w\in\C_u}\inf \Bigg\{\liminfn \int_{\Omega} g\Big(\frac{x}{\en},v_{\en}(x)\Big)\,dx:\\
\nn&\qquad v_{\en}\wk\iQ \iQ {\mathcal{A}}(y)(u(x)+w(x,y))\,dy\quad\text{weakly in }L^p(\Omega;\R^d),\\
\nn&\qquad {\rm div}\, v_{\en}\to 0\quad\text{strongly in }W^{-1,q}(\Omega;\R^l)\quad\text{for every }1\leq q<p,\\
\nn&\qquad\text{and }\aaa\Big(\frac{x}{\en}\Big)^{-1}v_{\en}\wk u\quad\text{weakly in }L^p(\Omega;\R^d)\Bigg\}\\
\nn&\geq\quad \inf_{w\in\C_u}\inf \Bigg\{\liminfn \int_{\Omega} g\Big(\frac{x}{\en},\iQ {\mathcal{A}}(y)(u(x)+w(x,y))\,dy+\tilde{v}_{\en}(x)\Big)\,dx:\\
\nn&\qquad \{\tilde{v}_{\en}\}\text{ is p-equiintegrable},\, \tilde{v}_{\en}\wk 0\quad\text{weakly in }L^p(\Omega;\R^d),\\
\nn&\qquad{\rm div}\, \tilde{v}_{\en}= 0\text{ in }W^{-1,p}(\Omega;\R^l),\, \text{and} \\
\nn&\qquad \aaa\Big(\frac{x}{\en}\Big)^{-1}\tilde{v}_{\en}\wk u(x)-\Big(\iQ\aaa(z)^{-1}\,dz\Big)\iQ\aaa(y)^{-1}(u(x)+w(x,y))\,dy\\
\nn&\qquad\text{weakly in }L^p(\Omega;\R^d)\Bigg\}.
\end{align*}
The result now follows by arguing exactly as in the proof of Proposition \ref{thm:liminf}.
\end{remark}
We finally prove the limsup inequality in Theorem \ref{thm:main-homogenized-pde}.
\begin{proposition}
\label{thm:limsup}
 Under the assumptions of Theorem \ref{thm:main-homogenized-pde}, for every $u\in\C$ there exists a sequence $\{u_{\ep}\}\subset L^p(\Omega;\rd)$ such that
\begin{eqnarray}
\label{eq:wk-limsup}
&&u_{\ep}\wk u\quad\text{weakly in }L^p(\Omega;\rd),\\
\label{eq:pde-limsup}
&&\pdeep{\ep} u_{\ep}\to 0\quad\text{strongly in }W^{-1,p}(\Omega;\rl),\\
\label{eq:energy-limsup}
&&\limsup_{\ep\to 0}\iO f(u_{\ep}(x))\,dx\leq\F(u).
\end{eqnarray}
\end{proposition}
\begin{proof}
We subdivide the proof into three steps.\\
\emph{Step 1}: Fix $n\in\N$. We first show that for every $u\in {\mathcal{C}}^{\pdeor(n\cdot)}\cap C^1(\Omega;\rd)$ and $w\in {\mathcal{C}}_u^{\pdeor(n\cdot)}\cap C^1(\Omega;C^1_{\rm per}(\Rn;\rd))$ there exists a sequence \mbox{$\{u_{\e}\}\subset L^p(\Omega;\rd)$} and a constant $C$ independent of $n$ and $\ep$ such that
\begin{align}
&\label{eq:wk-c1}u_{\e}\sts u+w\quad\text{strongly two-scale in }L^p(\Omega\times Q;\rd),\\
&\label{eq:pde-c1-renamed}\pdeep{\e} u_{\e}\to 0\quad\text{strongly in }W^{-1,p}(\Omega;\rl),\\
&\label{eq:energy-c1}\iO f(u_{\e}(x))\,dx\to\iOQ f(u(x)+w(x,y))\,dy\,dx,
\end{align}
as $\e\to 0$, and
\be{eq:bd-c1}\sup_{\e>0}\|u_{\e}\|_{L^p(\Omega;\R^d)}\leq C(\|u\|_{L^p(\Omega;\R^d)}+\|w\|_{L^p(\Omega\times Q;\R^d)}).\ee

Define
$$u_{\e}(x):=u(x)+w\Big(x,\frac{x}{n\e}\Big)\quad\text{for a.e. }x\in\Omega.$$
By Proposition \ref{prop:simple-2-scale} we have  
\begin{eqnarray*}
&&u_{\e}\wk \iQ (u(x)+w(x,y))\,dy=u(x)\quad\text{weakly in }L^p(\Omega;\rd),\\
&&u_{\e}\sts u+w\quad\text{strongly two-scale in }L^p(\Omega\times Q;\rd),\\
&&f(u_{\e})\wk \iQ f(u(x)+w(x,y))\,dy\quad\text{weakly in }L^1(\Omega),
\end{eqnarray*}
as $\ep\to 0$. In particular, we obtain immediately \eqref{eq:wk-c1} and \eqref{eq:energy-c1}. Property \eqref{eq:bd-c1} follows by \eqref{eq:wk-c1}, Proposition \ref{prop:isometry} and Theorem \ref{thm:equivalent-two-scale}. To prove \eqref{eq:pde-c1-renamed}, we notice that
\ba{eq:equality}
\pdeep{\e}u_{\e}&=\sum_{i=1}^N A^i\Big(\frac{x}{\e}\Big)\Big(\frac{\partial u(x)}{\partial x_i}+\frac{\partial w}{\partial {x_i}}\Big(x,\frac{x}{n\e}\Big)\Big)\\
\nn&\qquad+\frac{1}{n\e}\Big(n\frac{\partial A^i}{\partial y_i}\Big(\frac{x}{\e}\Big)\Big(u(x)+w\Big(x,\frac{x}{n\e}\Big)\Big)+A^i\Big(\frac{x}{\e}\Big)\frac{\partial w}{\partial y_i}\Big(x,\frac{x}{n\e}\Big)\Big)\\
\nn&\quad=\sum_{i=1}^N A^i\Big(n\frac{x}{n\e}\Big)\Big(\frac{\partial u(x)}{\partial x_i}+\frac{\partial w}{\partial x_i}\Big(x,\frac{x}{n\e}\Big)\Big)
\end{align}
where in the last equality we used the fact that $w\in{\mathcal{C}}_u^{\pdeor(n\cdot)}\cap C^1(\Omega;C^1_{\rm per}(\Rn;\rd))$. Applying Proposition \ref{prop:simple-2-scale} we obtain
\be{eq:add-1}
\pdeep{\e}u_{\e}\wk\iQ\sum_{i=1}^N A^i(ny)\Big(\frac{\partial u(x)}{\partial x_i}+\frac{\partial w}{\partial x_i}(x,y)\Big)\,dy\quad\text{weakly in }L^p(\Omega;\rd),
\ee
and hence strongly in $W^{-1,p}(\Omega;\rd)$ by the compact embedding $L^p\hookrightarrow W^{-1,p}.$ On the other hand, since $w\in {\mathcal{C}}_u^{\pdeor(n\cdot)}$, there holds
\be{eq:add-2}\iQ\sum_{i=1}^N A^i(ny)\Big(\frac{\partial u(x)}{\partial x_i}+\frac{\partial w}{\partial x_i}(x,y)\Big)\,dy=\sum_{i=1}^N\frac{\partial}{\partial x_i} \Big(\iQ A^i(ny)(u(x)+w(x,y))\,dy\Big)=0.\ee
Combining \eqref{eq:add-1} with \eqref{eq:add-2} we deduce \eqref{eq:pde-c1}.\\
\emph{Step 2}:
We will now extend the construction in Step 1 to the general case where $u\in{\mathcal{C}}^{\pdeor(n\cdot)}$ and $w\in {\mathcal{C}}_u^{\pdeor(n\cdot)}$.
Extend $u$ and $w$ by setting them equal to zero outside $\Omega$ and $\Omega\times Q$, respectively. We claim that we can find sequences $\{u^k\}$ and $\{w^k\}$ such that $u^k\in C^{\infty}(\bar{\Omega};\rd)$, $w^k\in C^{\infty}(\bar{\Omega};C^{\infty}_{\rm per}(\Rn;\rd))$, and
\begin{align}
\label{eq:conv-a}&u^k\to u\quad\text{strongly in }L^p(\Omega;\rd),\\
\label{eq:conv-b}&w^k\to w\quad\text{strongly in }L^p(\Omega\times Q;\rd),\\
\label{eq:conv-c}&\sum_{i=1}^N\frac{\partial}{\partial {x_i}}\Big(\iQ A^i(ny)(u^k(x)+w^k(x,y))\,dy\Big)\to 0\quad\text{strongly in }W^{-1,p}(\Omega;\rl),
\end{align}
and
\be{eq:conv-d}\sum_{i=1}^N\frac{\partial}{\partial {y_i}}\Big(A^i(ny)(u^k(x)+w^k(x,y))\Big)\to 0\quad\text{strongly in }L^p(\Omega;W^{-1,p}(Q;\R^l)).\ee

Indeed, by first regularizing $u$ and $w$ with respect to the variable $x$, we construct two sequences $\{{u}^k\}$ and $\{\tilde{w}^k\}$ such that ${u}^k\in C^{\infty}(\bar{\Omega};\rd)$, $\tilde{w}^k\in C^{\infty}(\bar{\Omega};L^p_{\rm per}(\Rn;\rd))$, and
\begin{eqnarray*}
&&{u}^k\to u\quad\text{strongly in }L^p(\Omega;\rd),\\
&&\tilde{w}^k\to w\quad\text{strongly in }L^p(\Omega\times Q;\rd),\\
&&\sum_{i=1}^N\frac{\partial}{\partial {x_i}}\Big(\iQ A^i(ny)({u}^k(x)+\tilde{w}^k(x,y))\,dy\Big)\to 0\quad\text{strongly in }W^{-1,p}(\Omega;\rl).
\end{eqnarray*}
In addition,
\be{eq:conv-de}\sum_{i=1}^N\frac{\partial}{\partial {y_i}}\Big(A^i(ny)({u}^k(x)+\tilde{w}^k(x,y))\Big)= 0\quad\text{in }W^{-1,p}(Q;\rl)\text{ for a.e. }x\in\Omega.\ee
Now, by regularizing with respect to $y$ we construct $\{w^k\}$, such that $w^k\in C^{\infty}(\bar{\Omega};C^{\infty}_{\rm per}(\Rn;\rd))$. It is immediate to see that $\{u^k\}$ and $\{w^k\}$ satisfy \eqref{eq:conv-a}--\eqref{eq:conv-c}, and in particular
\be{eq:conv-f}\tilde{w}^k-w^k\to 0\quad\text{strongly in }L^p(\Omega;L^p_{\rm per}(Q;\rd)).\ee
To prove \eqref{eq:conv-d}, consider maps $\varphi\in L^{p'}(\Omega)$ and $\psi\in W^{1,{p'}}_0(\Omega;\rl)$. By the regularity of the operators $A^i$ and by \eqref{eq:conv-de} there holds
\bas
&\Big|\iO\scal{\sum_{i=1}^N\frac{\partial }{\partial y_i}\Big(A^i(ny)(u^k(x)+w^k(x,y))\Big)}{\psi(y)}\varphi(x)\,dx\Big|\\
&\quad=\Big|\sum_{i=1}^N\iOQ A^i(ny)(u^k(x)+w^k(x,y))\cdot \varphi(x)\frac{\partial \psi(y)}{\partial y_i}\,dy\,dx\Big|\\
&\quad=\Big|\sum_{i=1}^N\iOQ A^i(ny)(u^k(x)+w^k(x,y)-({u}^k(x)+\tilde{w}^k(x,y)))\cdot \varphi(x)\frac{\partial \psi(y)}{\partial y_i}\,dy\,dx\Big|\\
&\quad\leq C\|w^k-\tilde{w}^k\|_{L^p(\Omega\times Q;\rd)}\|\psi\|_{W^{1,{p'}}_0(Q;\rd)}\|\varphi\|_{L^{p'}(\Omega)}.
\end{align*}
 Property \eqref{eq:conv-d} follows now by \eqref{eq:conv-f}.
 
Apply Lemma \ref{lemma:projection} and Remark \ref{remark:if-regular} to the sequence $\{u^k+w^k\}$ to construct a sequence \mbox{$\{v^k\}\subset C^1(\Omega;C^{\infty}_{\rm per}(\Rn;\rd))$} such that
\be{eq:conv-projected}
v^k-(u^k+w^k)\to 0\quad\text{strongly in }L^p(\Omega\times Q;\rd),
\ee
\be{eq:pde-projected-1}
\sum_{i=1}^N\frac{\partial}{\partial {x_i}}\Big(\iQ A^i(ny)v^k(x,y)\,dy\Big)=0\quad\text{in }W^{-1,p}(\Omega;\rl)
\ee
and
\be{eq:pde-projected-2}
\sum_{i=1}^N\frac{\partial}{\partial {y_i}}( A^i(ny)v^k(x,y))=0\quad\text{in }W^{-1,p}(Q;\rl)\text{ for a.e. }x\in\Omega.
\ee
Consider now the maps
$$v^k_{\e}(x):=v^k\Big(x,\frac{x}{n\e}\Big)\quad\text{for a.e. }x\in\Omega.$$
By Proposition \ref{prop:simple-2-scale}, arguing as in the proof of \eqref{eq:bd-c1}, we observe that
\begin{eqnarray*}
&&v^k_{\e}\wk \iQ v^k(x,y)\,dy\quad\text{weakly in }L^p(\Omega;\rd),\\
&&v^k_{\e}\sts v^k\quad\text{strongly two-scale in }L^p(\Omega\times Q;\rd),\\
&&\iO f(v^k_{\e}(x))\,dx\to\iOQ f(v^k(x,y))\,dy\,dx.
\end{eqnarray*}
as $\e\to 0$, and there exists a constant $C$ independent of $\ep$ and $k$ such that
$$\|v_{\e}^k\|_{L^p(\Omega;\R^d)}\leq C\|v^k\|_{L^p(\Omega\times Q;\R^d)}$$
for every $\e$ and $k$. Hence, by \eqref{eq:conv-a}, \eqref{eq:conv-b}, and \eqref{eq:conv-projected}, there exists a constant $C$ independent of $\ep$ and $k$ such that
$$\|v_{\e}^k\|_{L^p(\Omega;\R^d)}\leq C(\|u\|_{L^p(\Omega;\R^d)}+\|w\|_{L^p(\Omega\times Q;\R^d)})$$
for every $\e$ and $k$.
 In addition, again by Proposition \ref{prop:simple-2-scale}, proceeding as in the proof of \label{eq:pde-c1} we can establish the analogues of \eqref{eq:equality}--\eqref{eq:add-2} and we conclude that
\begin{multline}
\label{eq:pde-limsup-zero}
\pdeep{\e}v^k_{\e}=\sum_{i=1}^N \frac{\partial}{\partial {x_i}}\Big(A^i\Big(\frac{x}{\ep}\Big)v^k_{\e}(x)\Big)\to \sum_{i=1}^N\frac{\partial}{\partial {x_i}}\Big(\iQ A^i(ny)v^k(x,y)\,dy\Big)=0
\end{multline}
strongly in $W^{-1,p}(\Omega;\rl)$, as $\e\to 0$.
Now, by \eqref{eq:conv-a}, \eqref{eq:conv-b} and \eqref{eq:conv-projected},
\be{eq:limsup-sts}v^k_{\e}\sts u+w\quad\text{strongy two-scale in }L^p(\Omega\times Q;\rd)\ee
as $\e\to 0$ and $k\to +\infty$, in this order. Hence, by Theorem \ref{thm:equivalent-two-scale}, \eqref{eq:growth-p-pde}, \eqref{eq:conv-a}, \eqref{eq:conv-b}, \eqref{eq:conv-projected}, \eqref{eq:pde-limsup-zero} and \eqref{eq:limsup-sts},
\bas
\limsup_{k\to +\infty}\limsup_{\e\to 0}&\Big\{\|T_{\e}v_{\e}^k-(u+w)\|_{L^p(\Omega\times Q;\rd)}+\|\pdeep{\e}v_{\e}^k\|_{W^{-1,p}(\Omega;\rl)}\\
&\quad+\Big|\int_{\Omega}f(v_{\e}^k(x))\,dx-\int_{\Omega}\iQ f(u(x)+w(x,y))\,dy\,dx\Big|\Big\}=0.
\end{align*}
By Attouch's diagonalization lemma \cite[Lemma 1.15 and Corollary 1.16]{attouch} we can extract a sequence $\{k(\e)\}$ such that, setting 
$$v^{\e}:=v^{k(\e)}_{\e},$$
the sequence $\{v^{\e}\}$ satisfies \eqref{eq:wk-c1}--\eqref{eq:bd-c1}.
\\\emph{Step 3}: Let $u\in\C$ and $\eta>0$. Then $\F(u)<+\infty$ and there exists $r_{\eta}>0$ such that
$$\F(u)+\eta>\inf\Big\{\liminfn \FBN^{\,r_{\eta}}(u_n):\,u_n\wk u\quad\text{weakly in }L^p(\Omega;\R^d)\Big\}.$$
In particular, there exists a sequence $\{u_n^{\eta}\}\subset L^p(\Omega;\R^d)$ with $u_n\in\mathcal{C}^{\pdeor(n\cdot)}$ for every $n\in\N$, such that
\be{eq:u-eta-n}u_n^{\eta}\wk u\quad\text{weakly in }L^p(\Omega;\R^d)\ee
as $n\to +\infty$, and
$$\F(u)+\eta\geq\lim_{n\to +\infty}\FBN^{\,r_{\eta}}(u_n^{\eta})=\lim_{n\to +\infty}\FFN^{\,r_{\eta}}(u_n^{\eta}).$$

By the definition of $\FFN^{\,r_{\eta}}$, for every $n\in \N$ there exists $w_n^{\eta}\in {\mathcal{C}}^{\pdeor(n\cdot)}_{u_n^{\eta}}$ such that 
\be{eq:w-eta-n}\|w_n^{\eta}\|_{L^p(\Omega\times Q;\R^d)}\leq r_{\eta},\ee
 and
$$\FFN^{\,r_{\eta}}(u_n^{\eta})\leq \iOQ f(u_n^{\eta}(x)+w_n^{\eta}(x,y))\,dy\,dx\leq \FFN^{\,r_{\eta}}(u_n^{\eta})+\frac{1}{n}.$$
Applying Steps 1 and 2 we construct sequences $\{v^{\eta}_{n,\e}\}\subset L^p(\Omega\times Q;\R^d)$ such that 
\ba{eq:v-eta-n}
&\sup_{\e>0}\|v_{n,\e}^{\eta}\|_{L^p(\Omega\times Q;\R^d)}\leq C\|u_n^{\eta}+w_n^{\eta}\|_{L^p(\Omega\times Q;\R^d)},\\
\nn&v_{n,\e}^{\eta}\wk u_n^{\eta}\quad\text{weakly in }L^p(\Omega\times Q;\rd),\\
\nn&\pdeep{\e} v^{\eta}_{n,\e}\to 0\quad\text{strongly in }W^{-1,p}(\Omega;\rl),\\
\nn&\iO f(v_{n,\e}^{\eta}(x))\,dx\to\iOQ f(u_n^{\eta}(x)+w_n^{\eta}(x,y))\,dy\,dx,
\end{align}
as $\e\to 0$.
In addition, by \eqref{eq:u-eta-n}--\eqref{eq:v-eta-n}, the sequence $\{v_{n,\e}^{\eta}\}$ is uniformly bounded in $L^p(\Omega\times Q;\R^d)$
Therefore, by the metrizability of bounded sets in the weak $L^p$ topology and Attouch's diagonalization lemma \cite[Lemma 1.15 and Corollary 1.16]{attouch} there exists a sequence $\{n(\e)\}$ such that, setting
$$u_{\e}^{\eta}:=v_{n(\e),\e}^{\eta},$$
properties \eqref{eq:wk-limsup} and \eqref{eq:pde-limsup} are fulfilled, with
$$\limsup_{\e\to 0}\int_{\Omega}f(u_{\e}^{\eta}(x))\,dx\leq\F(u)+\eta.$$
The thesis follows now by the arbitrariness of $\eta$.
\end{proof}
\section{Homogenization for measurable operators}
\label{section:bdd}
Here we prove the main result of this paper, concerning the case in which $A^i\in L^{\infty}_{\rm per}(\R^N;\M^{l\times d})$, $i=1,\cdots,N$. 
\begin{theorem}
\label{thm:main-homogenized-pde-bdd}
 Let $1<p<+\infty$. Let $A^i\in L^{\infty}_{\rm per}(\R^N;\M^{l\times d})$, $i=1,\cdots,N$, assume that the operator $\aaa$ satisfies  
 the invertibility requirement in \eqref{eq:A-invertibility}, and let $\pdeep{\e}$ be the operator defined in \eqref{eq:def-pde-epsilon}. Let $f:\R^d\to[0,+\infty)$ be a continuous function satisfying the growth condition \eqref{eq:growth-p-pde}. Then, for every $u\in L^p(\Omega;\R^d)$ there holds
\bas
 &\inf\Big\{\liminf_{\ep\to 0}\iO f(u_{\ep}(x))\,dx:\,u_{\ep}\wk u\quad\text{weakly in }L^p(\Omega;\rd),\\
 &\qquad\text{and }\pdeep{\e}u_{\e}\to 0\quad\text{strongly in }W^{-1,q}(\Omega;\rl)\text{ for every }1\leq q<p\Big\}\\
 &\quad=\inf\Big\{\limsup_{\ep\to 0}\iO f(u_{\ep}(x))\,dx:\,u_{\ep}\wk u\quad\text{weakly in }L^p(\Omega;\rd),\\
 &\qquad\text{and }\pdeep{\e}u_{\e}\to 0\quad\text{strongly in }W^{-1,q}(\Omega;\rl)\text{ for every }1\leq q<p\Big\}= \F(u).
 \end{align*}
\end{theorem}
The strategy of the proof consists in constructing a sequence of operators $\pdeor_k$ with smooth coefficients which approximate the operator $\pdeor$, so that Theorem \ref{thm:main-homogenized-pde} can be applied to each $\pdeor_k$. Let $\{\rho_k\}$ be a sequence of mollifiers and consider the operators
 $$\pdeor_k:L^p(\Omega;\rd)\to W^{-1,p}(\Omega;\rl)$$ defined as
 \be{eq:operators-k}\pdeor_k v(x):=\sum_{i=1}^N A^i_k(x)\frac{\partial v(x)}{\partial {x_i}},\ee
 where $A^i_k:=A^i\ast\rho_k$ for every $i=1,\cdots,N$, and for every $k$. Then $A^i_k\in C^{\infty}_{\rm per}(\R^N;\M^{l\times d})$, $i=1,\cdots, N$, for every $k$, 
 \be{eq:meas-conv}A^i_k\to A^i\quad\text{strongly in }L^m(Q;\M^{l\times d})\ee
for $1\leq m<+\infty,\quad i=1,\cdots,N$,
 \be{eq:unif-bd}\|A^i_k\|_{L^{\infty}(Q;\M^{l\times d})}\leq \|A^i\|_{L^{\infty}(Q;\M^{l\times d})}\quad\text{for } i=1,\cdots,N,\ee
and the operators $\pdeor_k$ satisfy the uniform ellipticity condition
 \be{eq:A-invertibility-bdd}
\mathcal{A}_k(x)\lambda\cdot\lambda\geq\alpha|\lambda|^2\quad\text{for every }\lambda\in\rd,\quad\text{for every }k.
\ee

We first prove two preliminary lemmas. The first one will allow us to approximate every element $u\in\C$ by sequences $\{u^k\}\subset L^p(\Omega;\rd)$ with $u^k\in \ck$ for every $k$.
\begin{lemma}
\label{lemma:approximation}
Let $1<p<+\infty$. Let $\pdeor$ be as in Theorem \ref{thm:main-homogenized-pde-bdd} and let $\{\pdeor_k\}$ be the sequence of operators with smooth coefficients constructed as above. Let $\C$ be the class introduced in \eqref{eq:limit-domain} and let $u\in \C$. Then there exists a sequence $\{u^{k}\}\subset L^p(\Omega;\rd)$ such that $u^{k}\in\ck$ for every $k$, and
$$u^{k}\to u\quad\text{strongly in }L^p(\Omega;\rd).$$
Moreover, for every $w\in\C_u$ there exists a sequence $\{w^{k}\}\subset L^p(\Omega\times Q;\rd)$ such that $w^{k}\in\ck_{u^k}$ for every $k$ and
$$w^{k}\to w\quad\text{strongly in }L^p(\Omega\times Q;\rd).$$
\end{lemma}
\begin{proof}
Let $u\in\C$ and let $w\in\C_u$. We first construct a sequence $\{v^n\}\subset L^p(\Omega\times Q;\rd)$, defined as
$$v^n(x):=(u(x)+w(x,y))\varphi^n(x)\quad\text{for a.e. }x\in\Omega, y\in Q,$$
where $\{\varphi_n\}\in C^{\infty}_c(\Omega;[0,1])$ with $\varphi_n\nearrow 1$. Without loss of generality, up to a dilation and a translation we can assume that $\Omega\subset Q$. Extending each map $v^n$ by zero in $Q\setminus\Omega$ and then periodically, and arguing as in the proof of Lemma \ref{lemma:projection}, it is easy to see that \ba{eq:vn}&{v}^n\to u+w\quad\text{strongly in }L^p(Q\times Q;\rd),\\
\nonumber&\sum_{i=1}^N\frac{\partial}{\partial {x_i}}\Big(\iQ A^i(y){v}^n(x,y)\,dy\Big)\to 0\quad\text{strongly in }W^{-1,p}(Q;\rl),\\
\nonumber&\sum_{i=1}^N\frac{\partial}{\partial {y_i}}\Big(A^i(y){v}^n(x,y)\Big)=0\quad\text{in }L^p(Q;W^{-1,p}(Q;\rl)).
\end{align}
By \eqref{eq:meas-conv}, \eqref{eq:unif-bd} and the dominated convergence theorem, we also have
$$A^i_k(y)v^n(x,y)\to A^i(y)v^n(x,y)\quad\text{strongly in }L^p(Q\times Q;\rl)$$
as $k\to +\infty$, for every $n$. Therefore,
\bas
&\sum_{i=1}^N\frac{\partial}{\partial {x_i}}\Big(\iQ A^i_k(y){v}^n(x,y)\,dy\Big)\to 0\quad\text{strongly in }W^{-1,p}(Q;\rl),\\
&\sum_{i=1}^N\frac{\partial}{\partial {y_i}}\Big(A^i_k(y){v}^n(x,y)\Big)\to 0\quad\text{strongly in }L^p(Q;W^{-1,p}(Q;\rl))
\end{align*}
as $k\to +\infty$ and $n\to +\infty$, in this order. In particular,
\bas
&\lim_{n\to +\infty}\lim_{k\to +\infty}\Bigg\{\|{v}^n- (u+w)\|_{L^p(Q\times Q;\rd)}\\
&\quad+\Big\|\sum_{i=1}^N\frac{\partial}{\partial {x_i}}\Big(\iQ A^i_k(y){v}^n(x,y)\,dy\Big)\Big\|_{W^{-1,p}(Q;\rl)}\\
&\quad+\sum_{i=1}^N\|A^i_k(y)v^n(x,y)-A^i(y)(u(x)+w(x,y))\|_{L^p(Q\times Q;\R^d)}\\
&\quad+\Big\|\Big\|\sum_{i=1}^N\frac{\partial}{\partial {y_i}}\Big(A^i_k(y){v}^n(x,y)\Big)\Big\|_{W^{-1,p}(Q;\rl)}\Big\|_{L^p(Q)}\Bigg\}= 0,
\end{align*}
hence by Attouch's diagonalization lemma \cite[Lemma 1.15 and Corollary 1.16]{attouch}, we can extract a subsequence $\{n(k)\}$ such that
\ba{eq:number38}
&v^{n(k)}\to u+w\quad\text{strongly in }L^p(\Omega\times Q;\R^d),\\
&A^i_k(y)v^{n(k)}(x,y)\to A^i(y)(u(x)+w(x,y))\quad\text{strongly in }L^p(Q\times Q;\rl),\\
\nn&\sum_{i=1}^N\frac{\partial}{\partial {x_i}}\Big(\iQ A^i_{k}(y){v}^{n(k)}(x,y)\,dy\Big)\to 0\quad\text{strongly in }W^{-1,p}(Q;\rl),\\
\nn&\sum_{i=1}^N\frac{\partial}{\partial {y_i}}\Big(A^i_{k}(y){v}^{n(k)}(x,y)\Big)\to 0\quad\text{strongly in }L^p(Q;W^{-1,p}(Q;\rl))
\end{align}
as $k\to +\infty$. Setting 
$$R^i_{k}(x,y):=A^i_{k}(y)v^{n(k)}(x,y)\quad\text{for a.e. }(x,y)\in Q\times Q,$$ and defining the mappings $R^{k}\in L^p(Q;L^p_{\rm per}(\Rn;\rd))$ as
$$R^{k}_{ij}:=(R^i_{k})_j,\quad\text{for all }i=1,\cdots,N,\,j=1\cdots,l,$$
we have that
\bas
&R^i_{k}(x,y)\to A^i(y)(u(x)+w(x,y)),\,i=1,\cdots,N \quad\text{strongly in }L^p(Q\times Q;\R^d),\\
&\sum_{i=1}^N \frac{\partial}{\partial {x_i}}\Big(\iQ R^i_{k}(x,y)\,dy\Big)\to 0\quad\text{strongly in }W^{-1,p}(Q;\rl),\\
&\sum_{i=1}^N\frac{\partial}{\partial {y_i}}(R^i_{k}(x,y))\to 0\quad\text{strongly in }L^p(Q;W^{-1,p}(Q;\rl)).
\end{align*}
Therefore, using Lemma \ref{lemma:div} we argue as in Lemma \ref{lemma:projection} and construct a sequence $S^{k}\in L^p(Q;L^p_{\rm per}(\Rn;\rd))$, satisfying
\ba{eq:new-pde-0-r-bdd}
&S^{k}-R^{k}\to 0\quad\text{strongly in }L^p(Q\times Q;\R^d),\\
\label{eq:pde-1-r-bdd}
&\sum_{i=1}^N \frac{\partial}{\partial {x_i}}\Big(\iQ S^i_{k}(x,y)\,dy\Big)=0\quad\text{in }W^{-1,p}(Q;\rl),\\
\label{eq:pde-2-r-bdd}
&\sum_{i=1}^N\frac{\partial}{\partial {y_i}}(S^i_{k}(x,y))=0\quad\text{in }W^{-1,p}(Q;\rl),\quad\text{for a.e. }x\in Q.
\end{align}
Finally, setting 
$$u^k(x):=\iQ \mathcal{A}_{k}(y)^{-1}S^{k}(x,y)\,dy\quad\text{for a.e. }x\in\Omega$$
and
$$w^k(x,y):=\mathcal{A}_{k}(y)^{-1}S^{k}(x,y)-u^k(x)\quad\text{for a.e. }x\in\Omega\text{ and }y\in Q,$$
by \eqref{eq:pde-1-r-bdd} and \eqref{eq:pde-2-r-bdd} we deduce that $w^k\in \mathcal{C}_{u^k}^{\pdeor_{k}}$, {i.e.} $u^k\in\mathcal{C}^{\pdeor_{k}}$ for every $k$. Moreover, by  \eqref{eq:A-invertibility}, \eqref{eq:number38} and \eqref{eq:new-pde-0-r-bdd}, 
\bas
&\|u^k-u\|_{L^p(\Omega;\rd)}=\Big\|\iQ (u^k(x)+w^k(x,y)-(u(x)+w(x,y)))\,dy\Big\|_{L^p(\Omega;\rd)}\\
&\quad\leq C\|u^k+w^k-(u+w)\|_{L^p(\Omega\times Q;\rd)}\\
&\quad\leq\|u^k+w^k-v^{n(k)}\|_{L^p(\Omega\times Q;\rd)}+\|v^{n(k)}-(u+w)\|_{L^p(\Omega\times Q;\rd)}\\
&\quad=\|\mathcal{A}_{k}(y)^{-1}(S^{k}(x,y)-R^{k}(x,y))\|_{L^p(\Omega\times Q;\rd)}+\|v^{n(k)}-(u+w)\|_{L^p(\Omega\times Q;\rd)}\to 0.
\end{align*}
The convergence of $w^k$ to $w$ follows in a similar way.
\end{proof}
In view of Lemma \ref{lemma:approximation} we can prove the analog of Proposition \ref{lemma:limit-maps} in the case in which $A^i\in L^{\infty}_{\rm per}(\R^N;\M^{l\times d})$, $i=1,\cdots, N$.
\begin{lemma}
\label{lemma:limit-maps-bdd}
Under the assumptions of Theorem \ref{thm:main-homogenized-pde-bdd}, there holds
\ba{eq:charact-C-bdd}
 &\C=\Big\{u\in L^p(\Omega;\rd):\,\text{there exists a sequence }\{v_{\e}\}\subset L^p(\Omega;\rd)\text{ such that }\\
 \nn&\quad v_{\e}\wk u\quad\text{weakly in }L^p(\Omega;\rd)\\
 \nn&\quad\text{ and }\pdeep{\e}v_{\e}\to 0\quad\text{strongly in }W^{-1,q}(\Omega;\rl),\,\text{for all }1\leq q<p\Big\}.
 \end{align}
\end{lemma}
\begin{proof}
Let $\mathcal{D}$ be the set in the right-hand side of \eqref{eq:charact-C-bdd}. The inclusion $$\mathcal{D}\subset \C$$ follows by Remark \ref{rk:not-so-easy}. To prove the opposite inclusion, let $u\in\C$ and let $w\in \C_u$. By Lemma \ref{lemma:approximation} we construct sequences $\{u^{k}\}\subset L^p(\Omega;\rd)$ and $\{w^{k}\}\subset L^p(\Omega\times Q;\rd)$ such that $u^{k}\in\ck$ for every $k$, $w^{k}\in\ck_{u^k}$ for every $k$, 
$$u^{k}\to u\quad\text{strongly in }L^p(\Omega;\rd)$$
and
$$w^{k}\to w\quad\text{strongly in }L^p(\Omega\times Q;\rd),$$
where $\{\pdeor_k\}$ is the sequence of operators with smooth coefficients defined in \eqref{eq:operators-k}.
By Proposition \ref{lemma:limit-maps}, for every $k$ there exists a sequence $\{v_{\e}^k\}\subset L^p(\Omega;\rd)$ such that
\begin{eqnarray*}
&&v_{\e}^k\sts u^k+w^k\quad\text{strongly 2-scale in }L^p(\Omega\times Q;\rd),\\
&&\pdeepk{\e}v_{\e}^k:=\sum_{i=1}^N\frac{\partial}{\partial x_i}\Big(A^i_k\Big(\frac{x}{\ep}\Big)v_{\e}^k(x)\Big)\to 0\quad\text{strongly in }W^{-1,p}(\Omega;\rl)
\end{eqnarray*}
as $\e\to 0$. Hence, in particular, by Theorem \ref{thm:equivalent-two-scale},
$$\lim_{k\to +\infty}\lim_{\e\to 0}\|T_{\e}v_{\e}^k-(u+w)\|_{L^p(\Omega\times Q;\rd)}+\|\pdeepk{\e}v_{\e}^k\|_{W^{-1,p}(\Omega;\rl)}=0.$$
By Attouch's diagonalization lemma (\cite[Lemma 1.15 and Corollary 1.16]{attouch}), we can extract a subsequence $\{k(\e)\}$ such that
\begin{eqnarray*}
&&v_{\e}^{k(\e)}\sts u+w\quad\text{strongly 2-scale in }L^p(\Omega\times Q;\rd),\\
&&{\mathscr{A}_{k(\e),\e}}^{\rm div}v_{\e}^{k(\e)}\to 0\quad\text{strongly in }W^{-1,p}(\Omega;\rl)
\end{eqnarray*}
as $\e\to 0$. A truncation argument analogous to \cite[Lemma 2.15]{fonseca.muller} yields a $p-$equiintegrable sequence $\{v_{\e}\}$ satisfying
\begin{eqnarray}
\nonumber &&v_{\e}\wk u\quad\text{weakly in }L^p(\Omega;\rd),\\
\label{eq:diag-seq-k}&&{\mathscr{A}_{k(\e),\e}}^{\rm div}v_{\e}\to 0\quad\text{strongly in }W^{-1,q}(\Omega;\rl),
\end{eqnarray}
for every $1\leq q<p$.

To complete the proof, it remains to prove that
\be{eq:claim-all-ok}
\pdeep{\e}v_{\e}\to 0\quad\text{strongly in }W^{-1,q}(\Omega;\rl)\quad\text{for every }1\leq q<p. 
\ee
To this purpose, we first notice that, by \eqref{eq:meas-conv} and by Severini-Egoroff's theorem, there exists a sequence of measurable sets $\{E_n\}\subset Q$ such that
$|E_n|\leq\frac{1}{n}$ and  
\be{eq:unif-conv}A^i_{k(\e)}\to A^i\quad\text{uniformly on }Q\setminus E_n\ee
 for every $n=1,\cdots,+\infty$. Let $\eta>0$ and $1\leq q<p$ be fixed, and for $z\in\mathbb{Z}^N$, set
$$Q_{\e,z}:=\e z+\e Q\quad\text{and }\ze:=\{z\in\mathbb{Z}^N:Q_{\e,z}\cap \Omega\neq \emptyset\}.$$
By the $p$-equiintegrability of $\{v_{\e}\}$ and hence of $\{T_{\ep}v_{\ep}\}$ (see Proposition \ref{prop:conv-unf-op}), and by \eqref{eq:unif-bd}, we can assume that 
\be{eq:neg-set}\|T_{\e}v_{\e}\|_{L^p((\cup_{z\in\ze}\qe\setminus\Omega)\times Q;\rd)}\leq \eta,\ee
and $n$ is such that
\ba{eq:egoroff}
&\sum_{i=1}^N\int_{\Omega}\int_{E_n}|A^i_{k(\e)}(y)-A^i(y)|^q|T_{\e} v_{\e}(x,y)|^q\,dydx\\
\nn&\quad\leq C\sum_{i=1}^N(\|A^i_{k(\e)}\|^q_{L^{\infty}(Q;\M^{l\times d})}+\|A^i\|^q_{L^{\infty}(Q;\M^{l\times d})})\|T_{\e} v_{\e}(x,y)\|^q_{L^{q}(\Omega\times E_n;\rd)}\leq\eta.
\end{align}
We first notice that, by \eqref{eq:def-pde-epsilon}, there holds
\bas
 &\big\|\pdeep{\e}v_{\e}\big\|^q_{W^{-1,q}(\Omega;\rl)}\leq \big\|{\mathscr{A}_{k(\e),\e}}^{\rm div}v_{\e}\big\|^q_{W^{-1,q}(\Omega;\rl)}\\
 &\qquad+\sum_{i=1}^N\int_{\Omega}\Big|A^i_{k(\e)}\Big(\frac{x}{\ep}\Big)-A^i\Big(\frac{x}{\ep}\Big)\Big|^q|v_{\e}(x)|^q\,dx\\
 &\quad\leq \big\|{\mathscr{A}_{k(\e),\e}}^{\rm div}v_{\e}\big\|^q_{W^{-1,q}(\Omega;\rl)}+\sum_{i=1}^N\int_{\cup_{z\in\ze}(\qe\cap\Omega)}\Big|A^i_{k(\e)}\Big(\frac{x}{\ep}\Big)-A^i\Big(\frac{x}{\ep}\Big)\Big|^q|v_{\e}(x)|^q\,dx.
 \end{align*}
 Therefore, by \eqref{eq:diag-seq-k} we deduce 
\ba{eq:first-goes-zero}
  &\limsup_{\e\to 0}\big\|\pdeep{\e}v_{\e}\big\|^q_{W^{-1,q}(\Omega;\rl)}\\
  \nn&\quad\leq\limsup_{\ep\to 0}\sum_{i=1}^N\int_{\cup_{z\in\ze}(\qe\cap\Omega)}\Big|A^i_{k(\e)}\Big(\frac{x}{\ep}\Big)-A^i\Big(\frac{x}{\ep}\Big)\Big|^q|v_{\e}(x)|^q\,dx.
\end{align}
Changing variables, using the periodicity of the operators, and extending $v_{\e}$ to zero outside $\Omega$, the right-hand side of \eqref{eq:first-goes-zero} can be estimated as
 \ba{eq:second-goes-zero}
 &\sum_{i=1}^N\int_{\cup_{z\in\ze}(\qe\cap\Omega)}\Big|A^i_{k(\e)}\Big(\frac{x}{\ep}\Big)-A^i\Big(\frac{x}{\ep}\Big)\Big|^q|v_{\e}(x)|^q\,dx\\
\nn&\quad \leq
 \e^N\sum_{i=1}^N\sum_{z\in\ze}\iQ|A^i_{k(\e)}(y)-A^i(y)|^q|v_{\e}(\e z+\e y)|^q\,dy\\
 \nn&\quad= \sum_{i=1}^N\sum_{z\in\ze}\int_{\qe}\iQ|A^i_{k(\e)}(y)-A^i(y)|^q\Big|v_{\e}\Big(\e \floor[\Big]{\frac{x}{\e}}+\e y\Big)\Big|^q\,dydx\\
 \nn&\quad=\sum_{i=1}^N\int_{\cup_{z\in\ze}\qe}\iQ|A^i_{k(\e)}(y)-A^i(y)|^q|T_{\e} v_{\e}(x,y)|^q\,dydx.
 \end{align}
 By Proposition \ref{prop:isometry}, \eqref{eq:neg-set}, and by \eqref{eq:egoroff}, the right-hand side of \eqref{eq:second-goes-zero} is bounded from above as follows
 \ba{eq:third-goes-zero}
 &\sum_{i=1}^N\int_{\cup_{z\in\ze}\qe}\iQ|A^i_{k(\e)}(y)-A^i(y)|^q|T_{\e} v_{\e}(x,y)|^q\,dydx\\
 \nn&\quad\leq\sum_{i=1}^N\int_{\cup_{z\in\ze}\qe\setminus\Omega}\iQ|A^i_{k(\e)}(y)-A^i(y)|^q|T_{\e} v_{\e}(x,y)|^q\,dydx\\
\nn&\qquad+\sum_{i=1}^N\int_{\Omega}\int_{E_n}|A^i_{k(\e)}(y)-A^i(y)|^q|T_{\e} v_{\e}(x,y)|^q\,dydx\\
 \nn&\qquad+\sum_{i=1}^N\int_{\Omega}\int_{Q\setminus E_n}|A^i_{k(\e)}(y)-A^i(y)|^q|T_{\e} v_{\e}(x,y)|^q\,dydx\\
 \nn&\leq \eta+C\sum_{i=1}^N(\|A^i_{k(\e)}\|^q_{L^{\infty}(Q;\M^{l\times d})}+\|A^i\|^q_{L^{\infty}(Q;\M^{l\times d})})\|T_{\e}v_{\e}\|^q_{L^q((\cup_{z\in\ze}\qe\setminus\Omega)\times Q;\R^d)}\\
 \nn&\qquad+\sum_{i=1}^N\|A^i_{k(\e)}-A^i\|^q_{L^{\infty}(Q\setminus E_n;\M^{l\times d})}\|v_{\e}\|^q_{L^q(\Omega;\rd)}\\
\nn& \leq C\eta+C\sum_{i=1}^N\|A^i_{k(\e)}-A^i\|_{L^{\infty}(Q\setminus E_n;\M^{l\times d})}.
 \end{align}
 Finally, by \eqref{eq:unif-conv} and collecting \eqref{eq:first-goes-zero}--\eqref{eq:third-goes-zero}, we obtain
 $$\limsup_{\e\to 0}\big\|\pdeep{\ep}v_{\e}\big\|^q_{W^{-1,q}(\Omega;\rl)}\leq C\eta+\lim_{\e\to 0}C\sum_{i=1}^N\|A^i_{k(\e)}-A^i\|_{L^{\infty}(Q\setminus E_n;\M^{l\times d})}=C\eta$$
 for every $\eta>0$. By the arbitrariness of $\eta$ we conclude \eqref{eq:claim-all-ok}.
\end{proof}
\begin{proof}[Proof of Theorem \ref{thm:main-homogenized-pde}]
We first notice that by Remark \ref{rk:not-so-easy} the thesis is trivial if $u\notin \C$. By Proposition \ref{thm:liminf} and Remark \ref{rk:even-better},
 for every $u\in\C$ there holds
\bas
&\inf\Big\{\liminf_{\ep\to 0}\iO f(u_{\ep}(x))\,dx:\,u_{\ep}\wk u\quad\text{weakly in }L^p(\Omega;\rd),\\
 &\quad\text{and }\pdeep{\e}u_{\e}\to 0\quad\text{strongly in }W^{-1,q}(\Omega;\rl)\text{ for every }1\leq q<p\Big\}\geq \F(u).
 \end{align*}
 To complete the proof of the theorem, since
 \bas
 &\inf\Big\{\limsup_{\ep\to 0}\iO f(u_{\ep}(x))\,dx:\,u_{\ep}\wk u\quad\text{weakly in }L^p(\Omega;\rd),\\
 &\qquad\text{and }\pdeep{\e}u_{\e}\to 0\quad\text{strongly in }W^{-1,q}(\Omega;\rl)\text{ for every }1\leq q<p\Big\}\\
 &\quad\leq\inf\Big\{\limsup_{\ep\to 0}\iO f(u_{\ep}(x))\,dx:\,u_{\ep}\wk u\quad\text{weakly in }L^p(\Omega;\rd),\,\{u_{\ep}\}\text{ p-equiintegrable},\\
 &\qquad\text{and }\pdeep{\e}u_{\e}\to 0\quad\text{strongly in }W^{-1,q}(\Omega;\rl)\text{ for every }1\leq q<p\Big\},
 \end{align*}
 it suffices to show that
 \begin{multline}
 \label{eq:limsup-not-smooth}\inf\Big\{\limsup_{\ep\to 0}\iO f(u_{\ep}(x))\,dx:\,u_{\ep}\wk u\quad\text{weakly in }L^p(\Omega;\rd),\,\{u_{\ep}\}\text{ p-equiintegrable},\\
 \text{and }\pdeep{\e}u_{\e}\to 0\quad\text{strongly in }W^{-1,q}(\Omega;\rl)\text{ for every }1\leq q<p\Big\}\leq\F(u).
 \end{multline}
 
 To prove \eqref{eq:limsup-not-smooth}, we argue by approximation. Let $\{\pdeor_k\}$ be the sequence of operators constructed in \eqref{eq:operators-k} and satisfying \eqref{eq:meas-conv}--\eqref{eq:A-invertibility-bdd}.
 We first prove that for $n$, $u_n\in\mathcal{C}^{\pdeor(n\cdot)}$ and $w_n\in{\mathcal{C}}_{u_n}^{\pdeor(n\cdot)}$ fixed,
  \ba{eq:limsup-second-not-smooth}
 &\inf\Big\{\liminf_{k\to +\infty}\mathcal{F}_{\pdeor_k(n\cdot)}^{2\|w_n\|_{L^p(\Omega\times Q;\R^d)}}(u_n+\phi_k):\,\phi_k\to 0\quad\text{strongly in }L^p(\Omega;\rd),\\
 \nn&\qquad u_n+\phi_k\in{\mathcal{C}}^{\pdeor_k(n\cdot)}\text{ for every }k\Big\}\\
 \nn&\quad\leq\iOQ f(u_n(x)+w_n(x,y))\,dy\,dx.
 \end{align}
Indeed, by Lemma \ref{lemma:approximation} there exist sequences $\{u^k\}\subset L^p(\Omega;\rd)$ and $\{w^k\}\subset L^p(\Omega\times Q;\rd)$, such that $u^k\in{\mathcal{C}}^{\pdeor_k(n\cdot)}$, $w^k\in{\mathcal{C}}^{\pdeor_k(n\cdot)}_{u^k}$ for every $k$, 
$$u^k\to u_n\quad\text{strongly in }L^p(\Omega;\rd),$$
and
\be{eq:ns-w-k}w^k\to w_n\quad\text{strongly in }L^p(\Omega\times Q;\rd).\ee
For $k$ big enough,
$$\mathcal{F}_{\pdeor_k(n\cdot)}^{2\|w_n\|_{L^p(\Omega\times Q;\R^d)}}(u_n+(u^k-u_n))\leq\iOQ f(u^k(x)+w^k(x,y))\,dy\,dx\quad\text{for every }k,$$
and
$$\limsup_{k\to +\infty}\mathcal{F}_{\pdeor_k(n\cdot)}^{2\|w_n\|_{L^p(\Omega\times Q;\R^d)}}(u_n+(u^k(x)-u_n(x)))\leq \iOQ f(u_n(x)+w_n(x,y))\,dy\,dx,$$
which in turn implies \eqref{eq:limsup-second-not-smooth}.

Let now $u\in \C$ and $\eta>0$ be fixed. Then, there exist $r_{\eta}>0$ such that
\bas
\F(u)+\eta\geq\inf\Big\{\liminfn \FBN^{r_{\eta}}(u_n):\,u_n\wk u\quad\text{weakly in }L^p(\Omega;\R^d)\Big\},
\end{align*}
and sequences $\{u^{\eta}_n\}\in L^p(\Omega;\R^d)$, $\{w^{\eta}_n\}\in L^p(\Omega\times Q;\R^d)$ satisfying
\ba{eq:u-n-eta}
&u^{\eta}_n\wk u\quad\text{weakly in }L^p(\Omega;\R^d),\\
\nn&w^{\eta}_n\in {\mathcal{C}}^{\pdeor(n\cdot)}_{u^{\eta}_n}\quad\text{for every }n,\\
\nn&\|w^{\eta}_n\|_{L^p(\Omega\times Q;\R^d)}\leq r_{\eta}\quad\text{for every }n,
\end{align}
and
$$\F(u)+2\eta\geq \lim_{n\to +\infty}\iOQ f(u^{\eta}_n(x)+w^{\eta}_n(x,y))\,dy\,dx.$$
In particular, by \eqref{eq:limsup-second-not-smooth},
\bas
&\F(u)+2\eta\\ 
&\quad\geq\limsup_{n\to +\infty}\inf\Big\{\liminf_{k\to +\infty}\mathcal{F}_{\pdeor_k(n\cdot)}^{2\|w_n^{\eta}\|_{L^p(\Omega\times Q;\R^d)}}(u^{\eta}_n+\phi_k):\,\phi_k\to 0\quad\text{strongly in }L^p(\Omega;\rd),\\
 \nn&\qquad u^{\eta}_n+\phi_k\in{\mathcal{C}}^{\pdeor_k(n\cdot)}\text{ for every }k\Big\}.
\end{align*}
 For every $n,k$, let $w_{n,k}^{\eta}\in {\mathcal{C}}^{\pdeor_k(n\cdot)}_{u^{\eta}_n+\phi_k}$ be such that 
 \be{eq:w-n-k-eta}\|w_{n,k}^{\eta}\|_{L^p(\Omega\times Q;\R^d)}\leq 2\|w_{n}^{\eta}\|_{L^p(\Omega\times Q;\R^d)}\ee
 and
 \bas
 &\mathcal{F}_{\pdeor_k(n\cdot)}^{2\|w_{n}^{\eta}\|_{L^p(\Omega\times Q;\R^d)}}(u^{\eta}_n+\phi_k)\leq \iOQ f(u^{\eta}_n(x)+\phi_k(x)+w_{n,k}^{\eta}(x,y))\,dy\,dx\\
 &\quad \leq\mathcal{F}_{\pdeor_k(n\cdot)}^{2\|w_{n}^{\eta}\|_{L^p(\Omega\times Q;\R^d)}}(u^{\eta}_n+\phi_k)+\frac1k.
 \end{align*}
 Arguing as in Steps 1 and 2 of Proposition \ref{thm:limsup}, for every $n,k$ we construct a sequence $\{v_{\ep,n,k}^{\eta}(x)\}\subset L^p(\Omega;\R^d)$ such that
 \ba{eq:v-ep-n-k}
 &\|v_{\e,n,k}^{\eta}\|_{L^p(\Omega;\R^d)}\leq C\|u^{\eta}_n+\phi_k+w_{n,k}^{\eta}\|_{L^p(\Omega\times Q;\R^d)},\\
 \nn&v_{\e,n,k}^{\eta}\wk u_n^{\eta}+\phi_k\quad\text{weakly in }L^p(\Omega;\R^d),\\
 \nn&\pdeepk{\e}v_{\e,n,k}^{\eta}\to 0\quad\text{strongly in }W^{-1,p}(\Omega;\rl),\\
 \nn&\int_{\Omega}f(v_{\e,n,k}^{\eta}(x))\,dx\to\iOQ f(u^{\eta}_n(x)+\phi_k(x)+w_{n,k}^{\eta}(x,y))\,dy\,dx,
 \end{align}
 as $\ep\to 0$, where the constant $C$ is independent of $\ep$, $n$ and $k$. In particular, \eqref{eq:u-n-eta} and \eqref{eq:w-n-k-eta} yield that the sequence $\{v_{\e,n,k}^{\eta}\}$ is uniformly bounded in $L^p(\Omega;\R^d)$. By \eqref{eq:v-ep-n-k} the metrizability of bounded sets in the weak $L^p$ topology and Attouch's diagonalization lemma (\cite[Lemma 1.15 and Corollary 1.16]{attouch}), we can extract subsequences $\{n(\e)\}$ and $\{k(\e)\}$ such that, setting
 $$\tilde{u}_{\e}:=v_{\e,n(\e),k(\e)}^{\eta},$$
 there holds
 \bas
 &\tilde{u}_{\e}\wk u\quad\text{weakly in }L^p(\Omega;\R^d),\\
 &\pdeor_{k(\e)}^{\rm div}\tilde{u}_{\e}\to 0\quad\text{strongly in }W^{-1,p}(\Omega;\rl),\\
 &\limsup_{\e\to 0}\int_{\Omega}f(\tilde{u}_{\e}(x))\,dx\leq \F(u)+2\eta.
 \end{align*}
 In view of Lemma \ref{lemma:truncation} and Proposition \ref{prop:f-k}, we can construct a further sequence $\{u_{\e}\}\subset L^p(\Omega;\rd)$, $p-$equiintegrable and satisfying
 \bas
 &{u}_{\e}\wk u\quad\text{weakly in }L^p(\Omega;\R^d),\\
 &\pdeor_{k(\e)}^{\rm div}{u}_{\e}\to 0\quad\text{strongly in }W^{-1,q}(\Omega;\rl)\quad\text{for every }1\leq q<p,\\
 &\limsup_{\e\to 0}\int_{\Omega}f({u}_{\e}(x))\,dx\leq\limsup_{\e\to 0}\int_{\Omega}f(\tilde{u}_{\e}(x))\,dx\leq \F(u)+2\eta.
 \end{align*}
 Property \eqref{eq:limsup-not-smooth}
 follows now by noticing that exactly the same argument as in the proof of \eqref{eq:claim-all-ok} yields
 $$\pdeep{\ep}u_{\e}\to 0\quad\text{strongly in }W^{-1,q}(\Omega;\rl)\quad\text{for every }1\leq q<p,$$
 and by the arbitrariness of $\eta$.
 \end{proof}
\subsection{The case of constant coefficients}
\label{subsection:general}
In this subsection we show that the homogenized energy obtained in Theorem \ref{thm:main-homogenized-pde} reduces to the one identified by Braides, Fonseca and Leoni in \cite{braides.fonseca.leoni}, in the case in which the operators $A^i$, $i=1,\cdots, N$, are constant, and the constant rank condition \eqref{cr} is satisfied by the differential operator $\pdeor:L^p(\Omega;\R^d)\to W^{-1,p}(\Omega;\rl)$ defined as
$$\pdeor u:=\sum_{i=1}^N A^i \frac{\partial u}{\partial x_i}\quad\text{for every }u\in L^p(\Omega;\R^d).$$

In this case the classes $\C_{u}$  and $\C$ defined in \eqref{eq:limit-test-functions} and \eqref{eq:limit-domain} become, respectively
$$\cc^{{\rm const}}_{u}:=\Big\{w\in L^p(\Omega;L^p_{\rm per}(\Rn;\rd)):\,\iQ w(x,y)\,dy=0\text{ and }\sum_{i=1}^N A^i \frac{\partial w}{\partial y_i}(x,y)=0\Big\}$$
and
$$\cc^{\rm const}:=\Big\{u\in L^{p}(\Omega;\rd):\, \pdeor u=0\Big\}.$$
Recall that
$$\F(u)=\inf_{r>0}\inf\Big\{\liminfn\FB^{\,r}(u_n):\,u_n\wk u\quad\text{weakly in }L^p(\Omega;\R^d)\Big\}.$$

By definition of $\pdeor-$quasiconvex envelope, for every $u\in \cc^{\rm const}$ and $r>0$ we have
$$\FF^{\,r}(u)\geq \iO \qa f(u(x))\,dx.$$
 By \cite[Theorem 3.7]{fonseca.muller}, the $\pdeor-$quasiconvex envelope is lower semicontinuous with respect to the weak $L^p$ convergence of $\pdeor-$vanishing maps, hence
 $$\inf\Big\{\liminfn\FB^{\,r}(u_n):\,u_n\wk u\quad\text{weakly in }L^p(\Omega;\R^d)\Big\}\geq \iO\qa f(u(x))\,dx$$
 for every $r>0$, which yields
\be{eq:one-side}
\F(u)\geq \iO\qa f(u(x))\,dx.
\ee

Conversely, since $\pdeor$ is a differential operator with constant coefficients, for every $u\in\cc^{{\rm const}}$ the null map belongs to $\cc^{{\rm const}}_u$. Hence,
$$\iO f(u(x))\,dx\geq \FF^{\,r}(u)=\FB^{\,r}(u).$$
By taking the lower semicontinuous envelope of both sides with respect to the weak $L^p$ convergence of $\pdeor-$vanishing maps we obtain (see \cite[Theorem 1.1]{braides.fonseca.leoni})
\ba{eq:second-side}
\iO \qa(f(u(x)))\,dx&\geq \inf\Big\{\liminfn\FB^{\,r}(u_n):\,u_n\wk u\quad\text{weakly in }L^p(\Omega;\R^d)\Big\}\\
\nonumber&\geq \F(u).
\end{align}
Combining \eqref{eq:one-side} and \eqref{eq:second-side} we deduce that
$$\F(u)=\iO \qa f(u(x))\,dx.$$

\subsection{Nonlocality of the operator}
\label{subsection:nonlocal}
We end this section with an example that illustrates that, in general, when the operators $A^i$ are not constant then the functional $\F$ in \eqref{eq:def-lsc-energy} can be nonlocal, even when the energy density $f$ is convex.
\begin{example}
\label{example:nonlocal}
Let $N=d=p=2$ and $l=1$, and choose 
$$\Omega=(0,1)\times (0,1).$$
 Let $a\in C^{\infty}_{\rm per}(\R)$, with period $(-\tfrac12,\tfrac12)$, $a>-1$, and satisfying
\be{eq:after-ex}
\int_{-\tfrac12}^{\tfrac12}a(s)\,ds=0\quad\text{and}\quad \int_{-\tfrac12}^{\tfrac12}a^2(s)\,ds=1,
\ee
 and consider the operators $A^1,A^2:\R^2\to\R^2$ defined as 
$$A^1(x):=(1+a(x_2)\quad 0)\quad\text{and }A^2(x):=(0\quad 1)$$
for $x\in\R^2$. We have that
$$\aaa(y)\xi=\Bigg(\begin{array}{cc}1+a(y_2)&0\\0&1\end{array}\Bigg)\xi$$
for $\xi\in\R^2$ and $y\in Q$, therefore the operator $\aaa$ satisfies the uniform invertibility assumption \eqref{eq:A-invertibility}.

Consider the function $f:\R^2\to [0,+\infty)$, defined as $f(\xi):=|\xi|^2$ for every $\xi\in\R^2$. 
By \eqref{eq:after-ex}, for $n$ fixed, $u_n\in{\mathcal{C}^{\pdeor(n\cdot)}}$ if and only if there exists $w_n\in L^2(\Omega\times Q;\R^2)$ such that $\iQ w_n(x,y)\,dy=0$, and the following two conditions are satisfied
\be{eq:pde-ex-1}\frac{\partial}{\partial x_1}\Big(\iQ a(ny_2)w_{1,n}(x,y)\,dy\Big)+{\rm div}\,u_n(x)=0\quad\text{in }W^{-1,2}(\Omega),\ee
\be{eq:pde-ex-2}(1+a(ny_2))\frac{\partial w_{1,n}(x,y)}{\partial y_1}+\frac{\partial w_{2,n}(x,y)}{\partial y_2}=0\quad\text{in }W^{-1,2}(Q)\quad\text{for a.e. }x\in\Omega.\ee

Moreover, for every $\tilde{u}\in L^2(\Omega;\R^d)$ and $\tilde{w}\in L^2(\Omega\times Q;\R^2)$ with $\iQ \tilde{w}(x,y)\,dy=0$, there holds
\be{eq:ort-ex}\iOQ f(\tilde{u}(x)+\tilde{w}(x,y))\,dy=\iO|\tilde{u}(x)|^2\,dx+\iOQ |\tilde{w}(x,y)|^2\,dy\,dx.\ee
In view of \eqref{eq:ort-ex}, for every $r>0$, $n\in\N$, and $u_n\in {\mathcal{C}^{\pdeor(n\cdot)}}$, 
$$\FBN^{\,r}(u_n)\geq \iO|u_n(x)|^2\,dx,$$
hence by classical lower-semicontinuity results (see,e.g.,\cite[Theorem 5.14]{fonseca.leoni}),
\be{eq:05star}\F(u)\geq \iO|u(x)|^2\,dx\quad\text{for every }u\in\C.\ee

If $0\in\C_u$, i.e., ${\rm div}\,u=0$ (and hence $0\in\mathcal{C}_u^{\pdeor(n\cdot)}$ for every $n$), then 
$$\FBN^{\,r}(u)=\iO |u(x)|^2dx$$ for every $n$ and $r$, and choosing $u_n=u$ for every $n$ in \eqref{eq:def-lsc-energy}, we deduce that \eqref{eq:05star} holds with equality.
 
Strict inequality holds in \eqref{eq:05star} if $u\in\C$ but ${\rm div}\,u\neq 0$. Such fields exist, consider for example, 
 $$u(x):=\Big(\begin{array}{c}- x_1\\0\end{array}\Big)\quad\text{for }x\in\Omega.$$ Note that the map 
 $$w(x,y):=\Big(\begin{array}{c}a(y_2)x_1\\0\end{array}\Big)\quad\text{for every }(x,y)\in\Omega\times Q$$
 satisfies $w\in\C_u$ by \eqref{eq:after-ex}
 
 Assume by contradiction that there exists $u\in\C$, with ${\rm div}\,u\neq 0$, such that
 \be{eq:equalFalse}
 \F(u)=\int_{\Omega}|u(x)|^2\,dx.
 \ee
 By the definition of $\F$ there exists a sequence $\{r_m\}$ of real numbers such that
 $$\F(u)=\lim_{m\to +\infty}\inf\Big\{\liminfn \FBN^{\,r_m}(u_n):\,u_n\wk u\quad\text{weakly in }L^p(\Omega;\R^d)\Big\}.$$
 For every $m\in\N$, let $\{u_n^m\}\subset L^2(\Omega;\R^d)$ be such that 
 $$u_n^m\wk u\quad\text{weakly in }L^2(\Omega;\R^d)$$
 as $n\to +\infty$, and
 $$\inf\Big\{\liminfn \FBN^{\,r_m}(u_n):\,u_n\wk u\quad\text{weakly in }L^2(\Omega;\R^d)\Big\}+\frac1m\geq \liminf_{n\to +\infty} \FBN^{\,r_m}(u_n^m).$$
 Then
 $$\F(u)=\lim_{m\to +\infty}\liminf_{n\to +\infty} \FBN^{\,r_m}(u_n^m),$$
 and by a diagonal argument we can extract a subsequence $\{n(m)\}$ such that, setting $u^m:=u_{n(m)}^m$,
 \be{eq:diag-seq-ex}\F(u)=\lim_{m\to +\infty}\overline{\mathcal{F}}_{\pdeor(n(m)\cdot)}^{\,r_m}(u^m).\ee
 By \eqref{eq:ort-ex}, for every $m$ there exists $w_m\in {\mathcal C}^{\pdeor(n(m)\cdot)}_{u^m}$ such that $\|w_m\|_{L^2(\Omega\times Q;\R^d)}\leq r_{m}$ and
  \ba{eq:08star}
  \overline{\mathcal{F}}_{\pdeor(n(m)\cdot)}^{\,r_m}(u^m)+\frac 1m&\geq \iOQ f(u^m(x)+w_m(x,y))\,dy\,dx\\
  &\nonumber=\iO|u^m(x)|^2\,dx+\iOQ |w_m(x,y)|^2\,dy\,dx.
  \end{align}
 Hence, in view of \eqref{eq:equalFalse}, \eqref{eq:diag-seq-ex} and \eqref{eq:08star},
  $$\iO|u(x)|^2\,dx=\lim_{m\to +\infty}\Big(\iO|u^m(x)|^2\,dx+\iOQ |w_m(x,y)|^2\,dy\,dx\Big),$$
  and so
  \be{eq:um-ex}u^m\to u\quad\text{strongly in }L^2(\Omega;\R^d),\ee
  and
  \be{eq:wm-ex}w_m\to 0\quad\text{strongly in }L^2(\Omega\times Q;\R^d).\ee
  By \eqref{eq:pde-ex-1} and the boundedness of the function $a$, properties \eqref{eq:um-ex} and \eqref{eq:wm-ex} yield
  $${\rm div }\,u^m\to 0\quad\text{strongly in }W^{-1,2}(\Omega)$$
  which in turn implies that ${\rm div}\,u=0$, contradicting the assumptions on $u$.
  
  We conclude that if $u\in \C$ satisfies ${\rm div}\,u=0$, then 
  \be{eq:becomes-simple}\F(u)=\int_{\Omega}|u(x)|^2\,dx,
  \ee 
  whereas if $u\in \C$ satisfies ${\rm div}\,u\neq 0$, then 
  \be{eq:becomes-not-simple}\F(u)>\int_{\Omega}|u(x)|^2\,dx.
  \ee

We now provide an explicit expression for the functional $\F$. We claim that for every $u\in\C$ there holds
\be{eq:ss-ex-ex}
\F(u)=\int_{\Omega}|u(x)|^2\,dx+\int_{\Omega}|\phi_u^\Omega(x)|^2\,dx,
\ee
where $\phi_u^\Omega\in L^2(\Omega)$ is the unique function satisfying $\int_0^1\phi_u^\Omega(x_1,x_2)\,dx_1=0$ for a.e. $x_2\in (0,1)$, and $\frac{\partial \phi_u^\Omega(x)}{\partial x_1}=-{\rm div}\,u(x)$ in $W^{-1,2}(\Omega)$.

To prove \eqref{eq:ss-ex-ex} we first establish a preliminary lemma.
\begin{lemma}
\label{lemma:ident-part-ex-ex}
Let $n\in \N$, and let $v\in \mathcal{C}^{\pdeor(n\cdot)}$. Then
\ba{eq:ident-ex-ex}
&\inf_{w\in \mathcal{C}_{\pdeor(n\cdot)}(v)}\Big[\int_{\Omega}|v(x)|^2\,dx+\iOQ |w(x,y)|^2\,dx\,dy\Big]\\
&\nn\quad=\int_{\Omega}|v(x)|^2\,dx+\int_{\Omega}|\phi_v^\Omega(x)|^2\,dx,\end{align}
where $\int_0^1\phi_v^\Omega(x_1,x_2)\,dx_1=0$ for a.e. $x_2\in(0,1)$, and $\frac{\partial\phi_v^\Omega(x)}{\partial x_1}=-{\rm div}\,v(x)$ in $W^{-1,2}(\Omega)$.
\end{lemma}
\begin{proof}
We recall that by \eqref{eq:pde-ex-1} and \eqref{eq:pde-ex-2}, $w\in {\mathcal{C}}_{\pdeor(n\cdot)}(v)$ if and only if 
\begin{align}
&\nonumber\iQ w(x,y)\,dy=0,\quad\text{for a.e. }x\in\Omega,\\ 
&\label{eq:n1-ex-ex}\frac{\partial}{\partial x_1}\Big(\iQ a(ny_2)w_{1}(x,y)\,dy\Big)+{\rm div}\,v(x)=0\quad\text{in }W^{-1,2}(\Omega),\\
&\nn(1+a(ny_2))\frac{\partial w_{1}(x,y)}{\partial y_1}+\frac{\partial w_{2}(x,y)}{\partial y_2}=0\quad\text{in }W^{-1,2}(Q)\quad\text{for a.e. }x\in\Omega.
\end{align}
By \eqref{eq:n1-ex-ex}, the map $\phi_v^\Omega$ defined in the statement of Lemma \ref{lemma:ident-part-ex-ex} satisfies
$$\phi_v^\Omega(x)=\iQ a(ny_2)w_{1}(x,y)\,dy-\int_0^1\Big(\iQ a(ny_2)w_{1}(x,y)\,dy\Big)\,dx_1\quad\text{for a.e. }x\in\Omega.$$
Rewriting $w_1$ as
$$w_1(x,y):=a(ny_2)\iQ a(ny_2)w_{1}(x,y)\,dy+\eta_w(x,y),$$
we have
$$\iQ a(ny_2)\eta_w(x,y)\,dy=0,$$
and hence
\begin{align*}
&\iOQ |w(x,y)|^2\,dx\,dy=\iOQ |w_1(x,y)|^2\,dx\,dy+\iOQ |w_2(x,y)|^2\,dx\,dy\\
&\quad \geq\int_{\Omega}\Big(\iQ a(ny_2)w_{1}(x,y)\,dy\Big)^2\,dx+\iOQ |\eta_w(x,y)|^2\,dx\,dy\\
&\quad \geq \int_{\Omega}|\phi_v^\Omega(x)|^2\,dx+\int_{\Omega}\Big(\int_0^1\Big(\iQ a(ny_2)w_{1}(x,y)\,dy\Big)\,dx_1\Big)^2\,dx\\
&\quad\geq \iO |\phi_v^\Omega(x)|^2\,dx\quad\text{for every }w\in{\mathcal{C}}_{\pdeor(n\cdot)}(v).
\end{align*}
In particular, we obtain the lower bound
$$\inf_{w\in \mathcal{C}_{\pdeor(n\cdot)}(v)}\Big[\int_{\Omega}|v(x)|^2\,dx+\iOQ |w(x,y)|^2\,dx\,dy\Big]\geq\int_{\Omega}|v(x)|^2\,dx+\iO |\phi_v^\Omega(x)|^2\,dx.$$
Property \eqref{eq:ident-ex-ex} follows by observing that \eqref{eq:after-ex} yields $$\big(a(ny_2)\phi_v^\Omega(x),0\big)\in {\mathcal{C}_{\pdeor(n\cdot)}(v)}.$$
\end{proof}
Let $u\in\C$. Arguing as in the proof of \eqref{eq:becomes-not-simple} there exist a sequence $\{n_m\}\subset\N$, with $n_{m}\to +\infty$ as $m\to+\infty$, and sequences $\{u^m\}\subset L^2(\Omega;\R^2)$ and $\{w^m\}\subset L^2(\Omega\times Q;\R^2)$, such that
\begin{align}
&\label{eq:1-ex-ex-lim}u^m\wk u\quad\text{weakly in }L^2(\Omega;\R^2),\\
&\nn w^m\in {\mathcal C}^{\pdeor(n_m\cdot)}_{u^m},
\end{align}
and
$$\F(u)=\lim_{m\to +\infty}\Bigg\{\int_{\Omega}|u^m(x)|^2\,dx+\iOQ |w^m(x,y)|^2\,dx\,dy\Bigg\}.$$
In view of Lemma \ref{lemma:ident-part-ex-ex}, 
\be{eq:limit-attained-ex-ex}
\F(u)\geq \limsup_{m\to +\infty}\Bigg\{\int_{\Omega}|u^m(x)|^2\,dx+\int_{\Omega} |\phi^m(x)|^2\,dx\Bigg\},
\ee
where 
\be{eq:prop-phi-n-m-ex1}\int_{0}^1\phi^m(x_1,x_2)\,dx_1=0,\quad\text{for a.e. }x_2\in(0,1),\ee
and 
\be{eq:prop-phi-n-m-ex2}\frac{\partial \phi^m(x)}{\partial x_1}=-{\rm div}\, u_n^m(x)\quad\text{in }W^{-1,2}(\Omega).\ee
 Since $u\in \C$, there holds $\F(u)<+\infty$, and the sequence $\{\phi^m\}$ is uniformly bounded in $L^2(\Omega)$. Therefore there exists $\phi_u^\Omega\in L^2(\Omega)$ such that, up to the extraction of a (not relabeled) subsequence,
\be{eq:wk-conv-phi-ex-ex}\phi^m\wk\phi_u^\Omega\quad\text{weakly in }L^2(\Omega),\ee
where, by \eqref{eq:1-ex-ex-lim}, \eqref{eq:prop-phi-n-m-ex1} and \eqref{eq:prop-phi-n-m-ex2}, 
$$\int_{-\tfrac12}^{\tfrac12}\phi_u^\Omega(x_1,x_2)\,dx_1=0\quad\text{for a.e. }x_2\in(0,1)$$
and 
$$\frac{\partial \phi_u^\Omega(x)}{\partial x_1}=-{\rm div}\,u(x)\quad\text{in }W^{-1,2}(\Omega).$$
In particular, by \eqref{eq:1-ex-ex-lim} and the lower semicontinuity of the $L^2$-norm,
$$\F(u)\geq \int_{\Omega}|u(x)|^2\,dx+\int_{\Omega}|\phi_u^\Omega(x)|^2\,dx.$$
To prove the opposite inequality, choose $u_n:=u$, $w_n:=\big(a(ny_2)\phi_u^\Omega(x),0\big)$ for every $n\in\N$. By Lemma \ref{lemma:ident-part-ex-ex}, for $r$ big enough there holds
$$\overline{\mathcal F}^r_{\pdeor(n\cdot)}(u_n)=\int_{\Omega}|u(x)|^2\,dx+\int_{\Omega}|\phi_u^\Omega(x)|^2\,dx.$$
The characterization \eqref{eq:ss-ex-ex} follows now by the definition of $\F$.\\

We conclude this example by showing that the functional $\F$ is nonlocal. Indeed, assume by contradiction that $\F$ is local. Then for every $u\in\C$ we can associate to $\F$ an additive set function $\F(u,\cdot)$ on the class $\mathcal{O}(\Omega)$ of open subsets of $\Omega$.
 In particular, for every pair $\Omega_1,\Omega_2\subset \Omega$, with $\Omega_1\subset\subset\Omega_2\subset\subset\Omega$, and for every $u\in\C$ in $\Omega$, there holds
\be{eq:ineq-false-ex-ex}\F(u,\Omega)\leq \F(u,\Omega\setminus\bar{\Omega}_1)+\F(u,\Omega_2).\ee
Let $\xi_1,\xi_2\in\R$, with $\xi_1\neq\xi_2$, and let $\ep,\delta>0$ be such that $\ep<\tfrac12$ and $\delta<\tfrac12-\ep$. Define
$$\Omega_1:=\big(\tfrac12-\ep,\tfrac12+\ep\big)\times \big(\tfrac12-\ep,\tfrac12+\ep\big)\quad\text{and}\quad \Omega_2:=\big(\tfrac12-\ep-\delta,\tfrac12+\ep+\delta\big)\times \big(\tfrac12-\ep-\delta,\tfrac12+\ep+\delta\big).$$
Consider the function
$$u_0(x):=\begin{cases} \Big(\begin{array}{c}-\xi_1\\0\end{array}\Big)&\text{if }x\in \Omega_1,\\
\Big(\begin{array}{c}-\xi_2\\0\end{array}\Big)&\text{otherwise in }\Omega.
\end{cases}$$
We observe $u_0$ belongs to $\C$, since the map 
 $$w_0(x,y):=\begin{cases}\Big(\begin{array}{c}a(y_2)\xi_1\\0\end{array}\Big)&\text{if }x\in \Omega_1,\,y\in Q,\\
 \Big(\begin{array}{c}a(y_2)\xi_2\\0\end{array}\Big)&\text{otherwise in }\Omega\times Q,\end{cases}
$$
 satisfies $w_0\in\C_{u_0}$. By \eqref{eq:becomes-simple}, since ${\rm{div}}\,u_0=0$ in $\Omega\setminus\bar{\Omega}_1$, there holds
 \ba{eq:primo-pezzo}
 &\F(u_0;\Omega\setminus\bar{\Omega}_1)=\int_{\Omega\setminus\bar{\Omega}_1}|u_0(x)|^2\,dx=(1-4\ep^2)\xi_2^2.
 \end{align}
A direct computation yields
$$\phi_{u_0}^\Omega=\begin{cases}(1-2\ep)(\xi_1-\xi_2)&\text{in }\Omega_1\\
2\ep(\xi_2-\xi_1)&\text{in }[\big(0,\tfrac12-\ep\big)\cup\big(\tfrac12+\ep,1\big)]\times\big(\tfrac12-\ep,\tfrac12+\ep\big)\\
0&\text{otherwise in }\Omega,\end{cases}$$
and
$$\phi_{u_0}^{\Omega_2}=\begin{cases}\frac{\delta(\xi_1-\xi_2)}{\ep+\delta}&\text{in }\Omega_1\\
\frac{\ep(\xi_2-\xi_1)}{\ep+\delta}&\text{in }[\big(\tfrac12-\ep-\delta,\tfrac12-\ep\big)\cup\big(\tfrac12+\ep,\tfrac12+\ep+\delta\big)]\times\big(\tfrac12-\ep,\tfrac12+\ep\big)\\
0&\text{otherwise in }\Omega_2.\end{cases}$$
Therefore,
\ba{eq:secondo-pezzo}
\F(u_0,\Omega)&=\int_{\Omega}|u_0(x)|^2\,dx+\int_{\Omega}|\phi_{u_0}^{\Omega}(x)|^2\,dx\\
&=4\ep^2\xi_1^2+(1-4\ep^2)\xi_2^2+4\ep^2(1-2\ep)(\xi_1-\xi_2)^2,
\end{align}
and
\ba{eq:terzo-pezzo}
\F(u_0,\Omega_2)&=\int_{\Omega_2}|u_0(x)|^2\,dx+\int_{\Omega_2}|\phi_{u_0}^{\Omega}(x)|^2\,dx\\
&=4\ep^2\xi_1^2+4\delta(\delta+2\ep)\xi_2^2+\frac{4\ep^2\delta(\xi_1-\xi_2)^2}{\ep+\delta}.
\end{align}
Now \eqref{eq:ineq-false-ex-ex} becomes
$$4\ep^2(1-2\ep)(\xi_1-\xi_2)^2\leq 4\delta(\delta+2\ep)\xi_2^2+\frac{4\ep^2\delta(\xi_1-\xi_2)^2}{\ep+\delta}$$
for every $\ep<\tfrac12$ and $\delta<\tfrac12-\ep$. Letting $\delta\to 0$ we get
 $$4\ep^2(1-2\ep)(\xi_1-\xi_2)^2\leq 0.$$
 Since $\xi_1\neq\xi_2$, this contradicts the subadditivity of $\F(u_0,\cdot)$ and yields the nonlocality of $\F$.
  \end{example}
   
\section*{acknowledgements}
The authors thank the Center for Nonlinear Analysis (NSF Grant No. DMS-0635983), where this research was carried out. The research of 
E. Davoli, and I. Fonseca was funded by the National Science Foundation under Grant No. DMS-
0905778. I. Fonseca was also supported by the National Science Foundation under Grant No. DMS-1411646. E. Davoli and I. Fonseca
 acknowledge support of the National Science Foundation under the PIRE Grant No.
OISE-0967140. 

\bibliographystyle{plain}
\bibliography{ed}
\end{document}